\documentclass[11pt, reqno]{amsart}
\usepackage{amsmath, amsthm, amscd, amsfonts, amssymb, graphicx, color, mathtools, mathrsfs}
\usepackage[bookmarksnumbered, colorlinks, plainpages]{hyperref}
\usepackage{enumerate}
\usepackage{setspace}
\usepackage{multicol}
\usepackage[margin=2cm]{geometry}
\usepackage{comment}
\usepackage{booktabs}
\usepackage{caption}
\usepackage{subcaption}
\usepackage{pgfplots}

\theoremstyle{definition}
\newtheorem{theorem}{Theorem}[section]
\newtheorem{lemma}[theorem]{Lemma}
\newtheorem{proposition}[theorem]{Proposition}

\newtheorem{definition}[theorem]{Definition}

\newtheorem{remark}[theorem]{Remark}
\numberwithin{equation}{section}
\newtheorem{algorithm}[theorem]{Algorithm}

\newcommand{\cal}{\mathcal}
\newcommand{\bff}{\boldsymbol}
\newcommand{\bb}{\mathbb}

\newcommand{\dt}{\mathrm{d}t}
\newcommand{\ddt}{\frac{\mathrm{d}}{\mathrm{d}t}}
\newcommand{\dx}{\mathrm{d}x}
\newcommand{\ds}{\mathrm{d}s}

\newcommand{\norm}[2]{\left\|{#1}\right\|_{#2}}
\newcommand{\inpro}[2]{\left\langle#1,#2\right\rangle}
\newcommand{\abs}[1]{\left|{#1}\right|}

\allowdisplaybreaks

\begin{document}
	\setcounter{page}{1}
	
	\title[Numerical analysis of the LLB equation with spin-torques]
	{Numerical analysis of the Landau--Lifshitz--Bloch equation with spin-torques}
	
	\author[Agus L. Soenjaya]{Agus L. Soenjaya}
	\address{School of Mathematics and Statistics, The University of New South Wales, Sydney 2052, Australia}
	\email{\textcolor[rgb]{0.00,0.00,0.84}{a.soenjaya@unsw.edu.au}}
	
	\date{\today}
	
	
	\begin{abstract}
		The current-induced magnetisation dynamics in a ferromagnet at elevated temperatures can be described by the Landau--Lifshitz--Bloch (LLB) equation with spin-torque terms. In this paper, we focus on the regime above the Curie temperature. We first establish the existence and uniqueness of a global strong solution to the model in spatial dimensions $d=1,2,3$, under an additional smallness assumption on the initial data if $d=3$. Relevant smoothing and decay estimates are also derived. We then propose a fully discrete, linearly implicit finite element scheme for the problem and prove that it achieves optimal-order convergence, assuming adequate regularity of the exact solution. In addition, we introduce an unconditionally energy-stable finite element method for the case of negligible non-adiabatic torque. This scheme is also shown to converge optimally and, in the absence of current, preserves energy dissipation at the discrete level. Finally, we present numerical simulations that support the theoretical analysis and demonstrate the performance of the proposed methods.
	\end{abstract}
	\maketitle
	
\section{Introduction}

The theory of micromagnetism is a widely used framework to analyse the behaviour of a ferromagnetic material on the micron scale. In contrast with the atomistic approach, continuum models in micromagnetics can be solved numerically for a much larger system. A standard model to describe the time evolution of magnetic configuration of a ferromagnetic body subject to some non-steady conditions is given by the Landau--Lifshitz equation~\cite{Cim08, GuoDing08, LL35}. However, numerous theoretical and experimental studies show that this model is not suitable at elevated temperatures since it cuts off all contributions from high frequency spin waves responsible for longitudinal magnetisation fluctuations~\cite{ChuNie20, GriKoc03}. This problem is especially crucial in applications since many modern magnetic devices such as the HAMR (heat-assisted magnetic recording) works at very high temperatures, often exceeding the so-called Curie temperature~\cite{MeoPan20, VogAbeBru16, ZhuLi13}.

A variant of the Landau--Lifshitz model widely used in the literature to describe the magnetisation dynamics at high temperatures is the Landau--Lifshitz--Bloch (LLB) equation. This model takes into account longitudinal dynamics of the magnetisation and is valid below and above the critical temperature (the Curie temperature $T_c$)~\cite{AtxHinNow16, ChuNowChaGar06, Gar97}. Let $\mathscr{D}\subset \bb{R}^d$ ($d=1,2,3$) be a bounded domain. The general LLB equation describing the dynamics of magnetisation $\bff{u}:[0,T]\times \mathscr{D}\to \bb{R}^3$ is
\begin{align}\label{equ:general LLB}
	\partial_t \bff{u}= -\gamma \bff{u}\times \bff{H} + \frac{\alpha_1}{|\bff{u}|^2} (\bff{u}\cdot \bff{H})\bff{u} - \frac{\alpha_2}{|\bff{u}|^2} \bff{u}\times (\bff{u}\times \bff{H}),
\end{align}
where $\gamma>0$ is the gyromagnetic ratio, $\alpha_1>0$ is the longitudinal damping coefficient, and $\alpha_2>0$ is the transverse damping coefficient. This equation is equipped with an initial data and a homogeneous Neumann boundary condition. In~\eqref{equ:general LLB}, the effective field $\bff{H}:[0,T]\times \mathscr{D}\to \bb{R}^3$ is the negative variational derivative of the micromagnetic energy functional $\mathcal{E}$ to be specified later, i.e. $\bff{H}=-\nabla_{\bff{u}} \mathcal{E}(\bff{u})$. It is known that above the Curie temperature~\cite{AtxHinNow16, ChuNowChaGar06}, we have $\alpha_1=\alpha_2:=\alpha$. Thus by means of the vector identity
\begin{align*}
	\bff{u}\times (\bff{u}\times \bff{H})= (\bff{u}\cdot \bff{H})\bff{u}- |\bff{u}|^2 \bff{H},
\end{align*}
equation~\eqref{equ:general LLB} simplifies to
\begin{align}\label{equ:above Tc LLB}
	\partial_t \bff{u}= -\gamma\bff{u}\times \bff{H} + \alpha \bff{H}.
\end{align}

Here, we consider a standard micromagnetic energy functional
\begin{equation}\label{equ:energy}
	\mathcal{E}(\bff{u})
	:=
	\underbrace{\frac{\sigma}{2} \norm{\nabla \bff{u}}{\bb{L}^2}^2}_{\text{exchange}}
	+
	\underbrace{\frac{\kappa\mu}{2} \norm{\bff{u}}{\bb{L}^2}^2
		+\frac{\kappa}{4} \norm{\bff{u}}{\bb{L}^4}^4}_{\text{internal exchange}}
	+
	\underbrace{\frac{\lambda}{2} \int_\mathscr{D} (\bff{e}\cdot \bff{u})^2\, \dx}_{\text{anisotropy}},
\end{equation}
which consists of the exchange energy, the internal exchange energy of Ginzburg--Landau type above the Curie temperature (so $\mu>0$), and the anisotropy energy~\cite{ChuNowChaGar06, Le16}. This gives rise to an effective field
\begin{equation}\label{equ:eff field}
	\bff{H}
	= \underbrace{\sigma \Delta \bff{u}}_{\text{exchange}} - \underbrace{\kappa \mu \bff{u}- \kappa |\bff{u}|^2 \bff{u}}_{\text{internal exchange}}
	- \underbrace{\lambda \bff{e}(\bff{e}\cdot \bff{u})}_{\text{anisotropy}},
\end{equation}
where $\sigma>0$ is the exchange damping constant, $\kappa=(2\chi)^{-1}$ is a positive constant related to the longitudinal susceptibility of the material, $\mu>0$ is a constant related to the equilibrium magnetisation, $\lambda$ is the uniaxial anisotropy constant, and $\bff{e}$ is a given unit vector denoting the anisotropy axis of the material. The constant $\lambda$ in \eqref{equ:eff field} determines the ease of magnetising in certain directions, which is positive in the presence of an `easy' plane, or negative in the presence of an `easy' axis. We assume $\lambda>0$ in \eqref{equ:eff field} for simplicity.
The analysis done in this paper will continue to hold if we add a time-varying applied magnetic field $\bff{B}(t)$ to $\bff{H}$, but for ease of presentation we take $\bff{B}(t)\equiv \bff{0}$. Note that with this choice of $\bff{H}$, the simplified problem~\eqref{equ:above Tc LLB} which holds above the Curie temperature is still a quasilinear vector-valued PDE.

In this paper, we focus on the LLB equation above $T_c$. More generally, we also include spin-torque terms due to currents~\cite{Ans04, StiSasDon07, ThiNakMilSuz05, TseBraBau08} in our analysis, which are important for applications in spintronics~\cite{AyoKotMouZak21, YasFasIvaMak22}. Numerical analysis of the general LLB problem~\eqref{equ:general LLB} will be addressed in a forthcoming paper. 
All in all, for a given current density $\bff{\nu}:[0,T]\times \mathscr{D} \to \bb{R}^d$, the LLB equation with spin-torques above the Curie temperature reads:
\begin{subequations}\label{equ:llb a}
	\begin{alignat}{2}
		\label{equ:llb eq1}
		&\partial_t \bff{u}
		=
		-
		\gamma \bff{u} \times \bff{H}
		+
		\alpha \bff{H}
		+
		\beta_1 (\bff{\nu}\cdot\nabla) \bff{u}
		+
		\beta_2 \bff{u}\times (\bff{\nu}\cdot\nabla) \bff{u}
		\,
		\qquad && \text{for $(t,\bff{x})\in(0,T)\times\mathscr{D}$,}
		\\
		\label{equ:llb eq2}
		&\bff{H}=\sigma \Delta \bff{u} - \kappa\mu \bff{u} - \kappa |\bff{u}|^2 \bff{u} - \lambda \bff{e} (\bff{e}\cdot \bff{u}),
		\qquad && \text{for $(t,\bff{x})\in(0,T)\times\mathscr{D}$,}
		\\
		\label{equ:llb init}
		&\bff{u}(0,\bff{x})= \bff{u_0}(\bff{x}) 
		\qquad && \text{for } \bff{x}\in \mathscr{D},
		\\
		\label{equ:llb bound}
		&\displaystyle{
			\frac{\partial \bff{u}}{\partial \bff{n}}= \bff{0}}
		\qquad && \text{for } (t,\bff{x})\in (0,T) \times \partial \mathscr{D},
	\end{alignat}
\end{subequations}
where $\beta_1$ and $\beta_2$ are, respectively, the adiabatic and the non-adiabatic torques constants. Physical experiments~\cite{BurMih_etal10, EltWot_etal10} and theoretical studies~\cite{StiSasDon07, TseBraBau08} suggest that $|\beta_2| \ll |\beta_1|$, thus $\beta_2$ is sometimes neglected in certain cases. 
Note that a sufficiently regular solution of \eqref{equ:llb a} satisfies an energy identity~\cite{ChuNie20}. Indeed, formally taking the inner product of \eqref{equ:llb eq1} and \eqref{equ:llb eq2} with $\bff{H}$ and $-\partial_t \bff{u}$, respectively, gives
\begin{align}
	\label{equ:dtu H}
	\inpro{\partial_t \bff{u}}{\bff{H}} &= \alpha \norm{\bff{H}}{\bb{L}^2}^2 + \beta_1 \inpro{(\bff{\nu}\cdot\nabla) \bff{u}}{\bff{H}} + \beta_2 \inpro{\bff{u}\times (\bff{\nu}\cdot\nabla)\bff{u}}{\bff{H}},
	\\
	\label{equ:min H dtu}
	-\inpro{\bff{H}}{\partial_t \bff{u}} &= \ddt \left( \frac{\sigma}{2}  \norm{\nabla \bff{u}}{\bb{L}^2}^2 + \frac{\kappa\mu}{2} \norm{\bff{u}}{\bb{L}^2}^2 + \frac{\kappa}{4} \norm{\bff{u}}{\bb{L}^4}^4 + \frac{\lambda}{2} \norm{\bff{e}\cdot \bff{u}}{\bb{L}^2}^2 \right).
\end{align}
Adding both equations above and integrating over $t\in (t_1,t_2)$, we obtain
\begin{align}\label{equ:energy identity}
	&\mathcal{E}\big(\bff{u}(t_2)\big)+ \alpha \int_{t_1}^{t_2} \norm{\bff{H}(s)}{\bb{L}^2}^2
	\nonumber\\
	&=
	\mathcal{E}\big(\bff{u}(t_1)\big)+ \beta_1 \int_{t_1}^{t_2} \inpro{(\bff{\nu}(s)\cdot \nabla) \bff{u}(s)}{\bff{H}(s)} \ds + \beta_2 \int_{t_1}^{t_2} \inpro{\bff{u}(s)\times (\bff{\nu}(s)\cdot\nabla) \bff{u}(s)}{\bff{H}(s)} \ds,
\end{align}
where $\mathcal{E}$ is the micromagnetic energy functional given in~\eqref{equ:energy}. If no current is present ($\bff{\nu}=\bff{0}$ in~\eqref{equ:energy identity}), then the system dissipates energy, i.e. $\mathcal{E}\big(\bff{u}(t_2)\big)\leq \mathcal{E}\big(\bff{u}(t_1)\big)$. We remark that the spin-torque terms cannot be written as a variational derivative of some energy functional~\cite{StiSasDon07}, thus adding complications to the analysis of the problem with these terms.

Existing results in the literature which are relevant to the present paper will be reviewed next. We remark that many variants of the standard Landau--Lifshitz equation attract significant interest in the mathematical literature, a non-exhaustive list includes~\cite{AbeHrkPagPra14, DavDiPraRug22, FeiTra17a, MelPta13, Rug22, SoeTra23}, which take into account various effects not included in the standard model. Specifically for the LLB equation, a rigorous analysis of the model (without spin-torques) is initiated in~\cite{Le16}, where the existence of a weak solution is obtained. The existence, uniqueness, and decay of the strong solution are discussed in~\cite{LeSoeTra24}. These papers, however, do not take into account the spin-torque terms in the analysis. The existence of a weak solution (but without uniqueness in the general case) for the LLB equation with spin-torques is shown in~\cite{AyoKotMouZak21}, but the existence of a strong solution is not known yet. Spin-torque effects have also been considered for the Landau--Lifshitz--Gilbert equation~\cite{VinTra24}, where a local-in-time strong solution is obtained in that case.

On the aspect of numerical analysis, several fully discrete finite element methods to approximate the solution of the LLB equation without spin-torques are proposed in~\cite{LeSoeTra24} and~\cite{Soe24}. However, numerical schemes developed in these papers actually approximate the \emph{regularised} version of the equation, \emph{not} the LLB equation itself. The strong solution of the regularised problem is shown to converge to that of~\eqref{equ:llb a} as the regularisation parameter goes to zero. As such, the order of convergence of the numerical scheme is also dependent on this parameter, which theoretically blow up as the regularisation parameter vanishes. Similar approach using regularisation is also taken in the stochastic case~\cite{GolJiaLe24}, where a suboptimal order of convergence in probability is obtained for $d=1$ and $2$. In~\cite{BenEssAyo24}, a finite element scheme is proposed to approximate the LLB equation without spin-torques directly, but convergence is only shown conditionally along a subsequence in $\ell^\infty(0,T;\bb{H}^1)$ without rate, while energy dissipativity at the discrete level is not guaranteed by the scheme. Moreover, this scheme is only conditionally stable and is nonlinear, in that a system of nonlinear equations needs to be solved at each time step. In any case, numerical discretisation of the LLB equation with spin-torques~\eqref{equ:llb a} has not been explored yet in the literature.

This paper aims to address the aforementioned gaps. Our primary objectives are to establish the global well-posedness of~\eqref{equ:llb a} and to develop a numerical scheme that converges unconditionally in $\ell^\infty(0,T;\bb{H}^1)$ at a provably optimal rate, without requiring regularisation. To this end, we first show the unique existence of global strong solution to~\eqref{equ:llb a} for $d=1,2,3$. We then propose a fully discrete, linearly implicit scheme which applies directly to~\eqref{equ:llb a} without the need for regularisation. This scheme is stable and convergent in $\ell^\infty(0,T;\bb{L}^2) \cap \ell^2(0,T;\bb{H}^1)$. However, due to difficulties arising from nonlinear cross-product terms, we are unable to prove convergence -- or even stability -- in $\ell^\infty(0,T;\bb{H}^1)$ for this scheme. To achieve provably unconditional convergence in $\ell^\infty(0,T;\bb{H}^1)$, we are led to consider a second fully discrete scheme to solve~\eqref{equ:llb a}. For the case $\beta_2=0$, corresponding to negligible non-adiabatic torque (which is a valid approximation for transition-metal ferromagnets dominated by exchange fields~\cite{Ans04, TseBraBau08}), we are able to prove unconditional stability and convergence in $\ell^\infty(0,T;\bb{H}^1)$ for this scheme at an optimal rate. Additionally, this second scheme preserves energy dissipation at the discrete level in the absence of spin current ($\bff{\nu}=\bff{0}$), a desirable feature consistent with the underlying physical model.

As far as we know, this is the first time a linear scheme is proposed to directly approximate the LLB equation (with or without spin-torques) above the Curie temperature, for which unconditional convergence with an optimal rate is shown rigorously. An attempt to maintain energy dissipativity of the LLB equation without spin-torques at the discrete level, while still attaining convergence at the optimal rate, is also addressed for the first time by our second numerical scheme. Some numerical experiments to corroborate the analysis are also provided.
To summarise, our main contributions in this paper are:
\begin{enumerate}
	\item Proving the existence and uniqueness of global strong solution to the LLB equation with spin-torques, unconditionally when $d=1$ or $2$, and for small initial data when $d=3$, which build on the results in~\cite{AyoKotMouZak21, Le16}. Physically relevant smoothing and decay estimates are also provided (cf. Theorem~\ref{the:main existence} and \ref{the:decay}).
	\item Proposing a fully discrete linearly implicit finite element scheme to approximate~\eqref{equ:llb a} with an unconditionally optimal order of convergence, namely with a spatial rate of $r+1$ and a temporal rate of $1$ in the $\ell^\infty(0,T;\bb{L}^2) \cap \ell^2(0,T;\bb{H}^1)$ norm for a finite element space of degree $r$ (cf. Theorem~\ref{the:spin torq error}). This removes the need for regularisation in~\cite{LeSoeTra24}, while also taking spin-torque terms into account.
	\item Proposing a fully discrete (but nonlinear) finite element scheme to approximate the LLB equation for the case $\beta_2=0$, which maintains energy dissipativity at the discrete level in the absence of spin current. An optimal order of convergence in $\ell^\infty(0,T;\bb{H}^1)$ is also shown unconditionally (cf. Theorem~\ref{the:without spin error}), improving the results of~\cite{BenEssAyo24} in several aspects.
\end{enumerate}

The existence of a global strong solution is shown in Section~\ref{sec:exist} by deriving further a priori estimates in stronger norms. These estimates, in turn, depend on a uniform bound in the $\bb{L}^\infty$-norm, which -- unlike in the case of the standard Landau--Lifshitz equation -- is nontrivial to obtain due to the lack of magnetisation length conservation and the added complexity introduced by the spin-torque terms. We then proceed to propose some finite element schemes. In Section~\ref{sec:llb spin}, by using an elliptic projection adapted to the problem at hand, we show an optimal order of convergence for the linear finite element scheme to approximate~\eqref{equ:llb a}. In Section~\ref{sec:scheme 2}, we devise a nonlinear scheme (for the equation~\eqref{equ:llb a} with $\beta_2=0$) that converges unconditionally in $\ell^\infty(0,T;\bb{H}^1)$ and preserves energy dissipativity in the absence of spin current. A detailed analysis is carried out to establish the claimed optimal order of convergence. This analysis relies on a stability estimate in $\ell^2(0,T;\bb{L}^\infty)$, which, in our setting, is again not immediate. Several numerical simulations are presented in Section~\ref{sec:num exp} to support the theoretical results and to compare the two proposed schemes. Finally, Section~\ref{sec:conclusion} offers concluding remarks and outlines potential directions for future research. The discussion of a fixed-point iteration method for solving the nonlinear scheme at each time step is deferred to Appendix~\ref{sec:linearisation}.

\section{Preliminaries}

\subsection{Notations}
We begin by defining some notations used in this paper. Let $\mathscr{D}$ be a bounded open domain. The function space $\bb{L}^p := \bb{L}^p(\mathscr{D}; \bb{R}^3)$ denotes the usual space of $p$-th integrable functions taking values in $\bb{R}^3$ and $\bb{W}^{k,p} := \bb{W}^{k,p}(\mathscr{D}; \bb{R}^3)$ denotes the usual Sobolev space of 
functions on $\mathscr{D} \subset \bb{R}^d$ taking values in $\bb{R}^3$. The function space $\mathcal{C}^0:= \mathcal{C}^0(\mathscr{D};\bb{R}^3)$ is the space of continuous functions on $\mathscr{D}$ taking values in $\bb{R}^3$. We
write $\bb{H}^k := \bb{W}^{k,2}$. The Laplacian operator acting on $\bb{R}^3$-valued functions is denoted by $\Delta$ with domain $\bb{H}^2_{\bff{n}}:= \bb{H}^2_{\bff{n}}(\mathscr{D})$ given by
\begin{equation}\label{equ:H2n}
	\bb{H}^2_{\bff{n}}:= \left\{\bff{v}\in \bb{H}^2 : \frac{\partial \bff{v}}{\partial \bff{n}} = \bff{0} \text{ on } \partial\mathscr{D} \right\}.
\end{equation}

If $X$ is a Banach space, the spaces $L^p(0,T; X)$ and $W^{k,p}(0,T;X)$ denote respectively the usual Lebesgue and Sobolev spaces of functions on $(0,T)$ taking values in $X$. The space $C([0,T];X)$ denotes the space of continuous functions on $[0,T]$ taking values in $X$. For simplicity, we will write $L^p(\bb{W}^{m,r}) := L^p(0,T; \bb{W}^{m,r})$ and $L^p(\bb{L}^q) := L^p(0,T; \bb{L}^q)$. The norm in a Banach space $X$ is denoted by $\|\cdot\|_X$. We denote the scalar product in a Hilbert space $H$ by $\inpro{\cdot}{\cdot}_H$ and its corresponding norm by $\|\cdot\|_H$. We will not distinguish between the scalar product of $\bb{L}^2$ vector-valued functions taking values in $\bb{R}^3$ and the scalar product of $\bb{L}^2$ matrix-valued functions taking values in $\bb{R}^{3\times 3}$, and denote them by $\langle\cdot,\cdot\rangle$.

Finally, the constant $C$ in the estimates denotes a generic constant which may take different values at different occurrences. If
the dependence of $C$ on some variable, e.g.~$T$, is highlighted, we will write
$C(T)$.

\subsection{Formulations and assumptions}

Here, we formulate the problem more precisely and state several assumptions which will be used throughout the paper. For simplicity of presentation and to study the influence of the damping constant and spin-torques in the analysis, we only keep the coefficients $\alpha,\beta_1,\beta_2$ in the equation, and set the rest of the coefficients, namely~$\gamma, \sigma, \kappa, \mu, \lambda$ to $1$ in~\eqref{equ:llb a}, unless otherwise specified. We assume that $\bff{\nu}$ is a given vector field such that $\bff{\nu}\cdot\bff{n}=0$ on $\partial\mathscr{D}$ and $\norm{\bff{\nu}}{L^\infty(0,T;\bb{H}^4(\mathscr{D};\,\bb{R}^d))} \leq \nu_\infty$, where $\bff{n}$ is the outward pointing normal vector and $\nu_\infty$ is a positive constant. Without loss of generality, we assume the damping parameter $\alpha<1$ in the estimates.

\begin{definition}\label{def:strong sol}
Given $T>0$ and $\bff{u}_0\in \bb{H}^2(\mathscr{D})$, a \emph{strong solution} to \eqref{equ:llb a} is a function
\[
\bff{u}\in H^1(0,T;\bb{L}^2) \cap L^\infty(0,T;\bb{H}^2)
\]
such that $\bff{u}(0)=\bff{u}_0$, and that for all $\bff{\chi}\in \bb{L}^2$ and $t\in [0,T]$, 
\begin{align}\label{equ:weakform}
	\inpro{\partial_t \bff{u}(t)}{\bff{\chi}}
	&=
	-\inpro{\bff{u}(t)\times \big(\Delta \bff{u}(t)- \bff{w}(t) \big)}{\bff{\chi}}
	+
	\alpha \inpro{\big(\Delta \bff{u}(t)- \bff{w}(t)\big)}{\bff{\chi}}
	\nonumber \\
	&\quad
	+
	\beta_1 \inpro{(\bff{\nu}(t) \cdot\nabla) \bff{u}(t)}{\bff{\chi}}
	+
	\beta_2 \inpro{\bff{u}(t)\times (\bff{\nu}(t)\cdot \nabla) \bff{u}(t)}{\bff{\chi}},
\end{align}
where for almost every $t\in [0,T]$, 
\begin{align}\label{equ:w}
	\bff{w}(t)
	= 
	\bff{u}(t) 
	+
	 |\bff{u}(t)|^2 \bff{u}(t) 
	+ 
	\bff{e} \big(\bff{e}\cdot \bff{u}(t) \big).
\end{align}
In this case, \eqref{equ:llb a} is satisfied for almost every $(t,x)\in (0,T)\times \mathscr{D}$.
\end{definition}
At times, we write $\bff{H}:=\Delta \bff{u} - \bff{w}$, where $\bff{w}$ in~\eqref{equ:w} represents the remaining terms in the effective field.
Note that by the vector identity~\eqref{equ:div ab}, we can write
\begin{align}
	\label{equ:otimes 1}
	\beta_1 (\bff{\nu}\cdot \nabla) \bff{u}
	&=
	\beta_1 \nabla \cdot (\bff{u} \otimes \bff{\nu}) 
	- \beta_1 (\nabla \cdot \bff{\nu}) \bff{u},
	\\
	\label{equ:otimes 2}
	\beta_2 \bff{u}\times (\bff{\nu}\cdot \nabla)\bff{u} 
	&=
	\beta_2 \bff{u} \times \nabla \cdot (\bff{u} \otimes \bff{\nu})
	- \beta_2 \bff{u} \times (\nabla \cdot \bff{\nu}) \bff{u},
\end{align}
where the last term on the right-hand side of~\eqref{equ:otimes 2} is a zero vector.
Thus, if we assume that $\bff{\nu}$ is tangential to the boundary ($\bff{\nu}\cdot \bff{n}=0$ on $\partial \mathscr{D}$, which is physically reasonable), then by~\eqref{equ:otimes 1}, \eqref{equ:otimes 2}, and the divergence theorem,
\begin{align}
	\label{equ:div thm beta1}
	\beta_1 \inpro{(\bff{\nu}(t) \cdot\nabla) \bff{u}(t)}{\bff{\chi}}
	&=
	-
	\beta_1 \inpro{\bff{u}(t) \otimes \bff{\nu}(t)}{\nabla \bff{\chi}}
	-
	\beta_1 \inpro{(\nabla \cdot \bff{\nu}(t)) \bff{u}(t)}{\bff{\chi}}
	\\
	\label{equ:div thm beta2}
	\beta_2 \inpro{\bff{u}(t)\times (\bff{\nu}(t)\cdot \nabla) \bff{u}(t)}{\bff{\chi}}
	&=
	\beta_2 \inpro{\bff{u}(t) \otimes \bff{\nu}(t)}{\nabla(\bff{u}(t) \times \bff{\chi})}.
\end{align}
Therefore, integrating by parts and applying the above identities, we can also write~\eqref{equ:weakform} as
\begin{align}\label{equ:weak special}
	\inpro{\partial_t \bff{u}(t)}{\bff{\chi}}
	&=
	\inpro{\bff{u}(t) \times \nabla \bff{u}(t)}{\nabla \bff{\chi}}
	+
	\inpro{\bff{u}(t)\times \bff{e}(\bff{e}\cdot \bff{u}(t))}{\bff{\chi}}
	-
	\alpha \inpro{\nabla \bff{u}(t)}{\nabla \bff{\chi}}
	-
	\alpha \inpro{\bff{w}(t)}{\bff{\chi}} 
	\nonumber \\
	&\quad
	-
	\beta_1 \inpro{\bff{u}\otimes \bff{\nu}}{\nabla \bff{\chi}}
	-
	\beta_1 \inpro{(\nabla \cdot \bff{\nu}) \bff{u}}{\bff{\chi}}
	+
	\beta_2 \inpro{\bff{u}\otimes \bff{\nu}}{\nabla(\bff{u}\times \bff{\chi})},
\end{align}
for all $\bff{\chi}\in \bb{H}^1$, which is the weak formulation that we will use in Section~\ref{sec:llb spin}.

\subsection{Auxiliary results}

Our analysis requires several auxiliary results. We begin by recalling the following vector identities: For all $\bff{a},\bff{b}\in \bb{R}^3$,
\begin{align}
	\label{equ:div ab}
	\nabla \cdot (\bff{a}\otimes \bff{b})
	&=
	(\nabla \cdot \bff{b}) \bff{a}
	+
	(\bff{b}\cdot \nabla) \bff{a},
	\\
	\label{equ:nab uuv}
	\nabla \left(|\bff{a}|^2 \bff{b}\right) 
	&=
	2\bff{b}(\bff{a}\cdot\nabla\bff{a}) + |\bff{a}|^2 \nabla \bff{b},
	\\
	\label{equ:aab}
	2\bff{a} \cdot (\bff{a}-\bff{b}) 
	&= |\bff{a}|^2 - |\bff{b}|^2 + |\bff{a}-\bff{b}|^2,
	\\
	\label{equ:a2aab}
	4|\bff{a}|^2 \bff{a}\cdot (\bff{a}-\bff{b})
	&=
	|\bff{a}|^4 - |\bff{b}|^4 + \left(|\bff{a}|^2 - |\bff{b}|^2\right)^2 + 2|\bff{a}|^2 |\bff{a}-\bff{b}|^2.
\end{align}

Next, we state the following elliptic regularity result.

\begin{lemma}[Elliptic regularity]\label{lem:elliptic}
	Let $\mathscr{D}$ be a bounded domain with $C^2$-smooth boundary or a convex polyhedral domain. There exists a constant $C:=C(\mathscr{D})$ such that for all $\bff{v}\in \bb{H}^2_{\bff{n}}$, where $\bb{H}^2_{\bff{n}}$ is defined in~\eqref{equ:H2n}, we have
	\begin{align}
		\label{equ:D2L2}
		\norm{D^2 \bff{v}}{\bb{L}^2} &\leq C \norm{\Delta \bff{v}}{\bb{L}^2},
		\\
		\label{equ:v H2 elliptic}
		\norm{\bff{v}}{\bb{H}^2} &\leq C\left( \norm{\bff{v}}{\bb{L}^2} + \norm{\Delta \bff{v}}{\bb{L}^2} \right),
	\end{align}
	where $D^2 \bff{v}$ is the Hessian of $\bff{v}$.
	
	More generally, let $\mathscr{D}$ be a bounded domain with $C^{k+2}$-smooth boundary, where $k$ is a non-negative integer. There exists a constant $C:=C(k, \mathscr{D})$ such that for all $\bff{v}\in \bb{H}^{k+2} \cap \bb{H}^2_{\bff{n}}$,
	\begin{align}\label{equ:v Hk elliptic}
		\norm{\bff{v}}{\bb{H}^{k+2}} \leq C \left(\norm{\bff{v}}{\bb{L}^2}+ \norm{\Delta \bff{v}}{\bb{H}^k} \right).
	\end{align}
\end{lemma}

\begin{proof}
	See~\cite[Chapter~2--3]{Gri11}.
\end{proof}

Some frequently used interpolation inequalities are collected in the following lemma.

\begin{lemma}
Let $\epsilon>0$ be given and $d\in \{1,2,3\}$. Then there exists a positive constant $C$ such that the following inequalities hold.
\begin{enumerate}
	\renewcommand{\labelenumi}{\theenumi}
	\renewcommand{\theenumi}{{\rm (\roman{enumi})}}	
	\item There exists a constant $C:=C(\mathscr{D})$ such that for any $\bff{v}\in \bb{H}^1$,
	\begin{align}\label{equ:gal nir uh L3}
		\norm{\bff{v}}{\bb{L}^3}
		&\leq
		C \norm{\bff{v}}{\bb{L}^2}^{1-\frac{d}{6}} \norm{\bff{v}}{\bb{H}^1}^{\frac{d}{6}}
		\leq 
		C \norm{\bff{v}}{\bb{L}^2}^2 + \epsilon \norm{\nabla \bff{v}}{\bb{L}^2}^2,
		\\
		\label{equ:gal nir uh L4}
		\norm{\bff{v}}{\bb{L}^4}
		&\leq
		C \norm{\bff{v}}{\bb{L}^2}^{1-\frac{d}{4}} \norm{\bff{v}}{\bb{H}^1}^{\frac{d}{4}}
		\leq
		C \norm{\bff{v}}{\bb{L}^2}^2
		+
		\epsilon \norm{\nabla \bff{v}}{\bb{L}^2}^2.
	\end{align}
	\item Let
	\begin{equation}\label{equ:p}
		p \in 
		\begin{cases}
			[1,\infty], &\text{ if $d=1$}, \\
			[1,\infty), &\text{ if $d=2$}, \\
			[1,6], &\text{ if $d=3$}.
		\end{cases}
	\end{equation}
	There exists a constant $C:=C(\mathscr{D},p)$ such that for any $\bff{v}\in \bb{H}^2_{\bff{n}}$, we have
	\begin{align}\label{equ:nab v less Delta v}
		\norm{\nabla \bff{v}}{\bb{L}^p}
		\leq
		C\norm{\Delta \bff{v}}{\bb{L}^2}.
	\end{align}
	If, furthermore, $\Delta \bff{v}\in \bb{H}^1$, then we also have
	\begin{align}\label{equ:Delta v less nabdelta v}
		\norm{\Delta \bff{v}}{\bb{L}^p}
		&\leq
		C\norm{\nabla\Delta \bff{v}}{\bb{L}^2}
	\end{align}
	\item For any $\bff{v}\in \bb{H}^3 \cap \bb{H}^2_{\bff{n}}$, we have
	\begin{align}\label{equ:v H3 ellip reg}
		\norm{\bff{v}}{\bb{H}^3}
		&\leq
		C\left(\norm{\bff{v}}{\bb{L}^2} + \norm{\nabla\Delta \bff{v}}{\bb{L}^2} \right).
	\end{align}
	\item Let $\mathscr{D}$ be a convex polyhedral domain. There exists a constant $C:=C(\mathscr{D},q)$ such that for any $\bff{v}\in \bb{H}^2_{\bff{n}}$ satisfying $\Delta \bff{v}\in \bb{L}^q$ for some $q>d$, we have
	\begin{align}\label{equ:nab v L infty}
		\norm{\nabla \bff{v}}{\bb{L}^\infty} \leq C\norm{\Delta \bff{v}}{\bb{L}^q}.
	\end{align}
\end{enumerate}
\end{lemma}

\begin{proof}
Estimates \eqref{equ:gal nir uh L3} and \eqref{equ:gal nir uh L4} follow from the Gagliardo--Nirenberg and the Young inequalities. 

Next, we show~\eqref{equ:nab v less Delta v}. Let $\{\bff{\phi}_k\}_{k\geq 0}$ be a complete countable set in $\bb{L}^2(\mathscr{D})$ of orthonormal eigenfunctions of the negative Neumann Laplacian $-\Delta$, with corresponding eigenvalues $\{\lambda_k\}_{k\geq 0}$. The eigenvalues can be arranged so that $0=\lambda_0<\lambda_1\leq \lambda_2\leq \lambda_3\leq \ldots$, with $\bff{\phi}_0= |\mathscr{D}|^{-1}$. Therefore, for any $\bff{v}\in \bb{H}^2_{\bff{n}}$, we can write $\Delta \bff{v}= \sum_{k=1}^\infty a_k \bff{\phi}_k$, where $a_k=\inpro{\Delta \bff{v}}{\bff{\phi}_k}$ for $k\in \bb{N}$ and $a_0=0$. This means~$\bff{v}= b_0\bff{\phi}_0- \sum_{k=1}^\infty a_k \lambda_k^{-1} \bff{\phi}_k$, where $b_0:=\inpro{\bff{w}}{\bff{\phi}_0}$. We then have
\begin{align*}
	\norm{\nabla \bff{v}}{\bb{L}^2}^2 = -\inpro{\Delta \bff{v}}{\bff{v}} = \sum_{k=1}^\infty a_k^2 \lambda_k^{-1} \leq \lambda_1^{-1} \sum_{k=1}^\infty a_k^2 = \lambda_1^{-1} \norm{\Delta \bff{v}}{\bb{L}^2}^2,
\end{align*}
thus showing~\eqref{equ:nab v less Delta v} for $p=2$. For $p\in [1,2)$, we have $\bb{L}^2\hookrightarrow \bb{L}^p$ since $\mathscr{D}$ is bounded. For $p\in (2,p^\ast]$, where $p^\ast$ is any number (or $\infty$) such that the embedding $\bb{H}^1\hookrightarrow \bb{L}^{p^\ast}$ holds, we have by this embedding, \begin{align*}
	\norm{\nabla \bff{v}}{\bb{L}^p}
	\leq 
	\norm{\nabla \bff{v}}{\bb{H}^1}
	\leq
	\norm{\nabla \bff{v}}{\bb{L}^2} + \norm{D^2 \bff{v}}{\bb{L}^2}
	\leq
	C\norm{\Delta \bff{v}}{\bb{L}^2},
\end{align*}
where we used \eqref{equ:D2L2} and \eqref{equ:nab v less Delta v} for $p=2$ proven earlier. This proves \eqref{equ:nab v less Delta v} for all $p$ satisfying~\eqref{equ:p}.

Similarly, integrating by parts we have
\begin{align*}
	\norm{\Delta \bff{v}}{\bb{L}^2}^2
	=
	-\inpro{\nabla \bff{v}}{\nabla\Delta \bff{v}}
	\leq
	\norm{\nabla \bff{v}}{\bb{L}^2} \norm{\nabla\Delta \bff{v}}{\bb{L}^2}
	\leq
	C\norm{\Delta \bff{v}}{\bb{L}^2} \norm{\nabla\Delta \bff{v}}{\bb{L}^2},
\end{align*}
where in the last step we used~\eqref{equ:nab v less Delta v}. This implies~\eqref{equ:Delta v less nabdelta v} for $p=2$. For $p\in (2,p^\ast]$, with $p^\ast$ as defined above, by Sobolev embedding we then have
\begin{align*}
	\norm{\Delta \bff{v}}{\bb{L}^p} 
	\leq
	\norm{\Delta \bff{v}}{\bb{H}^1}
	\leq
	\norm{\Delta \bff{v}}{\bb{L}^2} + \norm{\nabla\Delta \bff{v}}{\bb{L}^2}
	\leq 
	C \norm{\nabla\Delta \bff{v}}{\bb{L}^2},
\end{align*}
thus showing \eqref{equ:Delta v less nabdelta v} for all $p$ satisfying \eqref{equ:p}.

Inequality~\eqref{equ:v H3 ellip reg} follows from~\eqref{equ:v H3 ellip reg} as
\begin{align*}
	\norm{\bff{v}}{\bb{H}^3} \leq
	C \left(\norm{\bff{v}}{\bb{L}^2}+ \norm{\Delta \bff{v}}{\bb{H}^1}\right)
	\leq
	C \left(\norm{\bff{v}}{\bb{L}^2}+ \norm{\nabla\Delta \bff{v}}{\bb{L}^2} \right),
\end{align*}
where in the last step we used \eqref{equ:Delta v less nabdelta v}. 

Finally, regularity estimate~\eqref{equ:nab v L infty} is a result of~\cite{Maz09}. This completes the proof of the lemma.
\end{proof}

We also state the following well-known (continuous and discrete) Gronwall-type inequalities for ease of reference later.

\begin{lemma}[Gronwall's lemma]\label{lem:gron}
Let $\eta(\cdot)$ be a non-negative absolutely continuous function on $[0,T]$, which satisfies for a.e. $t$ the differential inequality
\begin{align*}
	\eta'(t)\leq \phi(t) \eta(t)+ \psi(t),
\end{align*}
where $\phi(\cdot)$ and $\psi(\cdot)$ are non-negative integrable function on $[0,T]$. Then
\begin{align*}
	\eta(t) \leq \exp\left(\int_0^t \phi(s) \,\ds\right) \left[\eta(0)+ \int_0^t \psi(s)\,\ds \right]
\end{align*}
for all $t\in [0,T]$.
\end{lemma}

\begin{proof}
	See~\cite[Appendix~B.2]{Eva10}.
\end{proof}

\begin{lemma}[Discrete Gronwall's lemma]\label{lem:disc gron}
Let $k, B, a_j, b_j$, and $\gamma_j$ be non-negative numbers (for all integers $j\geq 0$) such that
\begin{align*}
	a_n+k \sum_{j=0}^n b_j \leq B + k \sum_{j=0}^{n-1} \gamma_j a_j.
\end{align*}
Then
\begin{align*}
	a_n+k \sum_{j=0}^n b_j \leq B \exp\left(\sum_{j=0}^{n-1} \gamma_j\right).
\end{align*}
\end{lemma}

\begin{proof}
Set $u_0=B$, and for any $m\geq 1$ let $u_m:= B+ k\sum_{j=0}^{m-1} \gamma_j a_j$. Then for any $m\leq n$, we have
\begin{align*}
	u_m=u_{m-1}+ \gamma_{m-1} a_{m-1} \leq (1+\gamma_{m-1}) u_{m-1} \leq e^{\gamma_{m-1}} u_{m-1}.
\end{align*}
Noting the assumption, this implies
\begin{align*}
	 a_n+k \sum_{j=0}^n b_j 
	 \leq
	 B+ k\sum_{j=0}^{n-1} \gamma_j a_j 
	 = 
	 u_n \leq u_0 \exp\left(\sum_{j=0}^{n-1} \gamma_j \right)
	 =
	 B \exp\left(\sum_{j=0}^{n-1} \gamma_j\right),
\end{align*}
as required.
\end{proof}

\begin{lemma}[Generalised discrete Gronwall's lemma]\label{lem:disc gen gron}
Let $k, B, a_j, b_j, c_j$, and $\gamma_j$ be non-negative numbers (for all integers $j\geq 0$) such that
\begin{align*}
	a_n + k\sum_{j=0}^n b_j \leq k \sum_{j=0}^n \gamma_j a_j + k\sum_{j=0}^n c_j + B.
\end{align*}
Suppose that $k\gamma_j <1$ for all $j\geq 0$, and set $\sigma_j:= (1-k\gamma_j)^{-1}$. Then
\begin{align*}
	a_n+ k\sum_{j=0}^n b_j \leq \exp\left(k\sum_{j=0}^n \sigma_j \gamma_j\right) \left[k\sum_{j=0}^n c_j +B\right].
\end{align*}
\end{lemma}

\begin{proof}
	See~\cite[Lemma~5.1]{HeyRan90}.
\end{proof}

\section{Existence and uniqueness of strong solution}\label{sec:exist}

To show the global existence and uniqueness of solution to the problem~\eqref{equ:llb a}, one could apply the Faedo--Galerkin method and establish appropriate uniform estimates on the approximate solution. In order to simplify presentation, we will work directly with the solution $\bff{u}$ (instead of its approximation). These a priori estimates can be made rigorous by working with the Galerkin approximation as in~\cite{AyoKotMouZak21, Le16}.

In this section, we assume that $\mathscr{D}$ is a bounded open domain with smooth boundary, so that relevant elliptic regularity results can be used. Existence and uniqueness of global strong solution still hold if $\mathscr{D}$ is a convex polyhedral domain (see Remark~\ref{rem:domain}).

\begin{proposition}\label{pro:ut Lp}
Let $\bff{u}$ be a solution of~\eqref{equ:llb a} and let $p\in [2,\infty)$. For all $t\in [0,T]$,
\begin{align}\label{equ:ut Lp}
	\norm{\bff{u}(t)}{\bb{L}^p}^p
	+
	\int_0^t \alpha p\norm{|\bff{u}(s)|^{\frac{p-2}{2}} |\nabla \bff{u}(s)|}{\bb{L}^2}^2 \ds
	+
	\int_0^t p\norm{\bff{u}(s)}{\bb{L}^{p+2}}^{p+2} \ds 
	\leq
	C_p (1+t),
\end{align}
where $C_p$ is a constant depending on $p$, $\norm{\bff{u}_0}{\bb{L}^p}$, and the coefficients of \eqref{equ:llb a}. Furthermore,
\begin{equation}\label{equ:u Linfty}
	\norm{\bff{u}(t)}{\bb{L}^\infty}
	\leq 
	\norm{\bff{u}_0}{\bb{L}^\infty} + \sqrt{\beta},
\end{equation}
where $\beta:=(\beta_1\nu_\infty)^2 \alpha^{-1}$.
\end{proposition}

\begin{proof}
Taking the inner product of \eqref{equ:llb eq1} with $|\bff{u}|^{p-2} \bff{u}$, integrating by parts, and noting that $(\bff{a}\times\bff{b})\cdot \bff{a}=0$ for any $\bff{a},\bff{b}\in \bb{R}^3$, we obtain
\begin{align*}
	\frac{1}{p} \ddt \norm{\bff{u}}{\bb{L}^p}^p
	+
	\alpha
	\inpro{\nabla \bff{u}}{\nabla (\abs{\bff{u}}^{p-2} \bff{u})}
	+
	\norm{\bff{u}}{\bb{L}^{p+2}}^{p+2}
	+
	\norm{\bff{u}}{\bb{L}^p}^p
	+
	\norm{|\bff{u}|^{\frac{p-2}{2}} (\bff{e}\cdot \bff{u})}{\bb{L}^2}^2
	=
	\beta_1 \inpro{(\bff{\nu}\cdot \nabla)\bff{u}}{|\bff{u}|^{p-2} \bff{u}}.
\end{align*}
Using the identity
\begin{align*}
	\nabla\left(|\bff{u}|^q \bff{u}\right)= |\bff{u}|^q \nabla \bff{u}+ q |\bff{u}|^{q-2} \bff{u}(\bff{u}\cdot\nabla \bff{u}), \quad\forall q\geq 0,
\end{align*}
to expand the second term on the left-hand side,
we have by H\"older's and Young's inequalities,
\begin{align}\label{equ:est Lp beta}
	&\frac{1}{p} \ddt \norm{\bff{u}}{\bb{L}^p}^p
	+
	\alpha \norm{|\bff{u}|^{\frac{p-2}{2}} |\nabla \bff{u}|}{\bb{L}^2}^2
	+
	\alpha (p-2) \norm{|\bff{u}|^{\frac{p-4}{2}} |\bff{u}\cdot\nabla\bff{u}|}{\bb{L}^2}^2
	+
	\norm{\bff{u}}{\bb{L}^p}^p
	+
	\norm{|\bff{u}|^{\frac{p-2}{2}} (\bff{e}\cdot \bff{u})}{\bb{L}^2}^2
	\nonumber\\
	&\leq
	\beta_1 \nu_\infty \norm{|\bff{u}|^{\frac{p-2}{2}} |\nabla \bff{u}|}{\bb{L}^2} \norm{|\bff{u}|^{\frac{p}{2}}}{\bb{L}^2}
	-
	\norm{\bff{u}}{\bb{L}^{p+2}}^{p+2}
	\nonumber\\
	&\leq
	\frac{(\beta_1 \nu_\infty)^2}{\alpha} \norm{\bff{u}}{\bb{L}^p}^p
	+
	\frac{\alpha}{4} \norm{|\bff{u}|^{\frac{p-2}{2}} |\nabla \bff{u}|}{\bb{L}^2}^2
	-
	\norm{\bff{u}}{\bb{L}^{p+2}}^{p+2}.
\end{align}
After rearranging the terms, we have
\begin{equation}\label{equ:ut Lpp}
\ddt \norm{\bff{u}}{\bb{L}^p}^p 
+
\frac{\alpha p}{2} \norm{|\bff{u}|^{\frac{p-2}{2}} |\nabla \bff{u}|}{\bb{L}^2}^2
\leq
\beta p \norm{\bff{u}}{\bb{L}^p}^p - p\norm{\bff{u}}{\bb{L}^{p+2}}^{p+2},
\end{equation}
where $\beta=(\beta_1 \nu_\infty)^2 \alpha^{-1}$.
By considering the maximum value of the function $x\mapsto \beta x^p-x^{p+2}$ for $p\geq 2$, we have the inequality
\begin{align*}
	\beta \norm{\bff{u}}{\bb{L}^p}^p
	-
	\norm{\bff{u}}{\bb{L}^{p+2}}^{p+2}
	\leq
	\left(\frac{2\beta |\mathscr{D}|}{p+2}\right) \left(\frac{\beta p}{p+2}\right)^{\frac{p}{2}}.
\end{align*}
Substituting this to the right-hand side of~\eqref{equ:ut Lpp}, we obtain
\begin{align*}
	\ddt \norm{\bff{u}}{\bb{L}^p}^p 
	+
	\frac{\alpha p}{2} \norm{|\bff{u}|^{\frac{p-2}{2}} |\nabla \bff{u}|}{\bb{L}^2}^2
	\leq
	\left(\frac{2\beta |\mathscr{D}| p}{p+2}\right) \left(\frac{\beta p}{p+2}\right)^{\frac{p}{2}}
	\leq
	2\beta^{1+\frac{p}{2}} |\mathscr{D}|.
\end{align*}
Integrating this over $(0,t)$, we obtain~\eqref{equ:ut Lp}. In particular, rearranging the terms and taking the $p$-th root, we have
\begin{equation*}
	\norm{\bff{u}(t)}{\bb{L}^p}
	\leq
	\norm{\bff{u}_0}{\bb{L}^p}
	+
	\left(2|\mathscr{D}|\right)^{\frac{1}{p}} \beta^{\frac1p +\frac12}.
\end{equation*}
Letting $p\to\infty$ then yield~\eqref{equ:u Linfty}. This completes the proof of the proposition.
\end{proof}

\begin{remark}\label{rem:coeff above Tc}
If $\beta$ in Proposition~\ref{pro:ut Lp} is sufficiently small such that $\beta<\mu$, then we can absorb all the terms on the right-hand side of \eqref{equ:est Lp beta} to the left, leading to an exponential decay estimate
\[
	\norm{\bff{u}(t)}{\bb{L}^\infty} 
	\leq
	e^{-(\mu-\beta) t} \norm{\bff{u}_0}{\bb{L}^\infty},
\]
similar to that in~\cite{LeSoeTra24}.
\end{remark}

\begin{proposition}\label{pro:est nab u L2}
	Let $\bff{u}$ be a solution of~\eqref{equ:llb a}. For all $t\in [0,T]$,
	\begin{align}\label{equ:nab u L2}
		\norm{\nabla \bff{u}(t)}{\bb{L}^2}^2
		+
		\alpha \int_0^t \norm{\Delta \bff{u}(s)}{\bb{L}^2}^2 \ds
		\leq
		C\alpha^{-2} (1+t),
	\end{align}
	where $C$ is a constant depending on $\norm{\bff{u}_0}{\bb{H}^1}$ and the coefficients of~\eqref{equ:llb a}.
\end{proposition}

\begin{proof}
Taking the inner product of~\eqref{equ:llb eq1} with $-\Delta \bff{u}$ and integrating by parts, we have
\begin{align}\label{equ:est nab u I}
	&\frac12 \ddt \norm{\nabla \bff{u}}{\bb{L}^2}^2
	+
	\alpha \norm{\Delta \bff{u}}{\bb{L}^2}^2
	+
	\norm{|\bff{u}||\nabla\bff{u}|}{\bb{L}^2}^2
	+
	\alpha \norm{\nabla \bff{u}}{\bb{L}^2}^2
	+
	\alpha \norm{\bff{e}\cdot \nabla \bff{u}}{\bb{L}^2}^2
	\nonumber\\
	&\leq
	\inpro{\bff{u}\times \bff{e}(\bff{e}\cdot \bff{u})}{\Delta \bff{u}}
	-
	\beta_1 \inpro{(\bff{\nu}\cdot \nabla)\bff{u}}{\Delta \bff{u}}
	-
	\beta_2 \inpro{\bff{u}\times (\bff{\nu}\cdot \nabla) \bff{u}}{\Delta \bff{u}}
	\nonumber\\
	&=:
	I_1+I_2+I_3.
\end{align}
We estimate each term by H\"older's and Young's inequalities as follows:
\begin{align*}
	\abs{I_1}
	&\leq
	C\norm{\bff{u}}{\bb{L}^4}^2 \norm{\Delta \bff{u}}{\bb{L}^2}
	\leq
	C\alpha^{-1}\norm{\bff{u}}{\bb{L}^4}^4
	+
	\frac{\alpha}{8} \norm{\Delta \bff{u}}{\bb{L}^2}^2,
	\\
	\abs{I_2}
	&\leq
	C\nu_\infty \norm{\nabla \bff{u}}{\bb{L}^2} \norm{\Delta \bff{u}}{\bb{L}^2}
	\leq 
	C\alpha^{-1} \norm{\nabla \bff{u}}{\bb{L}^2}^2
	+
	\frac{\alpha}{8} \norm{\Delta \bff{u}}{\bb{L}^2}^2,
	\\
	\abs{I_3}
	&\leq
	C\nu_\infty \norm{|\bff{u}||\nabla \bff{u}|}{\bb{L}^2} \norm{\nabla \bff{u}}{\bb{L}^2}
	\leq 
	C\alpha^{-1} \norm{|\bff{u}| |\nabla \bff{u}|}{\bb{L}^2}^2
	+
	\frac{\alpha}{8} \norm{\Delta \bff{u}}{\bb{L}^2}^2.
\end{align*}
We substitute this back into~\eqref{equ:est nab u I}, rearrange the terms, and integrate over $(0,t)$. Discarding some non-negative terms on the left-hand side and noting the assumption $\alpha<1$ (so that $\alpha^{-1}\leq \alpha^{-2}$), we obtain
\begin{align*}
	\norm{\nabla \bff{u}(t)}{\bb{L}^2}^2
	+
	\alpha \int_0^t \norm{\Delta \bff{u}(s)}{\bb{L}^2}^2 \ds
	&\leq
	\norm{\nabla \bff{u}_0}{\bb{L}^2}^2
	+
	C\alpha^{-1} \int_0^t \left(\norm{\bff{u}(s)}{\bb{L}^4}^4 + \norm{\nabla \bff{u}(s)}{\bb{L}^2}^2 + \norm{|\bff{u}(s)| |\nabla \bff{u}(s)|}{\bb{L}^2}^2\right) \ds
	\\
	&\leq
	\norm{\nabla \bff{u}_0}{\bb{L}^2}^2
	+
	C\alpha^{-2}(1+t),
\end{align*}
where the last estimate is obtained by applying~\eqref{equ:ut Lp} with $p=2$ and $4$. This proves \eqref{equ:nab u L2}.
\end{proof}

\begin{proposition}\label{pro:Delta u L2}
	Let $\bff{u}$ be a solution of~\eqref{equ:llb a}. If $d=1$ or $2$, then for all $t\in [0,T]$,
	\begin{align}\label{equ:Delta u L2}
		\norm{\Delta \bff{u}(t)}{\bb{L}^2}^2
		+
		\alpha \int_0^t \norm{\nabla \Delta \bff{u}(s)}{\bb{L}^2}^2 \ds
		\leq
		Ce^{C\alpha^{-8} t^3},
	\end{align}
	where $C$ is a constant depending on the coefficients of~\eqref{equ:llb a} and $\norm{\bff{u}_0}{\bb{H}^2}$. 
	
	If $d=3$, then inequality~\eqref{equ:Delta u L2} holds under an additional assumption that $\norm{\bff{u}_0}{\bb{L}^\infty}+\beta \lesssim \alpha$, where $\beta$ was defined in Proposition~\ref{pro:ut Lp}.
\end{proposition}

\begin{proof}
Taking the inner product of~\eqref{equ:llb eq1} with $\Delta^2 \bff{u}$ and integrating by parts as necessary, we obtain
\begin{align}\label{equ:ddt Delta u J}
	&\frac12 \ddt \norm{\Delta \bff{u}}{\bb{L}^2}^2
	+
	\alpha \norm{\nabla\Delta \bff{u}}{\bb{L}^2}^2
	+
	\alpha \norm{\Delta \bff{u}}{\bb{L}^2}^2
	+
	\alpha \norm{\bff{e}\cdot \Delta \bff{u}}{\bb{L}^2}^2
	\nonumber\\
	&=
	\inpro{\nabla\bff{u}\times \Delta \bff{u}}{\nabla\Delta \bff{u}}
	-
	\inpro{\nabla \big(\bff{u}\times \bff{e}(\bff{e}\cdot \bff{u})\big)}{\nabla\Delta \bff{u}}
	+
	\alpha \inpro{\nabla \big(|\bff{u}|^2 \bff{u}\big)}{\nabla\Delta \bff{u}}
	\nonumber\\
	&\quad
	-
	\beta_1 \inpro{\nabla\big((\bff{\nu} \cdot \nabla)\bff{u}\big)}{\nabla\Delta \bff{u}}
	-
	\beta_2 \inpro{\nabla\big(\bff{u}\times (\bff{\nu}\cdot\nabla)\bff{u}\big)}{\nabla\Delta \bff{u}}
	\nonumber\\
	&=: J_1+J_2+\cdots+J_5.
\end{align}
The term $J_1$ will be handled last. For the term $J_2$, noting~\eqref{equ:u Linfty} and~\eqref{equ:nab u L2}, by H\"older's and Young's inequalities as well as \eqref{equ:u Linfty} we have
\begin{align*}
	\abs{J_2}
	&\leq
	2\norm{\bff{u}}{\bb{L}^\infty} \norm{\nabla\bff{u}}{\bb{L}^2} \norm{\nabla\Delta \bff{u}}{\bb{L}^2}
	\\
	&\leq
	C\alpha^{-1} \norm{\bff{u}}{\bb{L}^\infty}^2 \norm{\nabla \bff{u}}{\bb{L}^2}^2
	+
	\frac{\alpha}{8} \norm{\nabla\Delta \bff{u}}{\bb{L}^2}^2
	\\
	&\leq
	C\alpha^{-3}(1+t) + \frac{\alpha}{8} \norm{\nabla\Delta \bff{u}}{\bb{L}^2}^2.
\end{align*}
For the term $J_3$, by Young's inequality, \eqref{equ:nab uuv}, and \eqref{equ:u Linfty} we have
\begin{align*}
	\abs{J_3}
	&\leq
	C\alpha \norm{\nabla(|\bff{u}|^2 \bff{u})}{\bb{L}^2}^2
	+
	\frac{\alpha}{8} \norm{\nabla\Delta \bff{u}}{\bb{L}^2}^2
	\\
	&\leq
	C\alpha \norm{\bff{u}}{\bb{L}^\infty}^4 \norm{\nabla \bff{u}}{\bb{L}^2}^2
	+
	\frac{\alpha}{8} \norm{\nabla\Delta \bff{u}}{\bb{L}^2}^2
	\\
	&\leq
	C\alpha^{-1}(1+t) + \frac{\alpha}{8} \norm{\nabla\Delta \bff{u}}{\bb{L}^2}^2.
\end{align*}
For the term $J_4$, by H\"older's and Young's inequality, \eqref{equ:ut Lp}, as well as \eqref{equ:D2L2} and \eqref{equ:nab v less Delta v}, we have
\begin{align*}
	\abs{J_4}
	&\leq
	C\nu_\infty \left(\norm{\nabla \bff{u}}{\bb{L}^4}+ \norm{\nabla \bff{u}}{\bb{H}^1}\right) \norm{\nabla \Delta \bff{u}}{\bb{L}^2}
	\\
	&\leq
	C\alpha^{-1} \norm{\Delta \bff{u}}{\bb{L}^2}^2
	+
	\frac{\alpha}{8} \norm{\nabla\Delta \bff{u}}{\bb{L}^2}^2.
\end{align*}
For the last term, we expand the gradient term, then use the H\"older and the Gagliardo--Nirenberg inequality~$\norm{\bff{v}}{\bb{W}^{1,4}}^2 \leq C \norm{\bff{v}}{\bb{L}^\infty} \norm{\bff{v}}{\bb{H}^2}$, as well as elliptic regularity~\eqref{equ:v H2 elliptic} to obtain
\begin{align*}
	\abs{J_5}
	&\leq
	C\nu_\infty \norm{\nabla \bff{u}}{\bb{L}^4}^2 \norm{\nabla\Delta \bff{u}}{\bb{L}^2}
	+
	C\nu_\infty \norm{\bff{u}}{\bb{L}^\infty} \norm{\nabla \bff{u}}{\bb{L}^2} \norm{\nabla\Delta \bff{u}}{\bb{L}^2}
	+
	C\nu_\infty \norm{\bff{u}}{\bb{L}^\infty} \norm{\bff{u}}{\bb{H}^2} \norm{\nabla\Delta \bff{u}}{\bb{L}^2}
	\\
	&\leq
	C \norm{\bff{u}}{\bb{L}^\infty} \norm{\bff{u}}{\bb{H}^2} \norm{\nabla\Delta \bff{u}}{\bb{L}^2}
	\\
	&\leq
	C \norm{\bff{u}}{\bb{L}^\infty} \left(\norm{\bff{u}}{\bb{L}^2}+ \norm{\Delta \bff{u}}{\bb{L}^2} \right) \norm{\nabla\Delta \bff{u}}{\bb{L}^2}
	\\
	&\leq
	C\alpha^{-1} (1+t)^2 + C\alpha^{-1} \norm{\Delta \bff{u}}{\bb{L}^2}^2
	+
	\frac{\alpha}{8} \norm{\nabla\Delta \bff{u}}{\bb{L}^2}^2,
\end{align*}
where in the last step we used elliptic regularity~\eqref{equ:v H2 elliptic} and Proposition~\ref{pro:ut Lp}.
It remains to estimate $J_1$. To this end, we will consider each case $d=1,2$, or $3$ separately.

\smallskip
\underline{Case 1 ($d=1$):} By H\"older's and Agmon's inequalities (noting~\eqref{equ:nab u L2}), we have
\begin{align*}
	\abs{J_1}
	&\leq
	\norm{\nabla \bff{u}}{\bb{L}^2} \norm{\Delta \bff{u}}{\bb{L}^\infty} \norm{\nabla\Delta \bff{u}}{\bb{L}^2}
	\\
	&\leq
	C \norm{\nabla \bff{u}}{\bb{L}^2} \norm{\Delta \bff{u}}{\bb{L}^2}^{\frac12} \norm{\Delta \bff{u}}{\bb{H}^1}^{\frac12} \norm{\nabla\Delta \bff{u}}{\bb{L}^2}
	\leq
	C\alpha^{-5}(1+t)^2 \norm{\Delta \bff{u}}{\bb{L}^2}^2
	+
	\frac{\alpha}{8} \norm{\nabla\Delta \bff{u}}{\bb{L}^2}^2.
\end{align*}

\smallskip
\underline{Case 2 ($d=2$):} By the Gagliardo--Nirenberg inequality \eqref{equ:gal nir uh L4} and~\eqref{equ:nab u L2}, we have
\begin{align*}
	\abs{J_1}
	&\leq
	\gamma \norm{\nabla\bff{u}}{\bb{L}^4} \norm{\Delta \bff{u}}{\bb{L}^4} \norm{\nabla\Delta \bff{u}}{\bb{L}^2}
	\\
	&\leq
	C \norm{\nabla \bff{u}}{\bb{L}^2}^{\frac12} \norm{\Delta \bff{u}}{\bb{L}^2} \norm{\nabla\Delta \bff{u}}{\bb{L}^2}^{\frac32}
	\leq
	C\alpha^{-5} (1+t)^2 \norm{\Delta \bff{u}}{\bb{L}^2}^4
	+
	\frac{\alpha}{8} \norm{\nabla\Delta \bff{u}}{\bb{L}^2}^2.
\end{align*}

\smallskip
\underline{Case 3 ($d=3$):} 
We note the following Gagliardo--Nirenberg inequalities in 3D:
\begin{align}
	\label{equ:BGL}
	\norm{\nabla \bff{u}}{\bb{L}^6} 
	&\leq
	B_{\mathrm{GL}} \norm{\bff{u}}{\bb{L}^\infty}^{\frac23} \norm{\bff{u}}{\bb{H}^3}^{\frac13}
	\\
	\label{equ:CGL}
	\norm{\Delta \bff{u}}{\bb{L}^3}
	&\leq
	C_{\mathrm{GL}} \norm{\bff{u}}{\bb{L}^\infty}^{\frac13} \norm{\bff{u}}{\bb{H}^3}^{\frac23},
\end{align} 
where $B_{\mathrm{GL}}$ and $C_{\mathrm{GL}}$ are constants depending on $\mathscr{D}$.
By \eqref{equ:BGL}, \eqref{equ:CGL}, \eqref{equ:u Linfty}, and \eqref{equ:v H3 ellip reg} we have
\begin{align}\label{equ:J1 3d}
	\abs{J_1}
	&\leq
	\norm{\nabla\bff{u}}{\bb{L}^6} \norm{\Delta \bff{u}}{\bb{L}^3} \norm{\nabla\Delta \bff{u}}{\bb{L}^2}
	\nonumber\\
	&\leq
	B_{\mathrm{GL}} C_{\mathrm{GL}} \norm{\bff{u}}{\bb{L}^\infty} \norm{\bff{u}}{\bb{H}^3} \norm{\nabla\Delta \bff{u}}{\bb{L}^2}
	\nonumber\\
	&\leq
	B_{\mathrm{GL}} C_{\mathrm{GL}} \left(\norm{\bff{u}_0}{\bb{L}^\infty} +\sqrt{\beta}\right)
	\big(\norm{\bff{u}}{\bb{L}^2} + \norm{\nabla\Delta \bff{u}}{\bb{L}^2} \big) \norm{\nabla\Delta \bff{u}}{\bb{L}^2}
	\nonumber\\
	&\leq
	C\alpha^{-1} \left(1+t \right)
	+
	\frac{\alpha}{8} \norm{\nabla\Delta \bff{u}}{\bb{L}^2}^2
	+
	B_{\mathrm{GL}} C_{\mathrm{GL}} \left(\norm{\bff{u}_0}{\bb{L}^\infty} +\sqrt{\beta}\right) \norm{\nabla\Delta \bff{u}}{\bb{L}^2}^2,
\end{align}
where in the last step we also used~\eqref{equ:ut Lp} and Young's inequality, while $\beta$ was defined in Proposition~\ref{pro:ut Lp}.

We now collect all the above estimates and substitute them back to~\eqref{equ:ddt Delta u J}. If $d=1$ or $2$, we have
\begin{align}\label{equ:ddt norm Delta L2}
	\ddt \norm{\Delta \bff{u}}{\bb{L}^2}^2
	+
	\norm{\nabla\Delta \bff{u}}{\bb{L}^2}^2
	\leq
	C\alpha^{-5} (1+t)^2 \left(1+\norm{\Delta \bff{u}}{\bb{L}^2}^4 \right).
\end{align}
Therefore, by the Gronwall lemma (Lemma~\ref{lem:gron}), we obtain
\begin{align}\label{equ:Delta exp}
	\norm{\Delta \bff{u}(t)}{\bb{L}^2}^2
	&\leq
	\exp\left( C\alpha^{-5} (1+t)^2 \int_0^t \norm{\Delta \bff{u}(s)}{\bb{L}^2}^2 \,\ds \right) \left[\norm{\Delta \bff{u}_0}{\bb{L}^2}^2 + C\alpha^{-5} (1+t^3)\right]
	\leq
	Ce^{C\alpha^{-8} t^3},
\end{align}
where $C$ depends on $\norm{\bff{u}_0}{\bb{H}^2}$, and in the last step we also used \eqref{equ:nab u L2}.

On the other hand, for~$d=3$, if
\begin{equation}\label{equ:small alpha}
		B_{\mathrm{GL}} C_{\mathrm{GL}} \left(\norm{\bff{u}_0}{\bb{L}^\infty} +\sqrt{\beta}\right)
		\leq \alpha/8,
\end{equation}
then we can absorb the term containing $\norm{\nabla\Delta \bff{u}}{\bb{L}^2}^2$ in~\eqref{equ:J1 3d}. Integrating \eqref{equ:ddt norm Delta L2} over $(0,t)$ and using \eqref{equ:Delta exp} again, we infer~\eqref{equ:Delta u L2}.
\end{proof}

\begin{remark}
Proposition~\ref{pro:Delta u L2} implies the existence of a global solution~$\bff{u}\in C([0,T];\bb{H}^2)\cap L^2(0,T;\bb{H}^4)$ for any initial data $\bff{u}_0\in \bb{H}^2$ for $d=1$ or $2$. For $d=3$, this holds under an additional assumption that $\alpha$ is sufficiently large, or $\norm{\bff{u}_0}{\bb{L}^\infty}+\sqrt{\beta}$ is sufficiently small (see~\eqref{equ:small alpha}). If the coefficients of~\eqref{equ:llb a} satisfy the assumptions described in Remark~\ref{rem:coeff above Tc}, then $\beta$ can be taken to be zero.
Note that a \emph{local} solution with the aforementioned regularity exists by similar argument as in~\cite{LeSoeTra24}.
\end{remark}

From this point onwards, for ease of presentation we will not track the dependence of the estimates on $\alpha$, $\beta_1$, or $\beta_2$ anymore since they will not be essential for the argument. Previously, we track this dependence to ensure it does not interfere with the smallness condition on the initial data assumed in~\eqref{equ:small alpha}.

\begin{proposition}
Let $\bff{u}$ be a solution of~\eqref{equ:llb a} such that Proposition~\ref{pro:Delta u L2} holds. For all $t\in [0,T]$,
\begin{align}\label{equ:nab Delta u L2}
	t\norm{\nabla \Delta \bff{u}(t)}{\bb{L}^2}^2
	+
	\alpha \int_0^t s\norm{\Delta^2 \bff{u}(s)}{\bb{L}^2}^2 \ds
	\leq
	Ce^{Ct^3},
\end{align}
where $C$ is a constant depending on $\norm{\bff{u}_0}{\bb{H}^2}$ and the coefficients of~\eqref{equ:llb a}.
\end{proposition}

\begin{proof}
Applying the operator $-\Delta$ on~\eqref{equ:llb eq1}, then taking the inner product of the result with $t\Delta^2 \bff{u}$ and integrating by parts as necessary, we obtain
\begin{align}\label{equ:ddt t nab Delta u}
	&\frac12 \ddt \left(t \norm{\nabla \Delta \bff{u}}{\bb{L}^2}^2\right)
	+
	\alpha t\norm{\Delta^2 \bff{u}}{\bb{L}^2}^2
	+
	\alpha t \norm{\nabla \Delta \bff{u}}{\bb{L}^2}^2
	+
	\alpha t \norm{\bff{e}\cdot \nabla\Delta \bff{u}}{\bb{L}^2}^2
	\nonumber\\
	&=
	\frac12 \norm{\nabla\Delta \bff{u}}{\bb{L}^2}^2
	-
	2 t\inpro{\nabla\bff{u}\times \nabla\Delta \bff{u}}{\Delta^2 \bff{u}}
	-
	t\inpro{\Delta \big(\bff{u}\times \bff{e}(\bff{e}\cdot \bff{u})\big)}{\Delta^2 \bff{u}}
	\nonumber\\
	&\quad
	+
	\alpha t \inpro{\Delta(|\bff{u}|^2 \bff{u})}{\Delta^2 \bff{u}}
	-
	\beta_1 t \inpro{\Delta\big((\bff{\nu} \cdot \nabla)\bff{u}\big)}{\Delta^2 \bff{u}}
	-
	\beta_2 t \inpro{\Delta\big(\bff{u}\times (\bff{\nu}\cdot\nabla)\bff{u}\big)}{\Delta^2 \bff{u}}
	\nonumber\\
	&=: J_1+J_2+\cdots+J_6.
\end{align}
We will estimate each term on the last line, noting the estimates~\eqref{equ:nab u L2} and~\eqref{equ:Delta u L2} obtained previously. The first term is left as is. To estimate $J_2$, we first note that by the Gagliardo--Nirenberg inequality~\eqref{equ:gal nir uh L3} and elliptic regularity we have
\begin{align*}
	\norm{\nabla\Delta \bff{u}}{\bb{L}^3} \leq 
	C\norm{\nabla\Delta \bff{u}}{\bb{L}^2}^{1-\frac{d}{6}} \norm{\nabla\Delta\bff{u}}{\bb{H}^1}^{\frac{d}{6}}
	\leq
	C\norm{\nabla\Delta \bff{u}}{\bb{L}^2}^{1-\frac{d}{6}} \norm{\Delta^2 \bff{u}}{\bb{L}^2}^{\frac{d}{6}},
\end{align*}
where $d\leq 3$ is the spatial dimension.
Therefore, we infer by~\eqref{equ:nab v less Delta v} and Young's inequality (with exponents $p=\frac{12}{6-d}$ and $q=\frac{12}{6+d}$) that
\begin{align*}
	\abs{J_2}
	\leq
	2 t \norm{\nabla\bff{u}}{\bb{L}^6} \norm{\nabla\Delta \bff{u}}{\bb{L}^3} \norm{\Delta^2 \bff{u}}{\bb{L}^2}
	&\leq
	Ct \norm{\Delta \bff{u}}{\bb{L}^2} \norm{\nabla\Delta \bff{u}}{\bb{L}^2}^{1-\frac{d}{6}}
	\norm{\Delta^2 \bff{u}}{\bb{L}^2}^{1+\frac{d}{6}}
	\\
	&\leq
	Ct \norm{\Delta \bff{u}}{\bb{L}^2}^{\frac{12}{6-d}} \norm{\nabla \Delta \bff{u}}{\bb{L}^2}^2
	+
	\frac{\alpha t}{10} \norm{\Delta^2 \bff{u}}{\bb{L}^2}^2.
	\\
	&\leq
	Ct e^{Ct^3} \norm{\nabla \Delta \bff{u}}{\bb{L}^2}^2
	+
	\frac{\alpha t}{10} \norm{\Delta^2 \bff{u}}{\bb{L}^2}^2,
\end{align*}
where in the last step we used \eqref{equ:Delta u L2}.
For the term $J_3$, using the fact that $\bb{H}^2$ is an algebra for $d\leq 3$, by elliptic regularity \eqref{equ:v H2 elliptic}, inequalities \eqref{equ:ut Lp} and \eqref{equ:Delta u L2}, we obtain
\begin{align*}
	\abs{J_3}
	&\leq
	Ct \norm{\bff{u}}{\bb{H}^2}^2 \norm{\Delta^2 \bff{u}}{\bb{L}^2}
	\leq
	Cte^{Ct^3} + \frac{\alpha t}{10} \norm{\Delta^2 \bff{u}}{\bb{L}^2}^2.
\end{align*}
Similarly, for the term $J_4$, we have
\begin{align*}
	\abs{J_4}
	&\leq
	Ct \norm{\bff{u}}{\bb{H}^2}^3 \norm{\Delta^2 \bff{u}}{\bb{L}^2}
	\leq
	Cte^{Ct^3} + \frac{\alpha t}{10} \norm{\Delta^2 \bff{u}}{\bb{L}^2}^2.
\end{align*}
For the last two terms, by H\"older's and Young's inequalities, \eqref{equ:ut Lp} and \eqref{equ:Delta u L2}, as well as elliptic regularity estimates~\eqref{equ:v H2 elliptic} and \eqref{equ:v H3 ellip reg}, we have
\begin{align*}
	\abs{J_5}
	&\leq 
	Ct \norm{\nabla\bff{u}}{\bb{H}^2} \norm{\Delta^2 \bff{u}}{\bb{L}^2}
	\leq
	Cte^{Ct^3} \left(1+\norm{\nabla\Delta \bff{u}}{\bb{L}^2}^2 \right)
	+
	\frac{\alpha t}{10} \norm{\Delta^2 \bff{u}}{\bb{L}^2}^2,
	\\
	\abs{J_6}
	&\leq
	Ct \norm{\bff{u}}{\bb{H}^2} \norm{\nabla \bff{u}}{\bb{H}^2} \norm{\Delta^2 \bff{u}}{\bb{L}^2}
	\leq
	Cte^{Ct^3} \left(1+\norm{\nabla\Delta \bff{u}}{\bb{L}^2}^2 \right)
	+
	\frac{\alpha t}{10} \norm{\Delta^2 \bff{u}}{\bb{L}^2}^2.
\end{align*}
Altogether, from~\eqref{equ:ddt t nab Delta u} we obtain
\begin{align*}
	&\ddt \left(t \norm{\nabla \Delta \bff{u}}{\bb{L}^2}^2\right)
	+
	\alpha t\norm{\Delta^2 \bff{u}}{\bb{L}^2}^2
	\leq
	\norm{\nabla\Delta \bff{u}}{\bb{L}^2}^2
	+
	Cte^{Ct^3} \left(1+\norm{\nabla\Delta \bff{u}}{\bb{L}^2}^2 \right).
\end{align*}
Integrating both sides over $(0,t)$ and applying~\eqref{equ:Delta u L2}, we obtain 
\begin{align*}
	t\norm{\nabla \Delta \bff{u}(t)}{\bb{L}^2}^2
	+
	\alpha \int_0^t s\norm{\Delta^2 \bff{u}(s)}{\bb{L}^2}^2 \ds
	&\leq
	\int_0^t \norm{\nabla\Delta \bff{u}(s)}{\bb{L}^2}^2
	+
	\int_0^t Cse^{Cs^3} \norm{\nabla\Delta \bff{u}(s)}{\bb{L}^2}^2 \ds 
	+
	\int_0^t Cse^{Cs^3} \ds 
	\\
	&\leq 
	Ce^{Ct^3} + Cte^{Ct^3} \int_0^t \norm{\nabla\Delta \bff{u}(s)}{\bb{L}^2}^2 \ds 
	+
	Ct^2 e^{Ct^3}
	\leq
	Ce^{Ct^3},
\end{align*}
where in the last step we used an elementary inequality $1+t+t^2\leq Ce^{Ct^3}$ for all $t\geq 0$ and \eqref{equ:Delta u L2} again. This proves the required estimate.
\end{proof}

\begin{proposition}
	Let $\bff{u}$ be a solution of~\eqref{equ:llb a} such that Proposition~\ref{pro:Delta u L2} holds. For all $t\in [0,T]$,
	\begin{align}\label{equ:Delta22 u L2}
		t^2 \norm{\Delta^2 \bff{u}(t)}{\bb{L}^2}^2
		+
		\alpha \int_0^t s^2 \norm{\nabla \Delta^2 \bff{u}(s)}{\bb{L}^2}^2 \ds
		\leq
		Ce^{Ct^3},
	\end{align}
	where $C$ is a constant depending on $\norm{\bff{u}_0}{\bb{H}^2}$ and the coefficients of~\eqref{equ:llb a}.
\end{proposition}

\begin{proof}
Applying the operator $\Delta^2$ on~\eqref{equ:llb eq1}, then taking the inner product of the result with $t^2 \Delta^2 \bff{u}$ and integrating by parts as necessary, we obtain	
\begin{align}\label{equ:ddt t2 Delta2u}
	&\frac12 \ddt\left(t^2\norm{\Delta^2 \bff{u}}{\bb{L}^2}^2\right)
	+
	\alpha t^2 \norm{\nabla\Delta^2 \bff{u}}{\bb{L}^2}^2
	+
	\alpha t^2 \norm{\Delta^2 \bff{u}}{\bb{L}^2}^2
	+
	\alpha t^2 \norm{\bff{e}\cdot \Delta^2 \bff{u}}{\bb{L}^2}^2
	\nonumber\\
	&=
	t\norm{\Delta^2 \bff{u}}{\bb{L}^2}^2
	+
	t^2 \inpro{\nabla\Delta\big(\bff{u}\times \Delta\bff{u}\big)}{\nabla\Delta^2 \bff{u}}
	-
	t^2 \inpro{\nabla\Delta \big(\bff{u}\times \bff{e}(\bff{e}\cdot \bff{u})\big)}{\nabla\Delta^2 \bff{u}}
	\nonumber\\
	&\quad
	+
	\alpha t^2 \inpro{\nabla\Delta\big(|\bff{u}|^2 \bff{u}\big)}{\nabla\Delta^2 \bff{u}}
	+
	\beta_1 t^2 \inpro{\Delta^2 \big((\bff{\nu}\cdot\nabla)\bff{u} \big)}{\Delta^2 \bff{u}}
	-
	\beta_2 t^2 \inpro{\nabla\Delta \big(\bff{u}\times (\bff{\nu}\cdot\nabla)\bff{u} \big)}{\nabla\Delta^2 \bff{u}}
	\nonumber\\
	&=
	I_1+I_2+\cdots+I_6.
\end{align}
We will estimate each term on the last line. In simplifying the estimates, we often use an elementary inequality $1+t+t^2 \leq Ce^{Ct^3}$. Now, the first term in \eqref{equ:ddt t2 Delta2u} is kept as is. For the term $I_2$, using the fact that $\bb{H}^2$ is an algebra for $d\leq 3$, H\"older's inequality, the embedding $\bb{H}^2\hookrightarrow \bb{L}^\infty$, and elliptic regularity estimates \eqref{equ:v H2 elliptic} and \eqref{equ:v H3 ellip reg}, we have
\begin{align*}
	\abs{I_2}
	&=
	t^2 \abs{\inpro{\Delta(\nabla \bff{u}\times \Delta \bff{u})}{\nabla\Delta^2 \bff{u}} + \inpro{\Delta(\bff{u}\times \nabla\Delta \bff{u})}{\nabla\Delta^2 \bff{u}}}
	\\
	&=
	t^2 \abs{\inpro{\Delta(\nabla \bff{u}\times \Delta \bff{u})}{\nabla\Delta^2 \bff{u}} + \inpro{\Delta\bff{u}\times \nabla\Delta \bff{u})}{\nabla\Delta^2 \bff{u}} + \inpro{\nabla\bff{u}\times \nabla^2 \Delta \bff{u})}{\nabla\Delta^2 \bff{u}}}
	\\
	&\leq
	Ct^2 \left(\norm{\nabla \bff{u}\times \Delta \bff{u}}{\bb{H}^2} +
	\norm{\Delta \bff{u}}{\bb{L}^\infty} \norm{\nabla\Delta\bff{u}}{\bb{L}^2}
	+
	\norm{\nabla \bff{u}}{\bb{L}^\infty} \norm{\nabla^2 \Delta \bff{u}}{\bb{L}^2}\right) \norm{\nabla\Delta^2 \bff{u}}{\bb{L}^2}
	\\
	&\leq
	Ct^2 \left(\norm{\nabla \bff{u}}{\bb{H}^2} \norm{\Delta \bff{u}}{\bb{H}^2} + \norm{\Delta \bff{u}}{\bb{H}^2} \norm{\nabla\Delta \bff{u}}{\bb{L}^2} + \norm{\nabla \bff{u}}{\bb{H}^2} \norm{\Delta \bff{u}}{\bb{H}^2} \right) \norm{\nabla\Delta^2 \bff{u}}{\bb{L}^2}
	\\
	&\leq
	Ct^2 \left(\norm{\bff{u}}{\bb{L}^2} + \norm{\nabla\Delta \bff{u}}{\bb{L}^2}\right) \left(\norm{\Delta \bff{u}}{\bb{L}^2} + \norm{\Delta^2 \bff{u}}{\bb{L}^2} \right) \norm{\nabla\Delta^2 \bff{u}}{\bb{L}^2}
	\\
	&\leq
	Ct e^{Ct^3} \norm{\Delta^2 \bff{u}}{\bb{L}^2}^2
	+
	Ce^{Ct^3}
	+
	\frac{\alpha t^2}{10} \norm{\nabla\Delta^2 \bff{u}}{\bb{L}^2}^2,
\end{align*}
where in the last step we used Young's inequality, \eqref{equ:ut Lp}, \eqref{equ:Delta u L2}, and \eqref{equ:nab Delta u L2}. For the term $I_3$, we argue in a similar manner to obtain
\begin{align*}
	\abs{I_3} \leq
	Ct^2 \norm{\nabla \bff{u}}{\bb{H}^2} \norm{\bff{u}}{\bb{H}^2}\norm{\nabla\Delta^2 \bff{u}}{\bb{L}^2}^2
	\leq
	Ce^{Ct^3} + \frac{\alpha t^2}{10} \norm{\nabla\Delta^2 \bff{u}}{\bb{L}^2}^2.
\end{align*}
For the term $I_4$, noting the identity~\eqref{equ:nab uuv}, using H\"older's and Young's inequalities, the fact that $\bb{H}^2$ is an algebra, and elliptic regularity estimates \eqref{equ:v H2 elliptic} and \eqref{equ:v H3 ellip reg}, we have
\begin{align*}
	\abs{I_4}
	&=
	\alpha t^2 \abs{2\inpro{\Delta\big(\bff{u} (\bff{u}\cdot\nabla\bff{u})\big)}{\nabla\Delta^2 \bff{u}} + \inpro{\Delta\big(|\bff{u}|^2 \nabla\bff{u}\big)}{\nabla\Delta^2 \bff{u}}}
	\\
	&\leq
	Ct^2 \norm{\bff{u}}{\bb{H}^2}^2 \norm{\nabla \bff{u}}{\bb{H}^2} \norm{\nabla\Delta^2 \bff{u}}{\bb{L}^2} 
	\\
	&\leq
	Ct^2 \norm{\bff{u}}{\bb{H}^2}^4 \left(\norm{\bff{u}}{\bb{L}^2}^2 + \norm{\nabla\Delta \bff{u}}{\bb{L}^2}^2 \right)
	+
	\frac{\alpha t^2}{10} \norm{\nabla\Delta^2 \bff{u}}{\bb{L}^2}^2
	\\
	&\leq
	Cte^{Ct^3}
	+
	\frac{\alpha t^2}{10} \norm{\nabla\Delta^2 \bff{u}}{\bb{L}^2}^2,
\end{align*}
where in the last step we used \eqref{equ:ut Lp}, \eqref{equ:Delta u L2}, and \eqref{equ:nab Delta u L2}. For the term $I_5$, it is clear that we have
\begin{align*}
	\abs{I_5}
	&\leq
	Ct^2 \norm{\Delta^2 \bff{u}}{\bb{L}^2}^2 +
	\frac{\alpha t^2}{10} \norm{\nabla\Delta^2 \bff{u}}{\bb{L}^2}^2.
\end{align*}
For the last term, we apply similar argument as in the estimate for $I_2$ to obtain
\begin{align*}
	\abs{I_6}
	&\leq
	Cte^{Ct^3} \norm{\Delta^2 \bff{u}}{\bb{L}^2}^2
	+
	Ce^{Ct^3}
	+
	\frac{\alpha t^2}{10} \norm{\nabla\Delta^2 \bff{u}}{\bb{L}^2}^2.
\end{align*}
Altogether, we substitute these estimates into~\eqref{equ:ddt t2 Delta2u} to obtain
\begin{align*}
	\ddt\left(t^2\norm{\Delta^2 \bff{u}}{\bb{L}^2}^2\right)
	+
	\alpha t^2 \norm{\nabla\Delta^2 \bff{u}}{\bb{L}^2}^2
	&\leq
	2t \norm{\Delta^2 \bff{u}}{\bb{L}^2}^2
	+
	Cte^{Ct^3} \norm{\Delta^2 \bff{u}}{\bb{L}^2}^2
	+
	Ce^{Ct^3}.
\end{align*}
Integrating over $(0,t)$, we obtain 
\begin{align*}
	&t^2 \norm{\Delta^2 \bff{u}(t)}{\bb{L}^2}^2
	+
	\alpha \int_0^t s^2 \norm{\nabla\Delta^2 \bff{u}(s)}{\bb{L}^2}^2 \ds 
	\\
	&\leq
	\int_0^t Cs \norm{\Delta^2 \bff{u}(s)}{\bb{L}^2}^2 \ds
	+
	\int_0^t Cse^{Cs^3} \norm{\Delta^2 \bff{u}(s)}{\bb{L}^2}^2 \ds 
	+
	\int_0^t Ce^{Cs^3} \ds
	\\
	&\leq
	Ce^{Ct^3}
	+
	Ce^{Ct^3} \int_0^t s \norm{\Delta^2 \bff{u}(s)}{\bb{L}^2}^2 \ds +
	\int_0^t Ce^{Cs^3} \ds
	\leq
	Ce^{Ct^3},
\end{align*}
where in the last step we used \eqref{equ:nab Delta u L2}.
This completes the proof.
\end{proof}

\begin{theorem}\label{the:main existence}
Let the initial data $\bff{u}_0\in \bb{H}^2$ be given, with an additional condition that $\norm{\bff{u}_0}{\bb{L}^\infty}$ is sufficiently small (as given by Proposition~\ref{pro:Delta u L2}) if $d=3$. Then there exists a unique global strong solution $\bff{u}$ to~\eqref{equ:llb a} in the sense of Definition~\ref{def:strong sol}.

Furthermore, this solution satisfies
\begin{align}\label{equ:smooth H4}
	\norm{\bff{u}(t)}{\bb{H}^4}^2 \leq Ce^{Ct^3}(1+t^{-2}),
\end{align}
where $C$ is a constant depending on $\norm{\bff{u}_0}{\bb{H}^2}$ and the coefficients of~\eqref{equ:llb a}.
\end{theorem}

\begin{proof}
The existence of a global strong solution follows from a standard compactness argument and the Aubin--Lions lemma, making use of the uniform a priori estimates in~\eqref{equ:u Linfty}, \eqref{equ:nab u L2}, and~\eqref{equ:Delta u L2}. Inequality~\eqref{equ:smooth H4} follows from \eqref{equ:u Linfty}, \eqref{equ:nab u L2}, \eqref{equ:Delta u L2}, \eqref{equ:nab Delta u L2}, and~\eqref{equ:Delta22 u L2}.

Finally, we show the uniqueness of strong solution by deriving a continuous dependence estimate. Let $\bff{u}$ and $\bff{v}$ be strong solutions of \eqref{equ:llb a} corresponding to initial data $\bff{u}_0$ and $\bff{v}_0$, respectively. Let $\bff{w}:=\bff{u}-\bff{v}$ and $\bff{w}_0:= \bff{u}_0-\bff{v}_0$. Then for almost every $(t,x)\in (0,T)\times \mathscr{D}$, we have
\begin{align*}
	\partial_t \bff{w}
	&=
	\alpha \Delta \bff{w}
	-
	\alpha \bff{w}
	-
	\alpha \left(|\bff{u}|^2\bff{w} + ((\bff{u}+\bff{v})\cdot \bff{w})\bff{v}\right)
	-
	\alpha \bff{e}(\bff{e}\cdot \bff{w})
	-
	\left(\bff{u}\times \Delta \bff{w} + \bff{w}\times \Delta \bff{v}\right)
	\\
	&\quad 
	+
	\left(\bff{u}\times \bff{e}(\bff{e}\cdot \bff{w}) + \bff{w}\times \bff{e}(\bff{e}\cdot \bff{v})\right)
	+
	\beta_1 (\bff{\nu}\cdot \nabla) \bff{w}
	+
	\beta_2 \left(\bff{u}\times (\bff{\nu}\cdot \nabla) \bff{w} + \bff{w} \times (\bff{\nu}\cdot\nabla)\bff{v} \right).
\end{align*}
Taking the inner product of this equation with $\bff{w}$, we obtain
\begin{align}\label{equ:ineq J1 J7}
	&\frac12 \ddt \norm{\bff{w}}{\bb{L}^2}^2
	+
	\alpha \norm{\nabla \bff{w}}{\bb{L}^2}^2
	+
	\alpha \norm{\bff{w}}{\bb{L}^2}^2
	+
	\alpha \norm{|\bff{u}| |\bff{w}|}{\bb{L}^2}^2
	+
	\alpha \norm{\bff{v}\cdot \bff{w}}{\bb{L}^2}^2
	+
	\alpha \norm{\bff{e}\cdot \bff{w}}{\bb{L}^2}^2
	\nonumber\\
	&=
	- \alpha \inpro{(\bff{u}\cdot \bff{w})\bff{v}}{\bff{w}}
	-
	\inpro{\nabla\bff{u}\times \bff{w}}{\nabla \bff{w}}
	+
	\gamma \inpro{\bff{u}\times \bff{e}(\bff{e}\cdot \bff{w})}{\bff{w}}
	+
	\beta_1 \inpro{(\bff{\nu}\cdot \nabla) \bff{w}}{\bff{w}}
	+
	\beta_2 \inpro{\bff{u}\times (\bff{\nu}\cdot \nabla) \bff{w}}{\bff{w}} 
	\nonumber\\
	&=:
	J_1+J_2+\cdots+J_5.
\end{align}
Straightforward applications of Young's inequality and Sobolev embedding yields
\begin{align*}
	\abs{J_1} 
	&\leq
	\frac{\alpha}{2} \norm{|\bff{u}||\bff{w}|}{\bb{L}^2}^2
	+
	\frac{\alpha}{2} \norm{\bff{v}\cdot \bff{w}}{\bb{L}^2}^2,
	\\
	\abs{J_3}
	&\leq
	C\norm{\bff{u}}{\bb{L}^\infty}^2 \norm{\bff{w}}{\bb{L}^2}^2
	+
	C\norm{\bff{w}}{\bb{L}^2}^2
	\leq
	C\norm{\bff{u}}{\bb{H}^2}^2 \norm{\bff{w}}{\bb{L}^2}^2
	+
	C\norm{\bff{w}}{\bb{L}^2}^2
	\\
	\abs{J_4}
	&\leq
	C\norm{\bff{w}}{\bb{L}^2}^2
	+
	\frac{\alpha}{8} \norm{\nabla \bff{w}}{\bb{L}^2}^2,
	\\
	\abs{J_5}
	&\leq
	C\norm{\bff{u}}{\bb{L}^\infty}^2 \norm{\bff{w}}{\bb{L}^2}^2
	+
	\frac{\alpha}{8} \norm{\nabla \bff{w}}{\bb{L}^2}^2
	\leq
	C\norm{\bff{u}}{\bb{H}^2}^2 \norm{\bff{w}}{\bb{L}^2}^2
	+
	\frac{\alpha}{8} \norm{\nabla \bff{w}}{\bb{L}^2}^2
\end{align*}
For the term $J_2$, by H\"older's inequality, \eqref{equ:nab v L infty}, and \eqref{equ:Delta v less nabdelta v}, we have
\begin{align*}
	\abs{J_2}
	\leq
	\norm{\nabla \bff{u}}{\bb{L}^\infty} \norm{\bff{w}}{\bb{L}^2} \norm{\nabla \bff{w}}{\bb{L}^2}
	&\leq
	C\norm{\Delta\bff{u}}{\bb{L}^4}^2 \norm{\bff{w}}{\bb{L}^2}^2 
	+
	\frac{\alpha}{8} \norm{\nabla \bff{w}}{\bb{L}^2}^2
	\\
	&\leq
	C\norm{\nabla\Delta\bff{u}}{\bb{L}^2}^2 \norm{\bff{w}}{\bb{L}^2}^2 
	+
	\frac{\alpha}{8} \norm{\nabla \bff{w}}{\bb{L}^2}^2.
\end{align*}
We substitute the above estimates into~\eqref{equ:ineq J1 J7} and use elliptic regularity to obtain
\begin{align*}
	\ddt \norm{\bff{w}}{\bb{L}^2}^2
	+
	\alpha \norm{\nabla \bff{w}}{\bb{L}^2}^2
	\leq
	C\left(1+ \norm{\bff{u}}{\bb{H}^2}^2 + \norm{\nabla \Delta \bff{u}}{\bb{L}^2}^2\right) \norm{\bff{w}}{\bb{L}^2}^2.
\end{align*}
By the Gronwall lemma (Lemma~\ref{lem:gron}), noting \eqref{equ:ut Lp}, \eqref{equ:nab u L2}, and \eqref{equ:Delta u L2}, we have
\begin{align*}
	\norm{\bff{w}(t)}{\bb{L}^2}^2
	\leq
	\exp\left(\int_0^t C\left(1+ \norm{\bff{u}(s)}{\bb{H}^2}^2 + \norm{\nabla \Delta \bff{u}(s)}{\bb{L}^2}^2\right) \ds \right) \norm{\bff{w}(0)}{\bb{L}^2}^2
	\leq
	\exp\left(Ce^{Ct^3}\right) \norm{\bff{w}_0}{\bb{L}^2}^2.
\end{align*}
Uniqueness of the solution then follows immediately from this stability estimate.
\end{proof}

Physical theory predicts a loss of magnetisation in a ferromagnet above the Curie temperature, the rate of which is important in the study of ultrafast demagnetisation phenomena at elevated temperatures~\cite{Acr25, ChuNie20, JiaRenZha13}. We end this section with the following result, which shows precisely the asymptotic behaviour of the magnetisation $\bff{u}$ in the absence of current ($\bff{\nu}=\bff{0}$) above the Curie temperature, as predicted by the physical theory. We keep the coefficients of the equation as in~\eqref{equ:llb a} for this proposition to obtain a precise statement on the rate of decay. Inequality~\eqref{equ:ut L infty exp} has been shown in~\cite{LeSoeTra24}, but we provide the proof here in our context for ease of reference. Inequality~\eqref{equ:ut energy exp} and its consequence is new.

\begin{theorem}\label{the:decay}
Let $\bff{\nu}=\bff{0}$ and let $\mathcal{E}$ be the energy functional defined in~\eqref{equ:energy}. Suppose that $\bff{u}$ is a strong solution to \eqref{equ:llb a} corresponding to an initial data $\bff{u}_0\in \bb{H}^2$. Then we have
\begin{align}
	\label{equ:ut L infty exp}
	\norm{\bff{u}(t)}{\bb{L}^\infty} &\leq e^{-\alpha\kappa\mu t} \norm{\bff{u}_0}{\bb{L}^\infty},
	\\
	\label{equ:ut energy exp}
	\mathcal{E}\big(\bff{u}(t)\big) &\leq e^{-2\alpha\kappa\mu t} \mathcal{E} \big(\bff{u}_0\big).
\end{align}
In particular, $\bff{u}(t)\to \bff{0}$ uniformly in space and $\mathcal{E}\big(\bff{u}(t)\big)\to 0$ (thus $\norm{\bff{u}(t)}{\bb{H}^1}\to 0$ as well) as $t\to\infty$. This implies that as time progresses, the magnetisation vectors will align with each other with magnitude decaying to zero.
\end{theorem}

\begin{proof}
As in the proof of~\eqref{equ:ut Lp} but with $\bff{\nu}=\bff{0}$, taking the inner product of~\eqref{equ:llb eq1} with $|\bff{u}|^{p-2} \bff{u}$ and integrating by parts, we have
\begin{align*}
	\frac{1}{p} \ddt \norm{\bff{u}}{\bb{L}^p}^p
	+
	\alpha\sigma
	\inpro{\nabla \bff{u}}{\nabla (\abs{\bff{u}}^{p-2} \bff{u})}
	+
	\alpha\kappa\mu \norm{\bff{u}}{\bb{L}^p}^p
	+
	\alpha\kappa \norm{\bff{u}}{\bb{L}^{p+2}}^{p+2}
	+
	\alpha\lambda \norm{|\bff{u}|^{\frac{p-2}{2}} (\bff{e}\cdot \bff{u})}{\bb{L}^2}^2
	=
	0.
\end{align*}
Therefore,
\begin{align*}
	&\frac{1}{p} \ddt \norm{\bff{u}}{\bb{L}^p}^p
	+
	\alpha \sigma \norm{|\bff{u}|^{\frac{p-2}{2}} |\nabla \bff{u}|}{\bb{L}^2}^2
	+
	\alpha \sigma (p-2) \norm{|\bff{u}|^{\frac{p-4}{2}} |\bff{u}\cdot\nabla\bff{u}|}{\bb{L}^2}^2
	\\
	&\quad
	+
	\alpha\kappa\mu \norm{\bff{u}}{\bb{L}^p}^p
	+
	\alpha\kappa \norm{\bff{u}}{\bb{L}^{p+2}}^{p+2}
	+
	\alpha\lambda \norm{|\bff{u}|^{\frac{p-2}{2}} (\bff{e}\cdot \bff{u})}{\bb{L}^2}^2
	= 0,
\end{align*}
which implies 
\begin{align*}
	\ddt \norm{\bff{u}}{\bb{L}^p}^p
	+
	p\alpha\kappa\mu \norm{\bff{u}}{\bb{L}^p}^p
	\leq 0,
\end{align*}
so that
\begin{align*}
	\ddt \left( e^{p\alpha\kappa\mu} \norm{\bff{u}}{\bb{L}^p}^p \right) \leq 0.
\end{align*}
Integrating this over $(0,t)$ and taking $p$-th root give
\begin{align*}
	\norm{\bff{u}(t)}{\bb{L}^p} \leq e^{-\alpha\kappa\mu t} \norm{\bff{u}_0}{\bb{L}^p}.
\end{align*}
Letting $p\to\infty$ shows~\eqref{equ:ut L infty exp}.

Next, we show~\eqref{equ:ut energy exp}. The same argument as \eqref{equ:dtu H} and \eqref{equ:min H dtu} but with $\bff{\nu}=\bff{0}$ yields
\begin{align}\label{equ:ddt energy}
	\ddt \left( \frac{\sigma}{2}  \norm{\nabla \bff{u}}{\bb{L}^2}^2 + \frac{\kappa\mu}{2} \norm{\bff{u}}{\bb{L}^2}^2 + \frac{\kappa}{4} \norm{\bff{u}}{\bb{L}^4}^4 + \frac{\lambda}{2} \norm{\bff{e}\cdot \bff{u}}{\bb{L}^2}^2 \right)
	+
	\alpha \norm{\bff{H}}{\bb{L}^2}^2 = 0.
\end{align}
On the other hand, noting~\eqref{equ:llb eq2} we have
\begin{align}\label{equ:expand H}
	\norm{\bff{H}}{\bb{L}^2}^2
	&=
	\inpro{\sigma\Delta \bff{u}-\kappa\mu \bff{u}-\kappa |\bff{u}|^2 \bff{u}-\lambda \bff{e}(\bff{e}\cdot \bff{u})}{\sigma\Delta \bff{u}-\kappa\mu \bff{u}-\kappa |\bff{u}|^2 \bff{u}-\lambda \bff{e}(\bff{e}\cdot \bff{u})}
	\nonumber\\
	&=
	\sigma^2 \norm{\Delta \bff{u}}{\bb{L}^2}^2
	+
	2\kappa\mu\sigma \norm{\nabla\bff{u}}{\bb{L}^2}^2
	+
	2\kappa\sigma \norm{|\bff{u}| |\nabla \bff{u}|}{\bb{L}^2}^2
	+
	4\kappa\sigma \norm{\bff{u}\cdot\nabla \bff{u}}{\bb{L}^2}^2
	+
	\kappa^2 \mu^2 \norm{\bff{u}}{\bb{L}^2}^2
	+
	2\kappa^2 \mu \norm{\bff{u}}{\bb{L}^4}^4
	\nonumber\\
	&\quad
	+
	\kappa^2 \norm{\bff{u}}{\bb{L}^6}^6
	+
	2\kappa\mu\lambda \norm{\bff{e}\cdot\bff{u}}{\bb{L}^2}^2
	+
	2\lambda \sigma \norm{\bff{e}\cdot \nabla \bff{u}}{\bb{L}^2}^2
	+
	2\kappa\lambda \norm{|\bff{u}| (\bff{e}\cdot \bff{u})}{\bb{L}^2}^2
	+
	\lambda^2 \norm{\bff{e}\cdot \bff{u}}{\bb{L}^2}^2,
\end{align}
where we expanded the inner product and used integration by parts and \eqref{equ:nab uuv} as necessary in the last step. Substituting \eqref{equ:expand H} into \eqref{equ:ddt energy}, noting the definition of $\mathcal{E}$ in \eqref{equ:energy}, and discarding some non-negative terms on the left-hand side, we obtain
\begin{align*}
	\ddt \mathcal{E}(\bff{u}) + 2 \alpha\kappa\mu \mathcal{E}(\bff{u}) \leq 0,
\end{align*}
which implies \eqref{equ:ut energy exp}. 

Letting $t\to\infty$ in \eqref{equ:ut L infty exp} and \eqref{equ:ut energy exp} shows the last statement in the proposition. As time advances, the magnetisation vectors will align with each other (approaching a constant vector field) due to $\norm{\nabla \bff{u}(t)}{\bb{L}^2}\to 0$ as $t\to\infty$.
\end{proof}

\begin{remark}\label{rem:domain}
For $d\leq 2$, the existence and uniqueness of global strong solution to \eqref{equ:llb a} also hold if $\mathscr{D}$ is a convex polyhedral domain. In this case, we still have the $\bb{H}^2$-elliptic regularity results \eqref{equ:D2L2} and \eqref{equ:v H2 elliptic}, thus Proposition~\ref{pro:ut Lp}, \ref{pro:est nab u L2}, and~\ref{pro:Delta u L2} still hold in this setting to give the necessary \emph{a priori} estimates for the existence of a global strong solution. Theorem~\ref{the:decay} also holds in this case. For $d=3$, the $\bb{H}^3$-elliptic regularity result is used in the proof of Proposition~\ref{pro:Delta u L2}, which may not hold in a convex polyhedral domain. Higher order regularity estimates such as \eqref{equ:smooth H4} also may not hold for a general convex domain.
\end{remark}

\section{A linear fully discrete scheme for the LLB equation with spin-torques}\label{sec:llb spin}

In this section, we propose a linear fully discrete finite element scheme for the LLB equation with spin-torques. We begin by discussing some preliminary results on finite element approximation in the next subsection.

\subsection{Finite element approximation}

Let $\mathscr{D}\subset \bb{R}^d$, $d=1,2, 3$, be a smooth or convex polygonal (or convex polyhedral) domain. Let $\mathcal{T}_h$ be a quasi-uniform triangulation of $\mathscr{D}$
into intervals (in 1D), triangles (in 2D), or tetrahedra (in 3D)
with maximal mesh-size $h$.
To discretise the LLB equation, we
introduce the conforming finite element space $\bb{V}_h \subset \bb{H}^1$ given by
\begin{equation}\label{equ:Vh}
	\bb{V}_h := \{\bff{\phi}_h \in \bff{C}(\overline{\mathscr{D}}; \bb{R}^3): \bff{\phi}|_K \in \cal{P}_r(K;\bb{R}^3), \; \forall K \in \cal{T}_h\},
\end{equation}
where $\cal{P}_r(K; \bb{R}^3)$ denotes the space of polynomials of degree $r$ on $K$ taking values in $\bb{R}^3$. The case $r=1$ (linear polynomials) or $r=2$ (quadratic polynomials) are most commonly used.

As a consequence of the Bramble--Hilbert lemma~\cite{BreSco08}, for $p\in [1,\infty]$, there exists a constant $C$ independent of $h$ such that for any $\bff{v} \in \bb{W}^{r+1,p}$ we have
\begin{align}\label{equ:fin approx}
	\inf_{\chi \in {\bb{V}}_h} \left\{ \norm{\bff{v} - \bff{\chi}}{\bb{L}^p} 
	+ 
	h \norm{\nabla (\bff{v}-\bff{\chi})}{\bb{L}^p} 
	\right\} 
	\leq 
	C h^{r+1} \norm{\bff{v}}{\bb{W}^{r+1,p}}.
\end{align}
Moreover, if the triangulation $\mathcal{T}_h$ is quasi-uniform, we have the following inverse estimate: for $s\in \{0,1\}$ and $1\leq q\leq p\leq \infty$,
\begin{align}\label{equ:inverse}
	\norm{\bff{\chi}}{\bb{W}^{s,p}} &\leq Ch^{-d \left(\frac{1}{q}-\frac{1}{p}\right)} \norm{\bff{\chi}}{\bb{W}^{s,q}}, \quad \forall \bff{\chi}\in \bb{V}_h.
\end{align}

In the analysis, we will use several projection and interpolation operators. The existence of such operators and the properties that they possess will be described below. Sufficient regularity conditions on the mesh will be assumed as needed so that the following stability and approximation properties hold.

Firstly, there exists an orthogonal projection operator $P_h: \bb{L}^2 \to \bb{V}_h$ such that
\begin{align}\label{equ:orth proj}
	\inpro{P_h \bff{v}-\bff{v}}{\bff{\chi}}=0,
	\quad
	\forall \bff{\chi}\in \bb{V}_h,
\end{align}
with the following stability and approximation properties~\cite{BreSco08, CroTho87, DouDupWah74}: for any $p\in [1,\infty]$ and $s\in\{0,1\}$, there exists a constant $C$ independent of $h$ such that
\begin{align}
	\label{equ:proj H1 stab}
	\norm{P_h \bff{v}}{\bb{W}^{s,p}}
	&\leq
	C \norm{\bff{v}}{\bb{W}^{s,p}}, \quad \forall \bff{v}\in \bb{W}^{s,p},
	\\
	\label{equ:proj approx}
	\norm{\bff{v}- P_h\bff{v}}{\bb{L}^p}
	+
	h \norm{\nabla( \bff{v}-P_h\bff{v})}{\bb{L}^p}
	&\leq
	Ch^{r+1} \norm{\bff{v}}{\bb{W}^{r+1,p}}, \quad \forall \bff{v}\in \bb{W}^{r+1,p}.
\end{align}
Furthermore, there exists a nodal (Lagrange) interpolation operator $\mathcal{I}_h:\mathcal{C}^0 \to \bb{V}_h$ with these properties (\cite[Section~1.5]{ErnGue04} and~\cite{Zep23}): for any $p\in [1,\infty]$, there exists a constant $C$ independent of $h$ such that
\begin{align}
	\label{equ:interp stab Linfty}
	\norm{\mathcal{I}_h \bff{v}}{\bb{L}^\infty}
	&\leq
	C\norm{\bff{v}}{\bb{L}^\infty}
	\\
	\label{equ:interp stab W1p}
	\norm{\mathcal{I}_h \bff{v}}{\bb{W}^{1,p}}
	&\leq
	C\norm{\bff{v}}{\bb{W}^{1,p}}
	\\
	\label{equ:interp approx}
	\norm{\bff{v}- \mathcal{I}_h \bff{v}}{\bb{L}^p}
	+
	h \norm{\nabla \left(\bff{v}-\mathcal{I}_h \bff{v}\right)}{\bb{L}^p}
	&\leq
	Ch^{r+1} \norm{\bff{v}}{\bb{W}^{r+1,p}}.
\end{align}
Next, we introduce the discrete Laplacian operator $\Delta_h: \bb{V}_h \to \bb{V}_h$ defined by
\begin{align}\label{equ:disc laplacian}
	\inpro{\Delta_h \bff{v}_h}{\bff{\chi}}
	=
	- \inpro{\nabla \bff{v}_h}{\nabla \bff{\chi}},
	\quad 
	\forall \bff{v}_h, \bff{\chi} \in \bb{V}_h,
\end{align}
as well as the Ritz projection $R_h: \bb{H}^1 \to \bb{V}_h$ defined by
\begin{align}\label{equ:Ritz}
	\inpro{\nabla R_h \bff{v}- \nabla \bff{v}}{\nabla \bff{\chi}}=0,
	\quad \text{such that} \quad \inpro{R_h \bff{v}-\bff{v}}{\bff{1}}=0,
	\quad
	\forall \bff{\chi}\in \bb{V}_h.
\end{align}
If $\bff{v}\in \bb{H}^2_{\bff{n}}$, then we have $\Delta_h R_h\bff{v}=P_h\Delta \bff{v}$.
For any $\bff{v}\in \bb{H}^2$, let $\bff{\omega}:= \bff{v}-R_h\bff{v}$. The approximation property for the Ritz projection is assumed to hold \cite{LeyLi21, LinThoWah91, RanSco82}, namely for $s=0$ or $1$ and $p\in (1,\infty)$:
\begin{align}\label{equ:Ritz ineq}
	\norm{\bff{\omega}}{\bb{W}^{s,p}} + \norm{\partial_t \bff{\omega}}{\bb{W}^{s,p}}
	&\leq
	C h^{r+1-s} \norm{\bff{v}}{\bb{W}^{r+1, p}},
	\\
	\label{equ:Ritz ineq L infty}
	\norm{\bff{\omega}}{\bb{L}^\infty} 
	&\leq 
	Ch^{r+1} \abs{\ln h} \norm{\bff{v}}{\bb{W}^{r+1,\infty}}.
\end{align}
We also have the $\bb{W}^{1,\infty}$ stability property of the operator $R_h$, namely~\cite{DemLeySchWah12, Li22, Sco76}:
\begin{align}
	\label{equ:Ritz stab u infty}
	\norm{R_h \bff{v}}{\bb{W}^{1,\infty}}
	&\leq 
	C \norm{\bff{v}}{\bb{W}^{1,\infty}}.
\end{align}
A condition which would guarantee all the stability and approximation properties listed above to hold is global quasi-uniformity of the triangulation (which we assumed for $\mathcal{T}_h$). However, they are also known to hold for mesh which is less regular (only locally quasi-uniform or mildly graded). We will not discuss these further, but instead refer interested readers to the references cited above.

Finally, the following inequalities hold: If $\mathscr{D}$ is a convex polygonal or polyhedral domain with globally quasi-uniform triangulation, then there exists a constant $C$ independent of $h$ such that for any $\bff{v}_h\in \bb{V}_h$,
\begin{align}
	\label{equ:disc lapl L infty}
	\norm{\bff{v}_h}{\bb{L}^\infty}
	&\leq
	C \norm{\bff{v}_h}{\bb{L}^2}^{1-\frac{d}{4}} \left(\norm{\bff{v}_h}{\bb{L}^2}^\frac{d}{4} + \norm{\Delta_h \bff{v}_h}{\bb{L}^2}^\frac{d}{4} \right),
	\\
	\label{equ:inverse disc lapl}
	\norm{\Delta_h \bff{v}_h}{\bb{L}^2}
	&\leq
	Ch^{-2} \norm{\bff{v}_h}{\bb{L}^2}.
\end{align}
Inequality~\eqref{equ:disc lapl L infty} is shown in~\cite[Appendix A]{GuiLiWan22}), while \eqref{equ:inverse disc lapl} is shown in~\cite{BarPro06}.

\subsection{A linear fully discrete scheme}

Before introducing a linear finite element scheme for solving~\eqref{equ:llb a}, we need to define several multilinear forms that will facilitate the subsequent analysis.

\begin{definition}[multilinear forms]\label{def:forms}
	Given~$\bff{\nu}\in L^\infty(\bb{L}^\infty)$, and~$\bff{\phi},\bff{\xi}\in \bb{H}^1\cap \bb{L}^\infty$, we define the maps
	\begin{align*}
		&\mathcal{A}_1 (\cdot\,,\,\cdot): \bb{H}^1\times \bb{H}^1 \to \bb{R},
		\\ 
		&\mathcal{B}(\bff{\phi},\bff{\xi};\,\cdot\, , \,\cdot): \bb{H}^1\times \bb{H}^1 \to \bb{R},
		\\
		&\mathcal{C}(\bff{\phi};\,\cdot\, ,\, \cdot): \bb{H}^1\times \bb{H}^1 \to \bb{R},
		\\
		&\mathcal{D}(\cdot\, , \,\cdot): \bb{H}^1 \times \bb{H}^1 \to \bb{R},
	\end{align*}
	by
	\begin{align*}
		\mathcal{A}_1 (\bff{v},\bff{w}) 
		&:= 
		\alpha \inpro{\nabla \bff{v}}{\nabla \bff{w}} + \alpha \inpro{\bff{v}}{\bff{w}}
		+
		\alpha\inpro{\bff{e}(\bff{e}\cdot \bff{v})}{\bff{w}},
		\\
		\mathcal{B}(\bff{\phi},\bff{\xi};\bff{v},\bff{w}) &:= 
		\alpha \inpro{(\bff{\phi}\cdot\bff{\xi})\bff{v}}{\bff{w}},
		\\
		\mathcal{C}(\bff{\phi};\bff{v},\bff{w})
		&:= -\inpro{\bff{\phi}\times \nabla \bff{v}}{\nabla \bff{w}}
		- \inpro{\bff{v}\times \bff{e}(\bff{e}\cdot \bff{\phi})}{\bff{w}}
		-\beta_2 \inpro{\bff{v}\times (\bff{\nu}\cdot \nabla)\bff{\phi}}{\bff{w}}, 
		\\
		\mathcal{D}(\bff{v},\bff{w})
		&:= 
		\beta_1 \inpro{(\bff{\nu}\cdot \nabla)\bff{v}}{\bff{w}}.
	\end{align*}
	If either $\bff{\nu}\cdot \bff{n}=0$ on $\partial \mathscr{D}$ or $\bff{\nu}=\bff{0}$ on $\partial \mathscr{D}$, then by~\eqref{equ:div ab} and the divergence theorem, we can write:
	\begin{align}
		\label{equ:C div}
		\mathcal{C}(\bff{\phi};\bff{v},\bff{w})
		&= 
		-\inpro{\bff{\phi}\times \nabla \bff{v}}{\nabla \bff{w}}
		- \inpro{\bff{v}\times \bff{e}(\bff{e}\cdot \bff{\phi})}{\bff{w}}
		-
		\beta_2 \inpro{\bff{\phi}\otimes \bff{\nu}}{\nabla(\bff{v}\times \bff{w})}
		\nonumber\\
		&\quad
		+
		\beta_2 \inpro{\bff{v}\times (\nabla \cdot \bff{\nu})\bff{\phi}}{\bff{w}},
		\\
		\label{equ:D div}
		\mathcal{D}(\bff{v},\bff{w})
		&= 
		-
		\beta_1 \inpro{\bff{v}\otimes \bff{\nu}}{\nabla \bff{w}}
		-
		\beta_1 \inpro{(\nabla \cdot \bff{\nu}) \bff{v}}{\bff{w}}.
	\end{align}
	Furthermore, let
	\begin{align}\label{equ:bilinear A}
		\mathcal{A}(\bff{\phi}; \bff{v},\bff{w})
		:=
		\mathcal{A}_1(\bff{v},\bff{w})
		+
		\mathcal{B}(\bff{\phi},\bff{\phi};\bff{v},\bff{w})
		+
		\mathcal{C}(\bff{\phi}; \bff{v},\bff{w}).
	\end{align}
\end{definition}

Some properties of the multilinear forms defined above are gathered in the following lemma.

\begin{lemma}
	Let $\bff{\phi}, \bff{\xi} \in \bb{H}^1\cap \bb{L}^\infty$, and let the maps $\mathcal{A}_1,\mathcal{B},\mathcal{C},\mathcal{D}$, and $\mathcal{A}$ be as defined in Definition~\ref{def:forms}.
	\begin{enumerate}[(i)]
		\item \label{item:bdd} There exists a constant $C>0$ such that for all $\bff{v},\bff{w}\in \bb{H}^1$,
		\begin{align}
			\label{equ:A delta bdd}
			\left|\mathcal{A}_1 (\bff{v},\bff{w}) \right|
			&\leq
			3\alpha \norm{\bff{v}}{\bb{H}^1} \norm{\bff{w}}{\bb{H}^1},
			\\
			\label{equ:B bdd}
			\left|\mathcal{B} (\bff{\phi},\bff{\xi}; \bff{v},\bff{w}) \right|
			&\leq
			\alpha \norm{\bff{\phi}}{\bb{L}^\infty} \norm{\bff{\xi}}{\bb{L}^\infty} \norm{\bff{v}}{\bb{L}^2} \norm{\bff{w}}{\bb{L}^2},
			\\
			\label{equ:C bdd}
			\left|\mathcal{C} (\bff{\phi}; \bff{v},\bff{w}) \right|
			&\leq
			C \left( \norm{\bff{\phi}}{\bb{H}^1} + \norm{\bff{\phi}}{\bb{L}^\infty}\right) \norm{\bff{v}}{\bb{H}^1} \norm{\bff{w}}{\bb{H}^1}.
		\end{align}
		\item \label{item:2} For all $\bff{v}\in \bb{H}^1$, $\mathcal{C}(\bff{\phi}; \bff{v},\bff{v})=0$.
		\item $\mathcal{A}(\bff{\phi};\, \cdot\, , \, \cdot)$ is bounded, i.e. there exists a constant $\beta_{\mathrm{b}}>0$ depending only on the coefficients of the equation, $\norm{\bff{\phi}}{\bb{L}^\infty}$, and~$\norm{\bff{\phi}}{\bb{H}^1}$, such that
		\begin{align}\label{equ:A bounded}
			\abs{\mathcal{A}(\bff{\phi};\bff{v},\bff{w})}
			\leq
			\beta_\mathrm{b} \norm{\bff{v}}{\bb{H}^1} \norm{\bff{w}}{\bb{H}^1},
			\quad
			\forall \bff{v},\bff{w}\in \bb{H}^1.
		\end{align}
		\item $\mathcal{A}(\bff{\phi};\, \cdot\, , \, \cdot)$ is coercive, i.e.
		\begin{align}\label{equ:A coercive}
			\mathcal{A}(\bff{\phi};\bff{v},\bff{v}) \geq \alpha \norm{\bff{v}}{\bb{H}^1}^2,
			\quad \forall \bff{v}\in \bb{H}^1.
		\end{align}
		Here, $\alpha$ is a numerical coefficient in \eqref{equ:llb eq1}.
		\item If additionally $\bff{\phi}\in \bb{W}^{1,4}$ and $\bff{w}\in \bb{H}^2_{\bff{n}}$, then there exists a constant $C>0$ such that
		\begin{align}\label{equ:C ineq W14}
			\left|\mathcal{C} (\bff{\phi}; \bff{v},\bff{w}) \right|
			&\leq
			C \norm{\bff{\phi}}{\bb{W}^{1,4}} \norm{\bff{v}}{\bb{L}^2} \norm{\bff{w}}{\bb{H}^2},
			\\
			\label{equ:C Delta w w}
			\left|\mathcal{C} (\bff{\phi}; \Delta \bff{w},\bff{w}) \right|
			&\leq
			C \norm{\bff{\phi}}{\bb{W}^{1,4}} \norm{\bff{w}}{\bb{W}^{1,4}} \norm{\Delta \bff{w}}{\bb{L}^2},
		\end{align}
	\end{enumerate}
\end{lemma}

\begin{proof}
	Inequalities~\eqref{equ:A delta bdd}, \eqref{equ:B bdd}, and \eqref{equ:C bdd} follow immediately by H\"older's inequality. Statement~\eqref{item:2} is obvious since $\bff{a}\times\bff{a}=\bff{0}$ and $(\bff{a}\times \bff{b})\cdot \bff{b}=0$ for any $\bff{a},\bff{b}\in \bb{R}^3$. Inequality~\eqref{equ:A bounded} follows from the estimates in part~\eqref{item:bdd} and the triangle inequality.
	Moreover, note that
	\begin{align*}
		\mathcal{A}(\bff{\phi}; \bff{v},\bff{v})
		&=
		\alpha \norm{\nabla \bff{v}}{\bb{L}^2}^2
		+
		\alpha \norm{\bff{v}}{\bb{L}^2}^2
		+
		\alpha \norm{\bff{e}\cdot \bff{v}}{\bb{L}^2}^2
		+
		\alpha \norm{|\bff{\phi}| |\bff{v}|}{\bb{L}^2}^2
		\geq
		\alpha \norm{\bff{v}}{\bb{H}^1}^2,
	\end{align*}
	showing~\eqref{equ:A coercive}.
	
	Next, we show~\eqref{equ:C ineq W14} and~\eqref{equ:C Delta w w}. Firstly, integrating by parts the first term in the definition of $\mathcal{C}$, and noting the assumptions, we have
	\begin{align}\label{equ:int parts C}
		\mathcal{C} (\bff{\phi}; \bff{v},\bff{w})
		&=
		\inpro{\bff{\phi}\times \bff{v}}{\Delta \bff{w}}
		+
		\inpro{\nabla \bff{\phi}\times \bff{v}}{\nabla \bff{w}}
		+\inpro{\bff{v}\times \bff{e}(\bff{e}\cdot \bff{\phi})}{\bff{w}}
		-\beta_2 \inpro{\bff{v}\times (\bff{\nu}\cdot \nabla)\bff{\phi}}{\bff{w}}.
	\end{align}
	By H\"older's inequality,
	\begin{align*}
		\abs{\mathcal{C} (\bff{\phi}; \bff{v},\bff{w})}
		&\leq
		\norm{\bff{\phi}}{\bb{L}^\infty} \norm{\bff{v}}{\bb{L}^2} \norm{\Delta \bff{w}}{\bb{L}^2}
		+
		\norm{\nabla \bff{\phi}}{\bb{L}^4} \norm{\bff{v}}{\bb{L}^2} \norm{\nabla\bff{w}}{\bb{L}^4}
		+
		\norm{\bff{v}}{\bb{L}^2} \norm{\bff{\phi}}{\bb{L}^4} \norm{\bff{w}}{\bb{L}^4}
		\\
		&\quad
		+
		\beta_2 \nu_\infty \norm{\bff{v}}{\bb{L}^2} \norm{\nabla\bff{\phi}}{\bb{L}^4} \norm{\bff{w}}{\bb{L}^4}
		\\
		&\leq
		C\norm{\bff{\phi}}{\bb{W}^{1,4}} \norm{\bff{v}}{\bb{L}^2} \norm{\bff{w}}{\bb{H}^2},
	\end{align*}
	where in the last step we used Sobolev embeddings $\bb{H}^2\hookrightarrow \bb{W}^{1,4} \hookrightarrow \bb{L}^\infty$, thus showing~\eqref{equ:C ineq W14}.
	Similarly, inequality~\eqref{equ:C Delta w w} follows from~\eqref{equ:int parts C} with $\bff{v}=\Delta \bff{w}$ and H\"older's inequality, noting that the first term on the right-hand side of~\eqref{equ:int parts C} now vanishes.
\end{proof}

In terms of the forms introduced in Definition~\ref{def:forms}, the weak formulation~\eqref{equ:weak special} of the problem can be written as
\begin{align}\label{equ:weak form AB}
	\inpro{\partial_t \bff{u}}{\bff{\chi}} 
	+
	\mathcal{A}(\bff{u};\bff{u},\bff{\chi})
	=
	\mathcal{D}(\bff{u},\bff{\chi}), \quad
	\forall \bff{\chi} \in \bb{H}^1.
\end{align}

Let $k$ be the time step and $\bb{V}_h$ be the finite element space~\eqref{equ:Vh}. Let $\bff{u}_h^n$ be the approximation in $\bb{V}_h$ of $\bff{u}(t)$ at time $t=t_n=nk\in [0,T]$, where $n=0,1,2,\ldots, \lfloor T/k \rfloor$. We denote $\bff{u}^n:= \bff{u}(t_n)$. For any discrete function $\bff{v}$, define for $n\in \bb{N}$,
\begin{align*}
	\mathrm{d}_t \bff{v}^{n} 
	:=
	\frac{\bff{v}^{n}-\bff{v}^{n-1}}{k}.
\end{align*}
We now describe a linear fully discrete scheme to solve~\eqref{equ:weak form AB}. We start with $\bff{u}_h^0= P_h \bff{u}_0 \in \bb{V}_h$. For $t_n\in [0,T]$, $n\in \bb{N}$, given $\bff{u}_h^{n-1}\in \bb{V}_h$, define $\bff{u}_h^n$ by
\begin{align}\label{equ:scheme spin}
	\inpro{\mathrm{d}_t \bff{u}_h^n}{\bff{\chi}}
	+
	\mathcal{A} (\bff{u}_h^{n-1}; \bff{u}_h^n, \bff{\chi})
	=
	\mathcal{D}(\bff{u}_h^{n-1}, \bff{\chi}),
	\quad 
	\forall \bff{\chi}\in\bb{V}_h.
\end{align}

To facilitate the proof of the error analysis, we split the approximation error as:
\begin{align}\label{equ:split theta eta}
	\bff{u}_h^n- \bff{u}^n
	= 
	\left(\bff{u}_h^n- \Pi_h \bff{u}^n\right) 
	+ 
	\left(\Pi_h \bff{u}^n- \bff{u}^n\right) =: \bff{\theta}^n+ \bff{\eta}^n,
\end{align}
where for any $t_n\in [0,T]$, $\Pi_h \bff{u}^n=: \Pi_h \bff{u}(t_n)$ is the elliptic projection of the solution $\bff{u}(t_n)$ defined by
\begin{align}\label{equ:elliptic proj}
	\mathcal{A}(\bff{u}(t_n); \bff{\eta}^n, \bff{\chi})=0, \quad \forall \bff{\chi}\in\bb{V}_h.
\end{align}
Note that scheme~\eqref{equ:scheme spin} and the elliptic projection are well-defined by the Lax--Milgram lemma. 
We study some properties of this projection in the following section, while the error analysis is performed in Section~\ref{subsec:error analysis}. For a technical reason (see Lemma~\ref{lem:stab elliptic}), we assume that in~\eqref{equ:Vh} the polynomial degree $r$ in $\bb{V}_h$ is:
\begin{equation}\label{equ:deg r}
	\begin{cases}
		r\geq 1, &\text{if } d\in \{1,2\},
		\\
		r\geq 2, &\text{if } d=3.
	\end{cases}
\end{equation}

In this section, to attain an optimal order of convergence for the underlying approximation \eqref{equ:scheme spin}, we assume that the problem~\eqref{equ:llb a} admits a sufficiently regular solution~$\bff{u}$ which satisfies
\begin{equation}\label{equ:ass 1}
	\norm{\bff{u}}{L^\infty(\bb{H}^{r+1})}
	+ \norm{\bff{u}}{L^\infty(\bb{W}^{1,\infty})}
	+ \norm{\partial_t \bff{u}}{L^\infty(\bb{H}^{r+1})}
	+ \norm{\partial_t^2 \bff{u}}{L^\infty(\bb{L}^2)}
	\leq K_r,
\end{equation}
for some positive constant $K_r$ depending on the initial data, $r$, and $T$. Here, $r$ is the degree of piecewise polynomials in the finite element space~\eqref{equ:Vh}.

Note that the existence of a strong solution satisfying~\eqref{equ:ass 1} for $r=1$ is given by Theorem~\ref{the:main existence}, at least for a smooth domain (so the assumption~\eqref{equ:ass 1} is reasonable). The existence of a more regular solution can be shown in a similar manner.

\subsection{Properties of elliptic projection}\label{subsec:elliptic}

Relevant approximation and stability properties for the elliptic projection defined by~\eqref{equ:elliptic proj} are derived in this section. We begin with the following lemma which will be needed to show some approximation properties of $\Pi_h$.

\begin{lemma}\label{lem:dual exist}
	Let $\mathscr{D}$ be a smooth or a convex polyhedral domain. Assume that the exact solution $\bff{u}$ of~\eqref{equ:llb a} belongs to $L^\infty(\bb{H}^2_{\bff{n}})$. For any $\bff{\varphi}\in \bb{L}^2$ and for each $t\in [0,T]$, there exists $\bff{\psi}(t)\in \bb{H}^2_{\bff{n}}$ such that
	\begin{align}\label{equ:A ut zeta}
		\mathcal{A}(\bff{u}(t); \bff{\zeta}, \bff{\psi}(t))
		=
		\inpro{\bff{\varphi}}{\bff{\zeta}}, \quad
		\forall \bff{\zeta}\in \bb{H}^1.
	\end{align}
	Moreover,
	\begin{align}\label{equ:psi H2}
		\norm{\bff{\psi}(t)}{\bb{H}^2} \leq C\norm{\bff{\varphi}}{\bb{L}^2},
	\end{align}
	where the constant $C$ depends on $\mathscr{D}$, $T$ and $\norm{\bff{u}}{L^\infty(\bb{H}^2)}$.
\end{lemma}

\begin{proof}
	We use the Faedo--Galerkin method. Let $\{\bff{e}_i\}_{i=1}^\infty$ be an orthonormal basis of $\bb{L}^2$ consisting of eigenfunctions for $-\Delta$ such that
	\begin{align*}
		-\Delta \bff{e}_i = \lambda_i \bff{e}_i \text{ in } \mathscr{D} \quad
		\text{and} \quad \frac{\partial \bff{e}_i}{\partial \bff{n}}=\bff{0} \text{ on } \partial \mathscr{D},
	\end{align*}
	where $\lambda_i \geq 0$ are the eigenvalues of $-\Delta$ associated with $\bff{e}_i$.
	Let $\bb{W}_n:= \text{span}\{\bff{e}_1,\ldots,\bff{e}_n\}$. 
	
	Given $\bff{\varphi}\in \bb{L}^2$, for each $t\in [0,T]$, we approximate the solution to~\eqref{equ:A ut zeta} by $\bff{\psi}_n(t)\in \bb{W}_n$ satisfying 
	\begin{align}\label{equ:A ut zeta psi n}
		\mathcal{A}(\bff{u}(t); \bff{\zeta}, \bff{\psi}_n(t))
		=
		\inpro{\bff{\varphi}}{\bff{\zeta}}, \quad
		\forall \bff{\zeta}\in \bb{W}_n.
	\end{align}
	We need the following uniform bound:
	\begin{align}\label{equ:psi n phi H2}
		\norm{\bff{\psi}_n(t)}{\bb{H}^2} \leq C \norm{\bff{\varphi}}{\bb{L}^2}.
	\end{align}
	Once this bound is established, we can apply the standard compactness argument: For each $t\in [0,T]$, the Banach--Alaoglu theorem and a compactness argument imply the existence of a subsequence, which is still denoted by $\{\bff{\psi}_n(t)\}$, such that
	\begin{align*}
		&\bff{\psi}_n(t) \rightharpoonup \bff{\psi}(t) \quad \text{weakly in } \bb{H}^2_{\bff{n}},
		\\
		&\bff{\psi}_n(t) \to \bff{\psi}(t) \quad \text{strongly in } \bb{W}^{1,4}.
	\end{align*}
	A standard argument then shows $\bff{\psi}(t)$ satisfies~\eqref{equ:A ut zeta}, while \eqref{equ:psi H2} follows from~\eqref{equ:psi n phi H2}.
	
	It remains to prove~\eqref{equ:psi n phi H2}. Setting $\bff{\zeta}= \bff{\psi}_n$ in~\eqref{equ:A ut zeta psi n} and using the coercivity of $\mathcal{A}$, we have
	\begin{align*}
		\mu \norm{\bff{\psi}_n}{\bb{H}^1}^2
		\leq
		\mathcal{A}(\bff{u}; \bff{\psi}_n, \bff{\psi}_n)
		=
		\inpro{\bff{\varphi}}{\bff{\psi}_n}
		\leq
		\norm{\bff{\varphi}}{\bb{L}^2} \norm{\bff{\psi}_n}{\bb{L}^2},
	\end{align*}
	which implies
	\begin{align}\label{equ:psi n H1}
		\norm{\bff{\psi}_n(t)}{\bb{H}^1} \leq C \norm{\bff{\varphi}}{\bb{L}^2}.
	\end{align}
	Next, taking $\bff{\zeta}= -\Delta \bff{\psi}_n$, integrating by parts as necessary, and applying~\eqref{equ:B bdd} and~\eqref{equ:C Delta w w} we have
	\begin{align*}
		\alpha \norm{\Delta \bff{\psi}_n}{\bb{L}^2}^2
		+
		\delta \norm{\nabla \bff{\psi}_n}{\bb{L}^2}^2
		&=
		\mathcal{B}(\bff{u},\bff{u}; \Delta \bff{\psi}_n, \bff{\psi}_n)
		+
		\mathcal{C}(\bff{u}; \Delta \bff{\psi}_n, \bff{\psi}_n)
		-
		\inpro{\bff{\varphi}}{\Delta \bff{\psi}_n}
		\\
		&\leq
		C\norm{\bff{u}}{\bb{L}^\infty}^2 \norm{\Delta \bff{\psi}_n}{\bb{L}^2} \norm{\bff{\psi}_n}{\bb{L}^2}
		+
		C \norm{\bff{u}}{\bb{W}^{1,4}} \norm{\bff{\psi}_n}{\bb{W}^{1,4}} \norm{\Delta \bff{\psi}_n}{\bb{L}^2}
		+
		\norm{\bff{\varphi}}{\bb{L}^2} \norm{\Delta \bff{\psi}_n}{\bb{L}^2}
		\\
		&\leq
		C \norm{\bff{\psi}_n}{\bb{H}^1}^2
		+
		\frac{\alpha}{2} \norm{\Delta \bff{\psi}_n}{\bb{L}^2}^2
		+
		\norm{\bff{\varphi}}{\bb{L}^2}^2,
	\end{align*}
	where in the last step we also used the Gagliardo--Nirenberg inequality, Sobolev embedding, and Young's inequality. This, together with~\eqref{equ:psi n H1} and the elliptic regularity result, implies~\eqref{equ:psi n phi H2}. The proof is now complete.
\end{proof}

With the above lemma, we can show the following approximation properties of $\Pi_h$.

\begin{proposition}\label{pro:eta t}
Let $\Pi_h \bff{u}$ be the elliptic projection of $\bff{u}$ defined by~\eqref{equ:elliptic proj}, $\bff{u}$ be the solution of~\eqref{equ:llb a} which satisfies~\eqref{equ:ass 1}, and $\bff{\eta}(t)=\Pi_h \bff{u}(t)- \bff{u}(t)$ be as defined in~\eqref{equ:split theta eta}. Then for any $t\in [0,T]$,
\begin{align}
	\label{equ:eta t}
	\norm{\bff{\eta}(t)}{\bb{L}^2} + h \norm{\nabla \bff{\eta}(t)}{\bb{L}^2}
	&\leq
	Ch^{r+1},
	\\
	\label{equ:dt eta t}
	\norm{\partial_t \bff{\eta}(t)}{\bb{L}^2} + h \norm{\nabla \partial_t \bff{\eta}(t)}{\bb{L}^2}
	&\leq
	Ch^{r+1},
\end{align}
where $r$ is the degree of polynomials in the finite element space $\bb{V}_h$.
The constant $C$ depends on $T$, $\mathscr{D}$, $K_r$, and the coefficients of the equation, where $K_r$ was defined in~\eqref{equ:ass 1}.
\end{proposition}

\begin{proof}
For all $\bff{\chi}\in \bb{V}_h$, by the boundedness and coercivity of $\mathcal{A}$ in \eqref{equ:A bounded} and \eqref{equ:A coercive}, as well as the definition of $\bff{\eta}$, we have
\begin{align*}
	\alpha \norm{\bff{\eta}(t)}{\bb{H}^1}^2
	&\leq
	\mathcal{A}(\bff{u}(t); \bff{\eta}(t), \bff{\eta}(t))
	\\
	&=
	\mathcal{A}(\bff{u}(t); \bff{\eta}(t), \Pi_h\bff{u}(t)- \bff{\chi})
	-
	\mathcal{A}(\bff{u}(t); \bff{\eta}(t), \bff{u}(t)-\bff{\chi})
	\\
	&\leq
	\beta_\mathrm{b} \norm{\bff{\eta}(t)}{\bb{H}^1} \norm{\bff{u}(t)-\bff{\chi}}{\bb{H}^1},
\end{align*}
since the first term in the second step is zero. Therefore, by~\eqref{equ:fin approx},
\begin{align}\label{equ:eta H1}
	\norm{\bff{\eta}(t)}{\bb{H}^1}
	\leq
	(\beta_\mathrm{b}/\alpha) \inf_{\bff{\chi}\in \bb{V}_h} \norm{\bff{u}(t)-\bff{\chi}}{\bb{H}^1}
	\leq
	Ch^r \norm{\bff{u}}{L^\infty(\bb{H}^{r+1})}.
\end{align}
To show the $\bb{L}^2$-estimate, we use duality argument. For each $\bff{u}(t)\in \bb{H}^2$, let $\bff{\psi}(t)\in \bb{H}^2_{\bff{n}}$ satisfy
\begin{align}\label{equ:A u zeta}
	\mathcal{A}(\bff{u}(t); \bff{\zeta}, \bff{\psi}(t)) = \inpro{\bff{\eta}(t)}{\bff{\zeta}},
	\quad \forall \bff{\zeta}\in \bb{H}^1.
\end{align}
For any $t\in [0,T]$, such $\bff{\psi}(t)$ exists by Lemma~\ref{lem:dual exist} under the assumed conditions on $\mathscr{D}$. Moreover,
\begin{align*}
	\norm{\bff{\psi}(t)}{\bb{H}^2} \leq
	C\norm{\bff{\eta}(t)}{\bb{L}^2}.
\end{align*}
Therefore, taking $\bff{\zeta}= \bff{\eta}(t)$ in~\eqref{equ:A u zeta} and noting~\eqref{equ:elliptic proj}, we have for all $\bff{\chi}\in \bb{V}_h$,
\begin{align*}
	\norm{\bff{\eta}(t)}{\bb{L}^2}^2
	&=
	\mathcal{A}(\bff{u}(t); \bff{\eta}(t), \bff{\psi}(t))
	=
	\mathcal{A}(\bff{u}(t); \bff{\eta}(t), \bff{\psi}(t)-\bff{\chi})
	\\
	&\leq
	\beta_\mathrm{b} \norm{\bff{\eta}(t)}{\bb{H}^1} \inf_{\bff{\chi}\in \bb{V}_h} \norm{\bff{\psi}(t)-\bff{\chi}}{\bb{H}^1}
	\leq
	Ch^{r+1} \norm{\bff{u}(t)}{\bb{H}^{r+1}} \norm{\bff{\eta}(t)}{\bb{L}^2},
\end{align*}
where in the last step we used~\eqref{equ:eta H1} and~\eqref{equ:fin approx}. This and~\eqref{equ:eta H1} then implies inequality~\eqref{equ:eta t}.

Next, we show~\eqref{equ:dt eta t}. For ease of presentation, we omit the dependence of the functions on $t$. By the coercivity of $\mathcal{A}$ and the definition of $\bff{\eta}$,
\begin{align}\label{equ:leq A dt eta}
	\alpha \norm{\partial_t \bff{\eta}(t)}{\bb{H}^1}^2
	\leq
	\mathcal{A}(\bff{u}; \partial_t \bff{\eta}, \partial_t \bff{\eta})
	=
	\mathcal{A}\left(\bff{u}; \partial_t \bff{\eta}, \mathcal{I}_h(\partial_t \bff{\eta})\right) 
	+
	\mathcal{A}\left(\bff{u}; \partial_t \bff{\eta}, \mathcal{I}_h(\partial_t \bff{u}) -\partial_t \bff{u}\right).
\end{align}
We will estimate each term on the last line. To this end, noting~\eqref{equ:bilinear A} and differentiating~\eqref{equ:elliptic proj} with respect to $t$, we have for all $\bff{\chi}\in \bb{V}_h$,
\begin{align}\label{equ:A dt eta}
	\mathcal{A} (\bff{u}; \partial_t \bff{\eta},\bff{\chi})
	+
	2 \mathcal{B}(\bff{u},\partial_t \bff{u}; \bff{\eta}, \bff{\chi})
	+
	\mathcal{C}(\partial_t \bff{u}; \bff{\eta}, \bff{\chi})
	= 0.
\end{align}
Thus, for the first term on the right-hand side of~\eqref{equ:leq A dt eta}, by the boundedness of $\mathcal{B}$ and $\mathcal{C}$ we have
\begin{align}\label{equ:A dt leq 1}
	\big| \mathcal{A}(\bff{u}; \partial_t \bff{\eta}, \mathcal{I}_h(\partial_t \bff{\eta})) \big| 
	&\leq
	\big| 2 \mathcal{B}(\bff{u},\partial_t \bff{u}; \bff{\eta}, \mathcal{I}_h(\partial_t \bff{\eta})) \big|
	+
	\big| \mathcal{C}(\partial_t \bff{u}; \bff{\eta}, \mathcal{I}_h(\partial_t \bff{\eta})) \big|
	\nonumber\\
	&\leq
	C \norm{\bff{\eta}}{\bb{H}^1} \norm{\mathcal{I}_h (\partial_t \bff{\eta})}{\bb{H}^1}
	\leq
	Ch^r \norm{\partial_t \bff{\eta}}{\bb{H}^1},
\end{align}
where in the last step we also used~\eqref{equ:eta t} and \eqref{equ:interp stab W1p}. Here, $C$ depends on $K_r$ among other things (but is independent of $h$). For the second term on the right-hand side of~\eqref{equ:leq A dt eta}, by the boundedness of $\mathcal{A}$ and \eqref{equ:interp approx} we have
\begin{align}\label{equ:A dt leq 2}
	\big| \mathcal{A}(\bff{u}; \partial_t \bff{\eta}, \mathcal{I}_h(\partial_t \bff{u})- \partial_t \bff{u}) \big| 
	\leq
	C \norm{\partial_t \bff{\eta}}{\bb{H}^1} \norm{\mathcal{I}_h(\partial_t \bff{u})- \partial_t \bff{u}}{\bb{H}^1}
	\leq
	Ch^r \norm{\partial_t \bff{\eta}}{\bb{H}^1} \norm{\partial_t \bff{u}}{L^\infty(\bb{H}^{r+1})}.
\end{align}
The estimates~\eqref{equ:A dt leq 1} and \eqref{equ:A dt leq 2}, together with~\eqref{equ:leq A dt eta} imply
\begin{align}\label{equ:dt eta H1}
	\norm{\partial_t \bff{\eta}(t)}{\bb{H}^1}
	\leq
	Ch^r,
\end{align}
where $C$ depends on $K_r$, but is independent of $h$.
To estimate $\norm{\partial_t \bff{\eta}}{\bb{L}^2}$, we use duality argument as before. For each $\bff{u}(t) \in \bb{H}^2$, let $\bff{\psi}(t)\in \bb{H}^2_{\bff{n}}$ satisfy
\begin{align}\label{equ:A pa t zeta}
	\mathcal{A}(\bff{u}(t); \bff{\zeta}, \bff{\psi}(t)) = \inpro{\partial_t \bff{\eta}(t)}{\bff{\zeta}}, \quad \forall \bff{\zeta}\in \bb{H}^1. 
\end{align}
The existence of $\bff{\psi}(t)$ is conferred by Lemma~\ref{lem:dual exist}, and furthermore we have
\begin{align}\label{equ:psi dt eta}
	\norm{\bff{\psi}(t)}{\bb{H}^2} \leq C \norm{\partial_t \bff{\eta}(t)}{\bb{L}^2}.
\end{align}
Taking $\bff{\zeta}= \partial_t \bff{\eta}(t)$ in~\eqref{equ:A pa t zeta}, we have
\begin{align*}
	\mathcal{A}(\bff{u}(t); \partial_t \bff{\eta}(t), \bff{\psi}(t))
	=
	\norm{\partial_t \bff{\eta}(t)}{\bb{L}^2}^2.
\end{align*}
This equation and~\eqref{equ:A dt eta} yield for all $\bff{\chi}\in \bb{V}_h$,
\begin{align*}
	\norm{\partial_t \bff{\eta}}{\bb{L}^2}^2
	&=
	\mathcal{A} (\bff{u}; \partial_t \bff{\eta},\bff{\psi}-\bff{\chi})
	+
	2 \mathcal{B}(\bff{u},\partial_t \bff{u}; \bff{\eta}, \bff{\chi})
	+
	\mathcal{C}(\partial_t \bff{u}; \bff{\eta}, \bff{\chi})
	\\
	&=
	\mathcal{A} (\bff{u}; \partial_t \bff{\eta},\bff{\psi}-\bff{\chi})
	+
	2\mathcal{B} (\bff{u},\partial_t \bff{u}; \bff{\eta}, \bff{\psi}-\bff{\chi})
	-
	2 \mathcal{B}(\bff{u},\partial_t \bff{u}; \bff{\eta}, \bff{\psi})
	\\
	&\quad
	+
	\mathcal{C}(\partial_t \bff{u}; \bff{\eta},\bff{\psi}-\bff{\chi})
	-
	\mathcal{C}(\partial_t \bff{u}; \bff{\eta}, \bff{\psi}).
\end{align*}
We estimate the first four terms on the right-hand side above by using the boundedness of $\mathcal{A}$, $\mathcal{B}$, and $\mathcal{C}$ (cf.~\eqref{equ:A bounded}, \eqref{equ:B bdd}, and~\eqref{equ:C bdd}). The fourth term is bounded using~\eqref{equ:C bdd}, while the last term is bounded using~\eqref{equ:C ineq W14}. Thus, we obtain
\begin{align*}
	\norm{\partial_t \bff{\eta}}{\bb{L}^2}^2
	&\leq
	C \norm{\partial_t \bff{\eta}}{\bb{H}^1} \norm{\bff{\psi}-\bff{\chi}}{\bb{H}^1}
	+
	2 \norm{\bff{u}}{\bb{L}^\infty} \norm{\partial_t \bff{u}}{\bb{L}^\infty} \norm{\bff{\eta}}{\bb{L}^2} \norm{\bff{\psi}-\bff{\chi}}{\bb{L}^2}
	+
	2 \norm{\bff{u}}{\bb{L}^\infty} \norm{\partial_t \bff{u}}{\bb{L}^\infty} \norm{\bff{\eta}}{\bb{L}^2} \norm{\bff{\psi}}{\bb{L}^2}
	\\
	&\quad
	+
	C \norm{\partial_t \bff{u}}{\bb{H}^1} \norm{\bff{\eta}}{\bb{H}^1} \norm{\bff{\psi}-\bff{\chi}}{\bb{H}^1}
	+
	C \norm{\partial_t \bff{u}}{\bb{W}^{1,4}} \norm{\bff{\eta}}{\bb{L}^2} \norm{\bff{\psi}}{\bb{H}^2}.
\end{align*}
We now choose $\bff{\chi}= \mathcal{I}_h \bff{\psi}$. Successively using~\eqref{equ:dt eta H1}, \eqref{equ:interp approx}, and~\eqref{equ:eta t}, noting the assumption~\eqref{equ:ass 1} we have
\begin{align*}
	\norm{\partial_t \bff{\eta}(t)}{\bb{L}^2}^2
	\leq
	(Ch^{r+1}+Ch^{r+3}) \norm{\bff{\psi}(t)}{\bb{H}^2}
	\leq
	Ch^{r+1} \norm{\partial_t \bff{\eta}(t)}{\bb{L}^2},
\end{align*}
where in the last step we used~\eqref{equ:psi dt eta} and the fact that $h<1$. Here, $C$ depends on $K_r$ (but is independent of $h$). This implies
\begin{align*}
	\norm{\partial_t \bff{\eta}(t)}{\bb{L}^2}
	\leq
	Ch^{r+1}.
\end{align*}
This inequality and~\eqref{equ:dt eta H1} together implies~\eqref{equ:dt eta t}. The proof of this proposition will then be complete once we show the following regularity result in Lemma~\ref{lem:dual exist}.
\end{proof}

Next, we have the following result on the stability of the elliptic projection.

\begin{lemma}\label{lem:stab elliptic}
Let $\Pi_h\bff{u}$ be the elliptic projection of $\bff{u}$, where $\bff{u}$ is the solution of~\eqref{equ:llb a} satisfying the assumption~\eqref{equ:ass 1}. Suppose that the polynomial degree $r$ in $\bb{V}_h$ satisfies~\eqref{equ:deg r}. Then for all $t\in [0,T]$,
\begin{align*}
	\norm{\Pi_h \bff{u}(t)}{\bb{W}^{1,\infty}}
	&\leq
	C.
\end{align*}
The constant $C$ depends on $K_r$ (defined in~\eqref{equ:ass 1}) and $T$, but is independent of $h$.
\end{lemma}

\begin{proof}
Let $P_h$ be the $\bb{L}^2$ projection defined in~\eqref{equ:orth proj}. By the triangle inequality, the inverse estimate~\eqref{equ:inverse}, the stability and approximation properties \eqref{equ:interp stab W1p} and \eqref{equ:interp approx}, and~\eqref{equ:eta t}, we obtain
\begin{align*}
	\norm{\Pi_h \bff{u}(t)}{\bb{W}^{1,\infty}}
	&\leq
	\norm{\Pi_h \bff{u}(t)- \mathcal{I}_h(\bff{u}(t))}{\bb{W}^{1,\infty}} 
	+
	\norm{\mathcal{I}_h(\bff{u}(t))}{\bb{W}^{1,\infty}}
	\\
	&\leq
	Ch^{-d/2} \big(\norm{\bff{\eta}(t)}{\bb{H}^1} + \norm{\mathcal{I}_h(\bff{u}(t))- \bff{u}(t)}{\bb{H}^1} \big) 
	+
	C\norm{\bff{u}(t)}{\bb{W}^{1,\infty}} 
	\\
	&\leq
	Ch^{-d/2} \big(Ch^r \big) 
	+
	C\norm{\bff{u}}{L^\infty(\bb{W}^{1,\infty})} 
	\\
	&\leq
	Ch^{r-d/2} + C \norm{\bff{u}}{L^\infty(\bb{W}^{1,\infty})}, 
\end{align*}
where $C$ depends on $K_r$ (defined in~\eqref{equ:ass 1}), but is independent of $h$. Since $r\geq d/2$ by assumption~\eqref{equ:deg r} and $h<1$, we obtain the required estimate.
\end{proof}

\subsection{Error analysis} \label{subsec:error analysis}

We perform the error analysis for scheme~\eqref{equ:scheme spin} in this section. Recall that we split the error as a sum of two terms in~\eqref{equ:split theta eta}. We also remark that for $p\in [1,\infty]$
\begin{align}\label{equ:norm delta un Lp}
	\norm{\mathrm{d}_t \bff{u}^n}{\bb{L}^p}
	=
	\norm{\frac{1}{k} \int_{t_{n-1}}^{t_n} \partial_t \bff{u}(t)\, \dt}{\bb{L}^p}
	\leq 
	\norm{\partial_t \bff{u}(t)}{\bb{L}^p}.
\end{align}
First, we have the following stability estimate.

\begin{lemma}
Let $T>0$ be given and let $\bff{u}_h^n$ be defined by~\eqref{equ:scheme spin} with initial data $\bff{u}^0\in \bb{V}_h$. Then for any $n\in \{1,2,\ldots, \lfloor T/k \rfloor\}$,
\begin{align}\label{equ:stab L2 lin scheme}
	\norm{\bff{u}_h^n}{\bb{L}^2}^2
	+
	\sum_{j=1}^n \norm{\bff{u}_h^j- \bff{u}_h^{j-1}}{\bb{L}^2}^2
	+
	k\sum_{j=1}^n \norm{\nabla \bff{u}_h^j}{\bb{L}^2}^2
	\leq
	C \norm{\bff{u}^0}{\bb{L}^2}^2,
\end{align}
where the constant $C$ depends on $T$, but is independent of $n$ and $k$.
\end{lemma}

\begin{proof}
We set $\bff{\chi}=\bff{u}_h^n$ in~\eqref{equ:scheme spin} and note the vector identity~\eqref{equ:aab}.
The required estimate then follows from similar argument as in~\cite[Lemma~3.1]{LeSoeTra24}.
\end{proof}

The following estimates on the nonlinear terms are needed later.

\begin{lemma}\label{lem:ineq forms}
Let $\epsilon>0$ be given. Let $\mathcal{A}, \mathcal{B}, \mathcal{C}$, and $\mathcal{D}$ be as defined in Definition~\ref{def:forms}, and suppose that $\bff{u}$ is a solution of~\eqref{equ:llb a} that satisfies~\eqref{equ:ass 1}. Then there exists a constant $C$ such that for any $\bff{\chi}\in \bb{V}_h$,
\begin{align}
	\label{equ:B nonlinear est}
	\big| \mathcal{B}(\bff{u}_h^{n-1}, \bff{u}_h^{n-1}; \Pi_h \bff{u}^n, \bff{\chi}) - \mathcal{B}(\bff{u}^n, \bff{u}^n; \Pi_h \bff{u}^n, \bff{\chi}) \big| 
	&\leq
	C\left(1+ \norm{\bff{u}_h^{n-1}}{\bb{L}^4}^2 \right) \norm{\bff{\theta}^{n-1}}{\bb{L}^2}^2
	\nonumber\\
	&\quad
	+
	C\left(1+ \norm{\bff{u}_h^{n-1}}{\bb{L}^4}^2 \right) (h^{2(r+1)}+k^2)
	+
	\epsilon \norm{\bff{\chi}}{\bb{H}^1}^2,
	\\
	\label{equ:C nonlinear est}
	\big| \mathcal{C}(\bff{u}_h^{n-1}; \Pi_h \bff{u}^n, \bff{\chi}) - \mathcal{C}(\bff{u}^n; \Pi_h \bff{u}^n, \bff{\chi}) \big| 
	&\leq
	C\norm{\bff{\theta}^{n-1}}{\bb{L}^2}^2 + Ch^{2(r+1)} + Ck^2 
	+
	\epsilon \norm{\bff{\chi}}{\bb{H}^1}^2,
	\\
	\label{equ:D nonlinear est}
	\big| \mathcal{D}(\bff{u}_h^{n-1}, \bff{\chi}) - \mathcal{D}(\bff{u}^n, \bff{\chi}) \big| 
	&\leq
	C \norm{\bff{\theta}^{n-1}}{\bb{L}^2}^2 + Ch^{2(r+1)} + Ck^2+ \epsilon \norm{\bff{\chi}}{\bb{H}^1}^2. 
\end{align}
Consequently, 
\begin{align}\label{equ:A nonlinear}
	\big|\mathcal{A}(\bff{u}_h^{n-1}; \Pi_h \bff{u}^n, \bff{\chi})- \mathcal{A}(\bff{u}^n;\Pi_h \bff{u}^n,\bff{\chi}) \big|
	&\leq
	C\left(1+ \norm{\bff{u}_h^{n-1}}{\bb{L}^4}^2 \right) \norm{\bff{\theta}^{n-1}}{\bb{L}^2}^2
	\nonumber\\
	&\quad
	+
	C\left(1+ \norm{\bff{u}_h^{n-1}}{\bb{L}^4}^2 \right) (h^{2(r+1)}+k^2)
	+
	\epsilon \norm{\bff{\chi}}{\bb{H}^1}^2.
\end{align}
The constant $C$ depends on $\epsilon$, the coefficients of~\eqref{equ:llb a}, $T$, $K_r$, and $\mathscr{D}$ (but is independent of $n$, $h$, or $k$).
\end{lemma}

\begin{proof}
Firstly, note that we have
\begin{align}\label{equ:uhn1 min un}
	\bff{u}_h^{n-1}- \bff{u}^n = \bff{\theta}^{n-1} + \bff{\eta}^{n-1} - k\cdot \mathrm{d}_t \bff{u}^n.
\end{align}
Thus, by H\"older's inequality, \eqref{equ:norm delta un Lp}, \eqref{equ:eta t}, and the assumption~\eqref{equ:ass 1}, we have
\begin{align}\label{equ:un2 L43}
	\norm{|\bff{u}_h^{n-1}|^2 - |\bff{u}^n|^2}{\bb{L}^{4/3}}
	&=
	\norm{(\bff{u}_h^{n-1}+\bff{u}^n) \cdot (\bff{u}_h^{n-1}-\bff{u}^n)}{\bb{L}^{4/3}}
	\nonumber\\
	&\leq
	\norm{\bff{u}_h^{n-1}+\bff{u}^n}{\bb{L}^4} \norm{\bff{\theta}^{n-1} + \bff{\eta}^{n-1} - k\cdot \mathrm{d}_t \bff{u}^n}{\bb{L}^2} 
	\nonumber\\
	&\leq
	C\left(1+ \norm{\bff{u}_h^{n-1}}{\bb{L}^4} \right) \norm{\bff{\theta}^{n-1}}{\bb{L}^2}
	+
	C\left(1+ \norm{\bff{u}_h^{n-1}}{\bb{L}^4} \right) (h^{r+1}+k).
\end{align}
We can then estimate
\begin{align*}
	&\big| \mathcal{B}(\bff{u}_h^{n-1}, \bff{u}_h^{n-1}; \Pi_h \bff{u}^n, \bff{\chi}) - \mathcal{B}(\bff{u}^n, \bff{u}^n; \Pi_h \bff{u}^n, \bff{\chi}) \big|
	\\
	&=
	\big| \alpha \inpro{(|\bff{u}_h^{n-1}|^2- |\bff{u}^{n}|^2) \Pi_h \bff{u}^n}{\bff{\chi}} \big|
	\\
	&\leq
	\alpha \norm{|\bff{u}_h^{n-1}|^2 - |\bff{u}^n|^2}{\bb{L}^{4/3}} \norm{\Pi_h \bff{u}^n}{\bb{L}^\infty} \norm{\bff{\chi}}{\bb{L}^4}
	\\
	&\leq
	\Big[ C\left(1+ \norm{\bff{u}_h^{n-1}}{\bb{L}^4} \right) \norm{\bff{\theta}^{n-1}}{\bb{L}^2}
	+
	C\left(1+ \norm{\bff{u}_h^{n-1}}{\bb{L}^4} \right) (h^{r+1}+k) \Big] \norm{\bff{\chi}}{\bb{H}^1} 
	\\
	&\leq
	C\left(1+ \norm{\bff{u}_h^{n-1}}{\bb{L}^4}^2 \right) \norm{\bff{\theta}^{n-1}}{\bb{L}^2}^2
	+
	C\left(1+ \norm{\bff{u}_h^{n-1}}{\bb{L}^4}^2 \right) (h^{2(r+1)}+k^2)
	+
	\epsilon \norm{\bff{\chi}}{\bb{H}^1}^2,
\end{align*}
where in the penultimate step we used~\eqref{equ:un2 L43}, Lemma~\ref{lem:stab elliptic}, and the embedding $\bb{H}^1\hookrightarrow \bb{L}^4$, while in the last step we used Young's inequality and \eqref{equ:eta t}. This proves~\eqref{equ:B nonlinear est}.

Next, noting~\eqref{equ:C div} and~\eqref{equ:uhn1 min un}, by H\"older's inequality and Lemma~\ref{lem:stab elliptic} we have
\begin{align*}
	&\big| \mathcal{C}(\bff{u}_h^{n-1}; \Pi_h \bff{u}^n, \bff{\chi}) - \mathcal{C}(\bff{u}^n; \Pi_h \bff{u}^n, \bff{\chi}) \big| 
	\\
	&\leq
	\big| \inpro{(\bff{u}_h^{n-1}-\bff{u}^n) \times \nabla \Pi_h \bff{u}^n}{\nabla \bff{\chi}}
	+
	\big| \inpro{\Pi_h\bff{u}^n \times \bff{e} \big( \bff{e}\cdot (\bff{u}_h^{n-1}-\bff{u}^n)\big)}{\bff{\chi}}
	\\
	&\quad
	+
	\big| \beta_2 \inpro{(\bff{u}_h^{n-1}-\bff{u}^n) \otimes \bff{\nu}}{\nabla(\Pi_h \bff{u}^n \times \bff{\chi})}
	+
	\big| \beta_2 \inpro{\Pi_h \bff{u}^n \times (\nabla\cdot\bff{\nu}) (\bff{u}_h^{n-1}-\bff{u}^n)}{\bff{\chi}}
	\\
	&\leq
	C \norm{\bff{u}_h^{n-1}- \bff{u}^n}{\bb{L}^2} \norm{\Pi_h \bff{u}^n}{\bb{W}^{1,\infty}} \norm{\bff{\chi}}{\bb{H}^1}
	\\
	&\leq
	C \norm{\bff{\theta}^{n-1} + \bff{\eta}^{n-1} - k\cdot \mathrm{d}_t \bff{u}^n}{\bb{L}^2}  \norm{\bff{\chi}}{\bb{H}^1} 
	\\
	&\leq
	C\norm{\bff{\theta}^{n-1}}{\bb{L}^2}^2 + Ch^{2(r+1)} + Ck^2
	+
	\epsilon \norm{\bff{\chi}}{\bb{H}^1}^2,
\end{align*}
where in the last step we used~\eqref{equ:norm delta un Lp}, \eqref{equ:eta t}, \eqref{equ:ass 1}, and Young's inequality. This proves~\eqref{equ:C nonlinear est}.

Noting~\eqref{equ:D div}, \eqref{equ:ass 1}, and \eqref{equ:uhn1 min un}, we have by H\"older's inequality,
\begin{align*}
	\big| \mathcal{D}(\bff{u}_h^{n-1}, \bff{\chi}) - \mathcal{D}(\bff{u}^n, \bff{\chi}) \big| 
	&\leq
	\big| \beta_1 \inpro{(\bff{u}_h^{n-1}-\bff{u}^n) \otimes \bff{\nu}}{\nabla \bff{\chi}}
	+
	\big| \beta_1 \inpro{(\nabla \cdot \bff{\nu}) (\bff{u}_h^{n-1}-\bff{u}^n)}{\bff{\chi}} \big| 
	\\
	&\leq
	C\norm{\bff{u}_h^{n-1}-\bff{u}^n}{\bb{L}^2} \norm{\bff{\chi}}{\bb{H}^1}
	\\
	&=
	C \norm{\bff{\theta}^{n-1} + \bff{\eta}^{n-1} - k\cdot \mathrm{d}_t \bff{u}^n}{\bb{L}^2} \norm{\bff{\chi}}{\bb{H}^1} 
	\\
	&\leq
	C \norm{\bff{\theta}^{n-1}}{\bb{L}^2}^2 + Ch^{2(r+1)} + Ck^2+ \epsilon \norm{\bff{\chi}}{\bb{H}^1}^2, 
\end{align*}
thus proving~\eqref{equ:D nonlinear est}. 

Finally, noting the definition of $\mathcal{A}$ in \eqref{equ:bilinear A}, we infer inequality~\eqref{equ:A nonlinear} from~\eqref{equ:B nonlinear est}, \eqref{equ:C nonlinear est}, and the triangle inequality. This completes the proof of the lemma.
\end{proof}

We can now prove an error estimate for scheme~\eqref{equ:scheme spin}.

\begin{proposition}\label{pro:est theta n L2}
	Let $\bff{\theta}^n$ be as defined in \eqref{equ:split theta eta}, where $\bff{u}_h^n$ and $\bff{u}^n$ solve~\eqref{equ:scheme spin} and~\eqref{equ:weak form AB}, respectively. Then for $n\in \{1,2,\ldots,\lfloor T/k \rfloor\}$,
	\begin{align*}
		\norm{\bff{\theta}^n}{\bb{L}^2}^2
		+
		k \sum_{m=1}^n \norm{\nabla \bff{\theta}^m}{\bb{L}^2}^2
		\leq
		C (h^{2(r+1)}+k^2),
	\end{align*}
	where $C$ depends on the coefficients of the equation, $K_r$, $T$, and $\mathscr{D}$ (but is independent of $n$, $h$ or $k$).
\end{proposition}

\begin{proof}
Noting the definition of $\bff{\theta}^n=\bff{u}_h^n-\Pi_h \bff{u}^n$ and $\bff{\eta}^n= \Pi_h\bff{u}^n-\bff{u}^n$, the scheme~\eqref{equ:scheme spin}, and the weak formulation~\eqref{equ:weak special}, we have for all $\bff{\chi}\in \bb{V}_h$,
\begin{align}\label{equ:A theta n}
	\inpro{\mathrm{d}_t \bff{\theta}^n}{\bff{\chi}}
	+
	\mathcal{A}(\bff{u}_h^{n-1};\bff{\theta}^n, \bff{\chi})
	&=
	\inpro{\mathrm{d}_t \bff{u}_h^n}{\bff{\chi}}
	+
	\mathcal{A}(\bff{u}_h^{n-1};\bff{u}_h^n, \bff{\chi})
	-
	\inpro{\mathrm{d}_t \Pi_h \bff{u}^n}{\bff{\chi}}
	-
	\mathcal{A}(\bff{u}_h^{n-1}; \Pi_h \bff{u}^n, \bff{\chi})
	\nonumber\\
	&=
	\mathcal{D}(\bff{u}_h^{n-1},\bff{\chi})
	-
	\inpro{\mathrm{d}_t \Pi_h\bff{u}^n- \partial_t \bff{u}^n}{\bff{\chi}}
	-
	\inpro{\partial_t \bff{u}^n}{\bff{\chi}}
	\nonumber\\
	&\quad
	-
	\left(\mathcal{A}(\bff{u}_h^{n-1}; \Pi_h \bff{u}^n, \bff{\chi})- \mathcal{A}(\bff{u}^n;\Pi_h \bff{u}^n,\bff{\chi})\right) 
	-
	\mathcal{A}(\bff{u}^n; \Pi_h\bff{u}^n, \bff{\chi})
	\nonumber\\
	&=
	\mathcal{D}(\bff{u}_h^{n-1} - \bff{u}^n, \bff{\chi})
	-
	\inpro{\mathrm{d}_t \bff{\eta}^n}{\bff{\chi}}
	-
	\inpro{\mathrm{d}_t \bff{u}^n-\partial_t \bff{u}^n}{\bff{\chi}}
	\nonumber\\
	&\quad
	-
	\left(\mathcal{A}(\bff{u}_h^{n-1}; \Pi_h \bff{u}^n, \bff{\chi})- \mathcal{A}(\bff{u}^n;\Pi_h \bff{u}^n,\bff{\chi})\right).
\end{align}
Now, we set $\bff{\chi}=\bff{\theta}^n$ in~\eqref{equ:A theta n}. Applying Lemma~\ref{lem:ineq forms} and Proposition~\ref{pro:eta t}, noting~\eqref{equ:norm delta un Lp}, the identity~\eqref{equ:aab} and the coercivity of $\mathcal{A}$ in \eqref{equ:A coercive}, we obtain
\begin{align*}
	&\frac{1}{2k} \left(\norm{\bff{\theta}^n}{\bb{L}^2}^2 - \norm{\bff{\theta}^{n-1}}{\bb{L}^2}^2 \right)
	+
	\frac{1}{2k} \norm{\bff{\theta}^n - \bff{\theta}^{n-1}}{\bb{L}^2}^2
	+
	\alpha \norm{\bff{\theta}^n}{\bb{H}^1}^2
	\\
	&\leq 
	\big| \mathcal{D}(\bff{u}_h^{n-1}-\bff{u}^n, \bff{\theta}^n) \big|
	+
	\big| \inpro{\mathrm{d}_t \bff{\eta}^n}{\bff{\theta}^n} \big| 
	+
	\big|\inpro{\mathrm{d}_t \bff{u}^n-\partial_t \bff{u}^n}{\bff{\theta}^n}\big|
	+
	\big|\mathcal{A}(\bff{u}_h^{n-1}; \Pi_h \bff{u}^n, \bff{\theta}^n)- \mathcal{A}(\bff{u}^n;\Pi_h \bff{u}^n,\bff{\theta}^n) \big|
	\\
	&\leq
	C\left(1+ \norm{\bff{u}_h^{n-1}}{\bb{L}^4}^2 \right) \norm{\bff{\theta}^{n-1}}{\bb{L}^2}^2
	+
	C\left(1+ \norm{\bff{u}_h^{n-1}}{\bb{L}^4}^2 \right) (h^{2(r+1)}+k^2)
	+
	\frac{\alpha}{2}  \norm{\bff{\theta}^n}{\bb{H}^1}^2.
\end{align*}
Summing over $m\in \{1,2,\ldots,n\}$, using the embedding $\bb{H}^1\hookrightarrow \bb{L}^4$ and noting~\eqref{equ:stab L2 lin scheme}, we infer
\begin{align*}
	\norm{\bff{\theta}^n}{\bb{L}^2}^2
	+
	k \sum_{m=1}^n \norm{\bff{\theta}^m}{\bb{H}^1}^2
	\leq
	\norm{\bff{\theta}^0}{\bb{L}^2}^2
	+
	C \left(h^{2(r+1)} + k^2 \right)
	+
	Ck \sum_{m=1}^n \left(1+ \norm{\bff{u}_h^{m-1}}{\bb{H}^1}^2 \right) \norm{\bff{\theta}^{m-1}}{\bb{L}^2}^2.
\end{align*}
Note that since we take $\bff{u}_h^0= P_h\bff{u}^0$, by~\eqref{equ:proj approx} and~\eqref{equ:eta t},
\begin{align*}
\norm{\bff{\theta}^0}{\bb{L}^2}= \norm{\bff{u}_h^0- \Pi_h \bff{u}^0}{\bb{L}^2}
\leq 
\norm{P_h \bff{u}^0- \bff{u}^0}{\bb{L}^2} + \norm{\bff{u}^0- \Pi_h \bff{u}^0}{\bb{L}^2}
\leq
Ch^{r+1}.
\end{align*}
Thus, by the discrete Gronwall lemma (Lemma~\ref{lem:disc gron}), we obtain
\begin{align*}
	\norm{\bff{\theta}^n}{\bb{L}^2}^2 + k \sum_{m=1}^n \norm{\bff{\theta}^m}{\bb{H}^1}^2
	\leq
	C \left(h^{2(r+1)} + k^2 \right)
	\exp\left[ Ck\sum_{m=1}^n \left(1+ \norm{\bff{u}_h^{m-1}}{\bb{H}^1}^2 \right) \right]
	\leq
	C \left(h^{2(r+1)} + k^2 \right),
\end{align*}
where in the last step we again used~\eqref{equ:stab L2 lin scheme}. This completes the proof of the proposition.
\end{proof}

\begin{theorem}\label{the:spin torq error}
	Let $\bff{u}_h^n$ and $\bff{u}$ be the solution of \eqref{equ:scheme spin} and \eqref{equ:weak form AB}, respectively. For $n\in \{1,2,\ldots,\lfloor T/k \rfloor\}$,
	\begin{align*}
		\norm{\bff{u}_h^n-\bff{u}(t_n)}{\bb{L}^2}
		+
		\left(k \sum_{m=1}^n \norm{\nabla \bff{u}_h^m-\nabla \bff{u}(t_n)}{\bb{L}^2}^2 \right)^{\frac12}
		\leq 
		C(h^{r+1} +k),
	\end{align*}
	where $C$ depends on the coefficients of the equation, $K_r$, $T$, and $\mathscr{D}$ (but is independent of $n$, $h$ or $k$).
\end{theorem}

\begin{proof}
	This follows from Proposition \ref{pro:est theta n L2} and the triangle inequality (noting \eqref{equ:split theta eta} and~\eqref{equ:eta t}).
\end{proof}

\section{An energy-stable implicit fully discrete scheme for the LLB equation}\label{sec:scheme 2}

In the previous section, we proposed a linear scheme for the LLB equation \eqref{equ:llb a} and proved its convergence in $\ell^\infty(0,T;\bb{L}^2)\cap \ell^2(0,T;\bb{H}^1)$. However, convergence in $\ell^\infty(0,T;\bb{H}^1)$ could not be established analytically, primarily due to difficulties in deriving an energy stability estimate caused by nonlinear cross-product terms. Despite this, numerical results (see Section~\ref{sec:num exp}) seem to suggest an optimal-rate convergence in~$\ell^\infty(0,T;\bb{H}^1)$.

With the aim of obtaining a scheme with provable convergence in the energy norm, we consider in this section a fully implicit scheme for the LLB equation, restricted to the physically relevant case of \emph{negligible non-adiabatic torque} ($\beta_2=0$). Even with this simplification, the equation remains a quasilinear vector-valued PDE. For the proposed scheme, we are able to show unconditional stability in $\ell^\infty(0,T;\bb{H}^1)$, leading to a proof of unconditional convergence at an optimal rate in the same norm. As a corollary, in the absence of current ($\bff{\nu}=\bff{0}$), we show that this scheme also ensures energy dissipation (see \eqref{equ:energy identity}) at the discrete level.

Recall the weak formulation for the problem in this special case ($\beta_2=0$):
\begin{align}\label{equ:weak form no spin}
	\inpro{\partial_t \bff{u}(t)}{\bff{\chi}}
	=
	-\inpro{\bff{u}(t)\times \bff{H}(t)}{\bff{\chi}}
	+
	\alpha \inpro{\bff{H}(t)}{\bff{\chi}}
	+
	\beta_1 \inpro{(\bff{\nu}(t) \cdot\nabla) \bff{u}(t)}{\bff{\chi}},
	\quad \forall \bff{\chi}\in \bb{H}^1.
\end{align}
Here, $\bff{H}=\Delta \bff{u}-\bff{w}$, where $\bff{w}=\bff{u}(t) +|\bff{u}(t)|^2 \bff{u}(t) + \bff{e} \big(\bff{e}\cdot \bff{u}(t) \big)$ as defined in~\eqref{equ:w}, and $\bff{\nu}$ is a given current density such that $\bff{\nu}\cdot \bff{n}=0$ on $\partial\mathscr{D}$.
A fully discrete scheme to solve~\eqref{equ:weak form no spin} can be described as follows. Let $\bb{V}_h$ be the finite element space defined in~\eqref{equ:Vh} with $r\geq 1$. Let $k$ be the time step and $\bff{u}_h^n$ be the approximation in $\bb{V}_h$ of $\bff{u}(t)$ at time $t=t_n:=nk\in [0,T]$ for $n=0,1,2,\ldots, N$, where $N:=\lfloor T/k \rfloor$.
We start with $\bff{u}_h^0= R_h \bff{u}_0 \in \bb{V}_h$ for simplicity. For $t_n\in [0,T]$, given $\bff{u}_h^{n-1}\in \bb{V}_h$, define $\bff{u}_h^n$ by
\begin{align}\label{equ:scheme 2}
	\inpro{\mathrm{d}_t \bff{u}_h^n}{\bff{\chi}}
	&=
	-
	\inpro{\bff{u}_h^n \times \bff{H}_h^n}{\bff{\chi}}
	+
	\alpha \inpro{\bff{H}_h^n}{\bff{\chi}}
	+
	\beta_1 \inpro{(\bff{\nu}^n \cdot\nabla) \bff{u}_h^{n-1}}{\bff{\chi}}, 
	\quad \forall \bff{\chi}\in \bb{V}_h,
\end{align}
where
\begin{align}
	\label{equ:Hhn}
	\bff{H}_h^n
	&:=
	\Delta_h \bff{u}_h^n
	-
	\bff{w}_h^n, 
	\\
	\label{equ:whn}
	\bff{w}_h^n
	&:=
	\bff{u}_h^n
	+
	P_h \big(|\bff{u}_h^n|^2 \bff{u}_h^n\big)
	+
	\bff{e}\left(\bff{e}\cdot \bff{u}_h^n \right). 
\end{align}
Throughout this section, for ease of presentation we assume $d=3$, noting that a similar (and simpler) argument will also hold for $d=1$ or $2$. 
Here, to attain an optimal order of convergence for scheme \eqref{equ:scheme spin}, we assume that the problem~\eqref{equ:llb a} admits a sufficiently regular solution~$\bff{u}$ which satisfies
\begin{equation}\label{equ:ass 2}
	\norm{\bff{u}}{L^\infty(\bb{H}^{r+1})}
	+ \norm{\bff{u}}{L^\infty(\bb{H}^4)}
	+ \norm{\partial_t \bff{u}}{L^\infty(\bb{H}^{r+1})}
	+ \norm{\partial_t^2 \bff{u}}{L^\infty(\bb{L}^2)}
	\leq K_r,
\end{equation}
for some positive constant $K_r$ depending on the initial data, $r$, and $T$. Here, $r$ is the degree of piecewise polynomials in the finite element space~\eqref{equ:Vh}.

In the following lemma, we show the unconditional stability of \eqref{equ:scheme 2} in the energy norm. 

\begin{lemma}\label{lem:H1 stable}
	Let $T>0$ be given and suppose that $\bff{u}_h^n$ satisfies \eqref{equ:scheme 2}. Then for any initial data $\bff{u}_h^0\in \bb{V}_h$ and $n\in \{1,2,\ldots, \lfloor T/k \rfloor\}$, we have
	\begin{align}\label{equ:stab H1}
		\norm{\bff{u}_h^n}{\bb{L}^4}^4
		+
		\norm{\bff{u}_h^n}{\bb{H}^1}^2
		+
		k\sum_{m=1}^n \norm{\bff{H}_h^{m}}{\bb{L}^2}^2
		\leq
		C_{\rm{S}},
	\end{align}
	where $C_{\rm{S}}$ depends only on the coefficients of the equation, $\norm{\bff{u}_h^0}{\bb{H}^1}$, and $\mathscr{D}$. If $\bff{\nu}=\bff{0}$, then for any $k\in (0,1)$,
	\begin{align}\label{equ:stab ene}
		\mathcal{E}(\bff{u}_h^{n}) 
		\leq
		\mathcal{E}(\bff{u}_h^{n-1}),
	\end{align}
	where $\mathcal{E}$ is the energy functional defined in \eqref{equ:energy}.
\end{lemma}

\begin{proof}
	Let $\bff{H}_h^n$ be as defined in~\eqref{equ:Hhn}. Taking $\bff{\chi}= \bff{H}_h^n$ in~\eqref{equ:scheme 2}, we have
	\begin{align}\label{equ:dtuhn H}
		\inpro{\mathrm{d}_t \bff{u}_h^n}{\bff{H}_h^n}
		&=
		\alpha \norm{\bff{H}_h^n}{\bb{L}^2}^2
		+
		\beta_1 \inpro{(\bff{\nu}^n \cdot\nabla) \bff{u}_h^{n-1}}{\bff{H}_h^n}
	\end{align}
	On the other hand, taking the inner product of \eqref{equ:Hhn} with $\mathrm{d}_t \bff{u}_h^n$, using vector identities \eqref{equ:aab} and \eqref{equ:a2aab}, we obtain
	\begin{align}\label{equ:H du mu}
		\inpro{\bff{H}_h^n}{\mathrm{d}_t \bff{u}_h^n}
		&=
		-
		\frac{1}{2k} \big(\norm{\bff{u}_h^{n}}{\bb{L}^2}^2 - \norm{ \bff{u}_h^{n-1}}{\bb{L}^2}^2 \big)
		-
		\frac{1}{2k}\norm{\bff{u}_h^{n}- \bff{u}_h^{n-1}}{\bb{L}^2}^2
		\nonumber\\
		&\quad
		-
		\frac{1}{2k} \big(\norm{\nabla \bff{u}_h^{n}}{\bb{L}^2}^2 - \norm{\nabla \bff{u}_h^{n-1}}{\bb{L}^2}^2 \big)
		-
		\frac{1}{2k}\norm{\nabla \bff{u}_h^{n}- \nabla \bff{u}_h^{n-1}}{\bb{L}^2}^2
		\nonumber\\
		\nonumber
		&\quad
		-
		\frac{1}{4k} \left(\norm{\bff{u}_h^{n}}{\bb{L}^4}^4 - \norm{\bff{u}_h^{n-1}}{\bb{L}^4}^4 \right)
		-
		\frac{1}{4k} \norm{\abs{\bff{u}_h^{n}}^2 - \abs{\bff{u}_h^{n-1}}^2}{\bb{L}^2}^2
		-
		\frac{k}{2} \norm{\abs{\bff{u}_h^{n}} \abs{\mathrm{d}_t \bff{u}_h^{n}}}{\bb{L}^2}^2
		\\
		&\quad
		-
		\frac{1}{2k} \left(\norm{\bff{e}\cdot \bff{u}_h^{n}}{\bb{L}^2}^2 - \norm{\bff{e}\cdot \bff{u}_h^{n-1}}{L^2}^2 \right)
		-
		\frac{1}{2k} \norm{\bff{e}\cdot (\bff{u}_h^{n}-\bff{u}_h^{n-1})}{\bb{L}^2}^2.
	\end{align}
	Substituting~\eqref{equ:H du mu} into \eqref{equ:dtuhn H} and rearranging the terms yields
	\begin{align}\label{equ:ene beta1}
		&\frac{1}{2k} \big(\norm{\bff{u}_h^{n}}{\bb{L}^2}^2 - \norm{ \bff{u}_h^{n-1}}{\bb{L}^2}^2 \big)
		+
		\frac{1}{2k}\norm{\bff{u}_h^{n}- \bff{u}_h^{n-1}}{\bb{L}^2}^2
		\nonumber\\
		&\quad
		+
		\frac{1}{2k} \big(\norm{\nabla \bff{u}_h^{n}}{\bb{L}^2}^2 - \norm{\nabla \bff{u}_h^{n-1}}{\bb{L}^2}^2 \big)
		+
		\frac{1}{2k}\norm{\nabla \bff{u}_h^{n}- \nabla \bff{u}_h^{n-1}}{\bb{L}^2}^2
		\nonumber\\
		\nonumber
		&\quad
		+
		\frac{1}{4k} \left(\norm{\bff{u}_h^{n}}{\bb{L}^4}^4 - \norm{\bff{u}_h^{n-1}}{\bb{L}^4}^4 \right)
		+
		\frac{1}{4k} \norm{\abs{\bff{u}_h^{n}}^2 - \abs{\bff{u}_h^{n-1}}^2}{\bb{L}^2}^2
		+
		\frac{k}{2} \norm{\abs{\bff{u}_h^{n}} \abs{\mathrm{d}_t \bff{u}_h^{n}}}{\bb{L}^2}^2
		\nonumber\\
		&\quad
		+
		\frac{1}{2k} \left(\norm{\bff{e}\cdot \bff{u}_h^{n}}{\bb{L}^2}^2 - \norm{\bff{e}\cdot \bff{u}_h^{n-1}}{L^2}^2 \right)
		+
		\frac{1}{2k} \norm{\bff{e}\cdot (\bff{u}_h^{n}-\bff{u}_h^{n-1})}{\bb{L}^2}^2
		+
		\alpha \norm{\bff{H}_h^n}{\bb{L}^2}^2
		\nonumber\\
		&=
		-\beta_1 \inpro{(\bff{\nu}^n \cdot\nabla) \bff{u}_h^{n-1}}{\bff{H}_h^n}
		\\
		\nonumber
		&\leq
		\frac{\alpha}{4} \norm{\bff{H}_h^n}{\bb{L}^2}^2
		+
		\frac{(\beta_1 \nu_\infty)^2}{\alpha} \norm{\nabla \bff{u}_h^{n-1}}{\bb{L}^2}^2,
	\end{align}
	where in the last step we used Young's inequality. Rearranging the terms and applying discrete Gronwall's lemma (Lemma~\ref{lem:disc gron}), we obtain~\eqref{equ:stab H1}. 
	
	Continuing from~\eqref{equ:ene beta1}, if $\bff{\nu}=\bff{0}$, then the right-hand side of equation~\eqref{equ:ene beta1} is zero, which implies \eqref{equ:stab ene} upon rearranging the terms. This completes the proof of the lemma.
\end{proof}

We also have the following discrete $\ell^2(0,T;\bb{L}^\infty)$ stability of the scheme.

\begin{lemma}\label{lem:Linfty stable}
	Let $T>0$ be given and let $\bff{u}_h^n$ be defined by \eqref{equ:scheme 2}. Then for any initial data $\bff{u}_h^0\in \bb{V}_h$ and $n\in \{1,2,\ldots, \lfloor T/k \rfloor\}$,
	\begin{align}\label{equ:stab disc lap}
		k \sum_{m=1}^n \norm{\Delta_h \bff{u}_h^{m}}{\bb{L}^2}^2
		\leq
		C_{\Delta}.
	\end{align}
	Consequently, if $\mathscr{D}$ is a convex polygonal or polyhedral domain with quasi-uniform triangulation, then
	\begin{align}\label{equ:stab L infty}
		k \sum_{m=1}^n \norm{\bff{u}_h^m}{\bb{L}^\infty}^2
		\leq
		C_{\infty},
	\end{align}
	where constants $C_{\Delta}$ and $C_\infty$ depend only on the coefficients of the equation, $\norm{\bff{u}_h^0}{\bb{H}^1}$, $\mathscr{D}$, and $T$.
\end{lemma}

\begin{proof}
	Noting~\eqref{equ:Hhn}, by H\"older's inequality we have
	\begin{align*}
		\norm{\Delta_h \bff{u}_h^{n}}{\bb{L}^2}
		&\leq
		\norm{\bff{H}_h^{n}}{\bb{L}^2}
		+
		\norm{\bff{u}_h^n}{\bb{L}^2}
		+
		\norm{\bff{u}_h^{n}}{\bb{H}^1}^3
		+
		\norm{\bff{e}}{\bb{L}^\infty}^2 \norm{\bff{u}_h^n}{\bb{L}^2},
	\end{align*}
	where in the last step we used the Sobolev embedding $\bb{H}^1 \hookrightarrow \bb{L}^6$. Multiplying both sides by $k$, then summing over $m\in \{1,2,\ldots, n\}$, and using \eqref{equ:stab H1} give \eqref{equ:stab disc lap}. Applying \eqref{equ:disc lapl L infty} then yields \eqref{equ:stab L infty}.
\end{proof}

Note that scheme~\eqref{equ:scheme 2} is nonlinear. In the following lemma, we show that the scheme is well-posed.

\begin{lemma}\label{lem:exact sol}
	Given $k>0$, $\bff{u}_h^{n-1}\in \bb{V}_h$, and current density $\bff{\nu}$ such that $\norm{\bff{\nu}}{L^\infty(\bb{L}^\infty)}\leq \nu_\infty$, there exists $\bff{u}_h^n\in \bb{V}_h$ that solves scheme \eqref{equ:scheme 2}.
\end{lemma}

\begin{proof}
	Given $\bff{u}_h^{n-1}\in \bb{V}_h$, let $\mathcal{A}_n:\bb{V}_h\to \bb{V}_h$ be a nonlinear operator defined by
	\begin{equation*}
		\mathcal{A}_n(\bff{v}):=
		\bff{v}
		+
		P_h\big(|\bff{v}|^2 \bff{v}\big)
		+
		\bff{e}(\bff{e}\cdot \bff{v}).
	\end{equation*}
	Define a map $G_h:\bb{V}_h\to \bb{V}_h$ by
	\begin{align*}
		G_h(\bff{v})
		&:=
		\bff{v}
		-
		\bff{u}_h^{n-1}
		+
		k P_h \big( \bff{v}\times (\Delta_h \bff{v}- \mathcal{A}_n(\bff{v})) \big)
		-
		k \alpha \big( \Delta_h \bff{v}- \mathcal{A}_n(\bff{v}) \big)
		-
		k\beta_1 (\bff{\nu}^n \cdot\nabla)\bff{u}_h^{n-1}.
	\end{align*}
	The scheme~\eqref{equ:scheme 2} is equivalent to solving~$G_h\big(\bff{u}_h^n \big)= \bff{0}$. The existence of solution to this equation will be shown using the Brouwer fixed point theorem. To this end, let~$B_\rho:= \{\bff{v}\in \bb{V}_h: \norm{\bff{v}}{\bb{L}^2} \leq \rho\}$. For any $\bff{v}\in \partial B_\rho:= \{\bff{v}\in \bb{V}_h: \norm{\bff{v}}{\bb{L}^2}=\rho\}$, we have
	\begin{align*}
		\inpro{G_h(\bff{v})}{\bff{v}}
		&=
		\norm{\bff{v}}{\bb{L}^2}^2
		-
		\inpro{\bff{u}_h^{n-1}}{\bff{v}}
		+
		k\alpha \norm{\nabla \bff{v}}{\bb{L}^2}^2
		+
		k\alpha \inpro{\mathcal{A}_n(\bff{v})}{\bff{v}}
		-
		k\beta_1 \inpro{(\bff{\nu}^n \cdot\nabla)\bff{u}_h^{n-1}}{\bff{v}}
		\\
		&=
		\frac12 \left(\norm{\bff{v}}{\bb{L}^2}^2 - \norm{\bff{u}_h^{n-1}}{\bb{L}^2}^2 \right) 
		+
		\frac12 
		\norm{\bff{v}-\bff{u}_h^{n-1}}{\bb{L}^2}^2
		+
		k\alpha \norm{\nabla \bff{v}}{\bb{L}^2}^2
		+
		k\alpha \norm{\bff{v}}{\bb{L}^2}^2
		+
		k\alpha \norm{\bff{v}}{\bb{L}^4}^4
		\\
		&\quad
		+
		k\alpha \norm{\bff{e}\cdot \bff{v}}{\bb{L}^2}^2
		-
		k\beta_1 \inpro{(\bff{\nu}^{n-1}\cdot\nabla)\bff{u}_h^{n-1}}{\bff{v}}
		\\
		&\geq
		\frac12 \left(\norm{\bff{v}}{\bb{L}^2}^2 - \norm{\bff{u}_h^{n-1}}{\bb{L}^2}^2 \right) 
		+
		k\alpha \norm{\nabla \bff{v}}{\bb{L}^2}^2
		+
		k\alpha \norm{\bff{v}}{\bb{L}^2}^2
		+
		k\alpha \norm{\bff{v}}{\bb{L}^4}^4
		-
		k\beta_1 \nu_\infty \norm{\nabla \bff{u}_h^{n-1}}{\bb{L}^2} \norm{\bff{v}}{\bb{L}^2}
		\\
		&\geq
		\frac12 \left(\norm{\bff{v}}{\bb{L}^2}^2 - \norm{\bff{u}_h^{n-1}}{\bb{L}^2}^2 \right) 
		+
		k\alpha \norm{\bff{v}}{\bb{L}^2}^2
		-
		k\beta_1 \nu_\infty \norm{\nabla \bff{u}_h^{n-1}}{\bb{L}^2} \norm{\bff{v}}{\bb{L}^2}
		\\
		&\geq
		\frac12 \left(\rho^2 - C_{\rm{S}}\right) 
		+ 
		k \Big( \alpha\rho^2  - \beta_1 \nu_\infty C_{\rm{S}}^{\frac12} \rho \Big).
	\end{align*}
	Here, $C_{\rm{S}}$ is the constant in the stability estimate~\eqref{equ:stab H1}. Therefore, for sufficiently large $\rho$, precisely
	\[
		\rho > \max \Big\{ C_{\rm{S}}^{\frac12}, \beta_1 \nu_\infty C_{\rm{S}}^{\frac12} \alpha^{-1} \Big\},
	\]
	we have $\inpro{G_h(\bff{v})}{\bff{v}}>0$. Thus, we infer the existence of $\bff{u}_h^n\in \bb{V}_h$ solving~\eqref{equ:scheme 2} by the Brouwer's fixed point theorem.
\end{proof}

\begin{remark}
In fact if $\bff{\nu}=\bff{0}$, then $\mathcal{E}\big(\bff{u}_h^n\big)$ decays to zero exponentially fast as $t_n\to \infty$, mimicking the behaviour of the exact solution (see Theorem~\ref{the:decay}) but possibly with different exponents. Indeed, taking $\bff{\chi}= \bff{u}_h^n$ in \eqref{equ:scheme 2} with $\bff{\nu}=\bff{0}$, noting \eqref{equ:Hhn} and \eqref{equ:aab}, gives
\begin{align}\label{equ:12k uhn L2}
	\frac{1}{2k} \left(\norm{\bff{u}_h^{n}}{\bb{L}^2}^2 - \norm{\bff{u}_h^{n-1}}{\bb{L}^2}^2 \right) 
	+
	\alpha \norm{\nabla \bff{u}_h^{n}}{\bb{L}^2}^2
	+
	\alpha \norm{\bff{u}_h^{n}}{\bb{L}^2}^2
	+
	\alpha \norm{\bff{u}_h^{n}}{\bb{L}^4}^4
	+
	\alpha \norm{\bff{e}\cdot\bff{u}_h^n}{\bb{L}^2}^2
	\leq
	0.
\end{align}
Moreover, discarding non-negative terms in equation~\eqref{equ:ene beta1} with $\bff{\nu}=\bff{0}$ yields
\begin{align}\label{equ:12k uhn H1}
	&\frac{1}{2k} \big(\norm{\bff{u}_h^{n}}{\bb{L}^2}^2 - \norm{ \bff{u}_h^{n-1}}{\bb{L}^2}^2 \big)
	+
	\frac{1}{2k} \big(\norm{\nabla \bff{u}_h^{n}}{\bb{L}^2}^2 - \norm{\nabla \bff{u}_h^{n-1}}{\bb{L}^2}^2 \big)
	\nonumber\\
	&\quad
	+
	\frac{1}{4k} \left(\norm{\bff{u}_h^{n}}{\bb{L}^4}^4 - \norm{\bff{u}_h^{n-1}}{\bb{L}^4}^4 \right)
	+
	\frac{1}{2k} \left(\norm{\bff{e}\cdot \bff{u}_h^{n}}{\bb{L}^2}^2 - \norm{\bff{e}\cdot \bff{u}_h^{n-1}}{L^2}^2 \right)
	\leq 0.
\end{align}
Adding~\eqref{equ:12k uhn L2} and \eqref{equ:12k uhn H1}, noting the definition of $\mathcal{E}$ in \eqref{equ:energy}, then applying a version of the discrete Gronwall lemma~\cite[equation (3.4)]{Emm99}, we obtain
\begin{align*}
	\mathcal{E}\big(\bff{u}_h^n\big)  \leq Ce^{-\beta t_n}\, \mathcal{E}\big(\bff{u}_0\big)
\end{align*}
for some positive constants $C$ and $\beta$.
\end{remark}

\begin{remark}
Since scheme~\eqref{equ:scheme 2} is nonlinear, it is necessary to solve a system of nonlinear equations at each time step. Thus, a linearisation of the scheme is required. In Appendix~\ref{sec:linearisation}, we describe a simple fixed-point iteration for solving the resulting nonlinear system and establish its convergence. A detailed study of related topics, such as the efficiency of alternative iterative solvers (e.g., Newton’s method) and the discrete energy stability of the scheme when the nonlinear problem is solved approximately, will be the subject of future work. In the analysis that follows, we assume that the nonlinear system associated with scheme~\eqref{equ:scheme 2} is solved \emph{exactly} at each time step (cf. Lemma~\ref{lem:exact sol}). Under this assumption, the stability estimates in stated in Lemma~\ref{lem:H1 stable} and \ref{lem:Linfty stable} hold.
\end{remark}

Next, we will derive some estimates for the nonlinear terms to aid in the analysis, analogous to those in the previous section. To this end, we split the approximation error by writing
\begin{align}\label{equ:theta rho split 2}
	\bff{u}_h^n-\bff{u}^n
	=
	\left(\bff{u}_h^n - R_h \bff{u}^n\right)
	+
	\left(R_h \bff{u}^n - \bff{u}^n\right)
	=
	\bff{\theta}^n + \bff{\rho}^n,
\end{align}
where $R_h$ is the Ritz projection defined in~\eqref{equ:Ritz}, while $\bff{\rho}^n$ enjoys the estimates~\eqref{equ:Ritz ineq}, \eqref{equ:Ritz ineq L infty}, and \eqref{equ:Ritz stab u infty}. 
We also remark that since $\Delta_h R_h \bff{v}=P_h \Delta \bff{v}$ for any $\bff{v}\in \bb{H}^2_{\mathrm{N}}$, we have
\begin{align}\label{equ:Delta uhn Delta un}
	\Delta_h \bff{u}_h^n- \Delta \bff{u}^n
	&=
	\Delta_h \bff{\theta}^n + (P_h-I) \Delta \bff{u}^n.
\end{align}

Noting \eqref{equ:Hhn}, \eqref{equ:Ritz}, and~\eqref{equ:Delta uhn Delta un}, we can write
\begin{align}\label{equ:Hh minus Hn}
	\nonumber
	\bff{H}_h^{n}-\bff{H}^{n}
	&=
	\Delta_h \bff{\theta}^{n} 
	+ 
	(P_h-I) \Delta \bff{u}^{n} 
	-
	(\bff{\theta}^n+ \bff{\rho}^n)
	-
	P_h \left(|\bff{u}_h^{n}|^2 \bff{u}_h^{n}-|\bff{u}^{n}|^2 \bff{u}^{n}\right) 
	\nonumber \\
	&\quad
	-
	(P_h-I) \left(|\bff{u}^{n}|^2 \bff{u}^{n}\right)
	-
	\bff{e} \left(\bff{e}\cdot (\bff{\theta}^n+\bff{\rho}^n)\right),
\end{align}
where $I$ is the identity operator. Moreover, we have
\begin{align}\label{equ:split cubic}
	|\bff{u}_h^{n}|^2 \bff{u}_h^{n}-|\bff{u}^{n}|^2 \bff{u}^{n}
	&=
	\abs{\bff{u}_h^n}^2 \left(\bff{\theta}^n+\bff{\rho}^n\right)
	+
	\big((\bff{\theta}^n+\bff{\rho}^n)\cdot (\bff{u}_h^n+\bff{u}^n)\big) \bff{u}^n.
\end{align}

\begin{lemma}\label{lem:inpro theta n}
Let $\epsilon >0$ be given. Let $\bff{u}_h^n$ and $\bff{H}_h^n$ be defined by \eqref{equ:scheme 2} and \eqref{equ:Hhn}, respectively, with initial data $\bff{u}_h^0\in \bb{V}_h$. Then there exists a constant $C$ such that for all $n\in \{1,2,\ldots, \lfloor T/k \rfloor\}$,
\begin{align}
	\label{equ:inpro uhn theta n}
	\big| \inpro{\bff{u}_h^n \times \bff{H}_h^n - \bff{u}^n \times \bff{H}^n}{\bff{\theta}^n} \big| 
	&\leq
	C \left(1+ \norm{\bff{u}_h^n}{\bb{L}^\infty}^2\right) h^{2(r+1)}
	+ 
	C\norm{\bff{\theta}^n}{\bb{L}^2}^2
	+
	\epsilon \norm{\nabla \bff{\theta}^n}{\bb{L}^2}^2
	+
	\epsilon \norm{\Delta_h \bff{\theta}^n}{\bb{L}^2}^2,
	\\
	\label{equ:inpro uhn Delta theta n}
	\big| \inpro{\bff{u}_h^n \times \bff{H}_h^n - \bff{u}^n \times \bff{H}^n}{\Delta_h \bff{\theta}^n} \big| 
	&\leq 
	C \left(1+ \norm{\bff{u}_h^n}{\bb{L}^\infty}^2\right) h^{2(r+1)}  
	+ 
	C\norm{\bff{\theta}^n}{\bb{H}^1}^2
	+
	\epsilon \norm{\Delta_h \bff{\theta}^n}{\bb{L}^2}^2.
\end{align}
Here, $C$ depends on $\epsilon$, the coefficients of the equation, $K_r$, $T$, and $\mathscr{D}$ (but is independent of $n$, $h$ or $k$).
\end{lemma}

\begin{proof}
Noting~\eqref{equ:Hh minus Hn}, we can write
\begin{align*}
	&\bff{u}_h^n \times \bff{H}_h^n- \bff{u}^n\times \bff{H}^n
	\\
	&=
	\bff{u}_h^n \times \left(\bff{H}_h^n-\bff{H}^n\right) 
	+
	\left(\bff{u}_h^n- \bff{u}^n\right) \times \bff{H}^n
	\\
	&=
	\bff{u}_h^n \times \Delta_h \bff{\theta}^n
	+
	\bff{u}_h^n \times (P_h-I) \Delta \bff{u}^n
	-
	\bff{u}_h^n\times \left(\bff{\theta}^n+\bff{\rho}^n\right) 
	-
	\bff{u}_h^n \times 	P_h \left(|\bff{u}_h^{n}|^2 \bff{u}_h^{n}-|\bff{u}^{n}|^2 \bff{u}^{n}\right) 
	\\
	&\quad
	-
	\bff{u}_h^n \times (P_h-I) \left(|\bff{u}^{n}|^2 \bff{u}^{n}\right)
	-
	\bff{u}_h^n \times \left(\bff{e}\cdot (\bff{\theta}^{n}+\bff{\rho}^{n})\right)
	+
	\left( \bff{\theta}^n + \bff{\rho}^n \right) \times \bff{H}^n
	\\
	&=:
	I_1+I_2+\cdots+I_7.
\end{align*}
We will take the inner product of each term on the last line with $\bff{\theta}^n$ and estimate it in the following. For the first term, by H\"older's and Young's inequalities, the stability estimate \eqref{equ:stab H1}, as well as estimates~\eqref{equ:Ritz ineq} and~\eqref{equ:gal nir uh L4}, we have
\begin{align*}
	\abs{\inpro{I_1}{\bff{\theta}^n}}
	&\leq
	\norm{\bff{u}_h^n}{\bb{L}^4} \norm{\Delta_h \bff{\theta}^n}{\bb{L}^2} \norm{\bff{\theta}^n}{\bb{L}^4}
	\leq
	C \norm{\bff{\theta}^n}{\bb{L}^2}^2
	+
	\epsilon \norm{\nabla \bff{\theta}^n}{\bb{L}^2}^2
	+
	\epsilon \norm{\Delta_h \bff{\theta}^n}{\bb{L}^2}^2.
\end{align*}
Similarly, for the terms $I_3$, $I_6$, and $I_7$, we have
\begin{align*}
	\abs{\inpro{I_3}{\bff{\theta}^n}}
	&\leq
	\norm{\bff{u}_h^n}{\bb{L}^4} \norm{\bff{\rho}^n}{\bb{L}^4} \norm{\bff{\theta}^n}{\bb{L}^2}
	\leq
	Ch^{2(r+1)}
	+
	\epsilon \norm{\bff{\theta}^n}{\bb{L}^2}^2,
	\\
	\abs{\inpro{I_6}{\bff{\theta}^n}}
	&\leq
	\norm{\bff{u}_h^n}{\bb{L}^4} \norm{\bff{\theta}^n+\bff{\rho}^n}{\bb{L}^2} \norm{\bff{\theta}^n}{\bb{L}^4}
	\leq
	Ch^{2(r+1)}
	+
	C\norm{\bff{\theta}^n}{\bb{L}^2}^2
	+
	\epsilon \norm{\nabla \bff{\theta}^n}{\bb{L}^2}^2,
	\\
	\abs{\inpro{I_7}{\bff{\theta}^n}}
	&\leq
	\norm{\bff{\rho}^n}{\bb{L}^2} \norm{\bff{H}^n}{\bb{L}^\infty} \norm{\bff{\theta}^n}{\bb{L}^2}
	\leq
	Ch^{2(r+1)}
	+
	\epsilon\norm{\bff{\theta}^n}{\bb{L}^2}^2.
\end{align*}
For the terms $I_2$ and $I_5$, noting the regularity assumption~\eqref{equ:ass 2}, we use \eqref{equ:proj approx} and Young's inequality to obtain
\begin{align*}
	\abs{\inpro{I_2}{\bff{\theta}^n}} + \abs{\inpro{I_5}{\bff{\theta}^n}} 
	&\leq
	\norm{\bff{u}_h^n}{\bb{L}^\infty} \norm{(P_h-I)\Delta \bff{u}^n}{\bb{L}^2} \norm{\bff{\theta}^n}{\bb{L}^2}
	+
	\norm{\bff{u}_h^n}{\bb{L}^\infty} \norm{(P_h-I) \big(|\bff{u}^n|^2 \bff{u}^n\big)}{\bb{L}^2} \norm{\bff{\theta}^n}{\bb{L}^2} 
	\\
	&\leq
	Ch^{2(r+1)} \norm{\bff{u}_h^n}{\bb{L}^\infty}^2 + \epsilon \norm{\bff{\theta}^n}{\bb{L}^2}^2.
\end{align*}
Finally, for the term $I_4$, we use~\eqref{equ:split cubic}, the $\bb{L}^2$-stability of $P_h$, and stability estimate~\eqref{equ:stab H1}, to obtain
\begin{align}\label{equ:I5 theta n}
	\abs{\inpro{I_4}{\bff{\theta}^n}}
	&\leq
	\norm{\bff{u}_h^n}{\bb{L}^6} \norm{P_h \left(|\bff{u}_h^{n}|^2 \bff{u}_h^{n}-|\bff{u}^{n}|^2 \bff{u}^{n}\right)}{\bb{L}^2} \norm{\bff{\theta}^n}{\bb{L}^3}
	\nonumber\\
	&\leq
	\norm{\bff{u}_h^n}{\bb{L}^6} \norm{\bff{u}_h^n}{\bb{L}^6}^2 \norm{\bff{\theta}^n + \bff{\rho}^n}{\bb{L}^6} \norm{\bff{\theta}^n}{\bb{L}^3}
	\nonumber\\
	&\quad
	+
	\norm{\bff{u}_h^n}{\bb{L}^6} \norm{\bff{u}_h^n+\bff{u}^n}{\bb{L}^6} \norm{\bff{u}^n}{\bb{L}^6} \norm{\bff{\theta}^n+\bff{\rho}^n}{\bb{L}^6} \norm{\bff{\theta}^n}{\bb{L}^3}
	\nonumber\\
	&\leq
	C \norm{\bff{\theta}^n}{\bb{L}^6} \norm{\bff{\theta}^n}{\bb{L}^3}
	+
	C \norm{\bff{\rho}^n}{\bb{L}^6} \norm{\bff{\theta}^n}{\bb{L}^3}
	\nonumber\\
	&\leq
	C \norm{\bff{\theta}^n}{\bb{H}^1}^{\frac32} \norm{\bff{\theta}^n}{\bb{L}^2}^{\frac12}
	+
	C h^{r+1} \norm{\bff{\theta}^n}{\bb{H}^1}^{\frac12} \norm{\bff{\theta}^n}{\bb{L}^2}^{\frac12}
	\nonumber \\
	&\leq
	Ch^{2(r+1)}
	+
	C\norm{\bff{\theta}^n}{\bb{L}^2}^2
	+
    \epsilon \norm{\bff{\theta}^n}{\bb{H}^1}^2,
\end{align}
where in the penultimate step we used \eqref{equ:Ritz ineq}, Sobolev embedding $\bb{H}^1\hookrightarrow \bb{L}^6$, and Gagliardo--Nirenberg inequality~\eqref{equ:gal nir uh L3}, while in the last step we used Young's inequality.
Altogether, we obtain~\eqref{equ:inpro uhn theta n} from the above estimates.

In a similar manner, we obtain the following bounds:
\begin{align*}
	\abs{\inpro{I_1}{\Delta_h \bff{\theta}^n}}
	&=
	0,
	\\
	\abs{\inpro{I_2}{\Delta_h \bff{\theta}^n}}
	&\leq
	\norm{\bff{u}_h^n}{\bb{L}^\infty} \norm{(P_h-I)\Delta \bff{u}^n}{\bb{L}^2} \norm{\Delta_h \bff{\theta}^n}{\bb{L}^2}
	\leq
	C \norm{\bff{u}_h^n}{\bb{L}^\infty}^2 h^{2(r+1)} 
	+
	\epsilon \norm{\Delta_h \bff{\theta}^n}{\bb{L}^2}^2,
	\\
	\abs{\inpro{I_3}{\Delta_h \bff{\theta}^n}}
	&\leq
	\norm{\bff{u}_h^n}{\bb{L}^6} \norm{\bff{\theta}^n+\bff{\rho}^n}{\bb{L}^3} \norm{\Delta_h \bff{\theta}^n}{\bb{L}^2}
	\leq
	Ch^{2(r+1)}
	+
	C \norm{\bff{\theta}^n}{\bb{L}^2}^2
	+
	\epsilon \norm{\nabla \bff{\theta}^n}{\bb{L}^2}^2
	+
	\epsilon \norm{\Delta_h \bff{\theta}^n}{\bb{L}^2}^2,
	\\
	\abs{\inpro{I_5}{\Delta_h \bff{\theta}^n}}
	&\leq
	\norm{\bff{u}_h^n}{\bb{L}^4} \norm{(P_h-I)\left(\abs{\bff{u}^n}^2 \bff{u}^n \right)}{\bb{L}^4} \norm{\Delta_h \bff{\theta}^n}{\bb{L}^2}
    \leq
	C h^{2(r+1)} 
	+
	\epsilon \norm{\Delta_h \bff{\theta}^n}{\bb{L}^2}^2,
	\\
	\abs{\inpro{I_6}{\Delta_h \bff{\theta}^n}}
	&\leq
	\norm{\bff{u}_h^n}{\bb{L}^6} \norm{\bff{\theta}^n+\bff{\rho}^n}{\bb{L}^3} \norm{\Delta_h \bff{\theta}^n}{\bb{L}^2}
	\leq
	Ch^{2(r+1)} + \norm{\bff{\theta}^n}{\bb{L}^2} +
	\epsilon \norm{\nabla \bff{\theta}^n}{\bb{L}^2}^2
	+
	\epsilon \norm{\Delta_h \bff{\theta}^n}{\bb{L}^2}^2,
	\\
	\abs{\inpro{I_7}{\Delta_h \bff{\theta}^n}}
	&\leq
	\norm{\bff{\theta}^n+\bff{\rho}^n}{\bb{L}^2} \norm{\bff{H}^n}{\bb{L}^\infty} \norm{\Delta_h \bff{\theta}^n}{\bb{L}^2}
	\leq
	Ch^{2(r+1)} + C\norm{\bff{\theta}^n}{\bb{L}^2}^2 
	+
	\epsilon \norm{\Delta_h \bff{\theta}^n}{\bb{L}^2}^2.
\end{align*}
Finally, for the term $\inpro{I_4}{\Delta_h \bff{\theta}^n}$, we apply similar argument as in~\eqref{equ:I5 theta n} to obtain
\begin{align}\label{equ:I5 Delta theta n}
	\abs{\inpro{I_4}{\Delta_h \bff{\theta}^n}}
	&\leq
	\norm{\bff{u}_h^n}{\bb{L}^\infty} \norm{P_h \left(|\bff{u}_h^{n}|^2 \bff{u}_h^{n}-|\bff{u}^{n}|^2 \bff{u}^{n}\right)}{\bb{L}^2} \norm{\Delta_h \bff{\theta}^n}{\bb{L}^2}
	\nonumber\\
	&\leq
	C \norm{\bff{u}_h^n}{\bb{L}^\infty} \norm{\bff{u}_h^n}{\bb{L}^6}^2 \norm{\bff{\theta}^n + \bff{\rho}^n}{\bb{L}^6} \norm{\Delta_h \bff{\theta}^n}{\bb{L}^2}
	\nonumber\\
	&\quad
	+
	C \norm{\bff{u}_h^n}{\bb{L}^\infty} \norm{\bff{u}_h^n+\bff{u}^n}{\bb{L}^6} \norm{\bff{u}^n}{\bb{L}^6} \norm{\bff{\theta}^n+\bff{\rho}^n}{\bb{L}^6} \norm{\Delta_h \bff{\theta}^n}{\bb{L}^2}
	\nonumber\\
	&\leq
	C \norm{\bff{u}_h^n}{\bb{L}^\infty} \norm{\bff{\theta}^n+\bff{\rho}^n}{\bb{L}^6} \norm{\Delta_h \bff{\theta}^n}{\bb{L}^2}
	\nonumber \\
	&\leq
	C \norm{\bff{u}_h^n}{\bb{L}^\infty}^2 h^{2(r+1)} 
    + 
    C \norm{\bff{u}_h^n}{\bb{L}^\infty}^2 \norm{\bff{\theta}^n}{\bb{H}^1}^2 
    + \epsilon \norm{\Delta_h \bff{\theta}^n}{\bb{L}^2}^2
    \nonumber \\
    &\leq
    C \norm{\bff{u}_h^n}{\bb{L}^\infty}^2 h^{2(r+1)} 
    + 
    C \norm{\bff{\theta}^n+R_h \bff{u}^n}{\bb{L}^\infty}^2 \norm{\bff{\theta}^n}{\bb{H}^1}^2 
    + \epsilon \norm{\Delta_h \bff{\theta}^n}{\bb{L}^2}^2
    \nonumber \\
    &\leq
    C \norm{\bff{u}_h^n}{\bb{L}^\infty}^2 h^{2(r+1)} 
    + 
    C \norm{\bff{\theta}^n}{\bb{L}^\infty}^2 
    + C \norm{\bff{\theta}^n}{\bb{H}^1}^2 
    + \epsilon \norm{\Delta_h \bff{\theta}^n}{\bb{L}^2}^2
     \nonumber \\
    &\leq
    C \norm{\bff{u}_h^n}{\bb{L}^\infty}^2 h^{2(r+1)} 
    + C \norm{\bff{\theta}^n}{\bb{H}^1}^2 
    + \epsilon \norm{\Delta_h \bff{\theta}^n}{\bb{L}^2}^2,
\end{align}
where in the penultimate step we used the fact that $\norm{\bff{\theta}^n}{\bb{H}^1}\leq C$ and \eqref{equ:Ritz stab u infty}, while in the last step we used \eqref{equ:disc lapl L infty} and Young's inequality.
Altogether, we obtain the inequality~\eqref{equ:inpro uhn Delta theta n}. This completes the proof of the lemma.
\end{proof}

\begin{remark}\label{rem:1d 2d}
If $d=1$ or $d=2$, then we have the Sobolev embedding $\bb{H}^1\hookrightarrow \bb{L}^8$. In these cases, one could remove the term $\norm{\bff{u}_h^n}{\bb{L}^\infty}^2$ in \eqref{equ:inpro uhn theta n} and \eqref{equ:inpro uhn Delta theta n} by estimating  $\abs{\inpro{I_4}{\bff{\theta}^n}}$ in \eqref{equ:I5 theta n} and $\abs{\inpro{I_4}{\Delta_h \bff{\theta}^n}}$ in \eqref{equ:I5 Delta theta n} differently. Indeed, we have
\begin{align*}
	\nonumber
	\abs{\inpro{I_4}{\Delta_h \bff{\theta}^n}}
	&\leq
	\norm{\bff{u}_h^n}{\bb{L}^8}
	\left( \norm{\bff{u}_h^n}{\bb{L}^8}^2 \norm{\bff{\theta}^n+\bff{\rho}^n}{\bb{L}^8}
	+
	\norm{\bff{\theta}^n}{\bb{L}^8}
	\norm{\bff{u}_h^n+\bff{u}^{n}}{\bb{L}^8}
	\norm{\bff{u}^{n}}{\bb{L}^8} \right)
	\norm{\Delta_h \bff{\theta}^{n}}{\bb{L}^2}
	\\
	&\leq
	C h^{2(r+1)}
	+
	C \norm{\bff{\theta}^n}{\bb{H}^1}^2
	+
	\epsilon \norm{\Delta_h \bff{\theta}^n}{\bb{L}^2}^2,
\end{align*}
where in the last step we also used~\eqref{equ:Ritz ineq}. Similar bound also holds for $\abs{\inpro{I_4}{\bff{\theta}^n}}$. Furthermore, if we assume more regularity on the exact solution $\bff{u}$, say $\bff{u}\in L^\infty(\bb{W}^{4,3})$, then
\begin{align*}
	\abs{\inpro{I_2}{\Delta_h \bff{\theta}^n}}
	&\leq 
	\norm{\bff{u}_h^n}{\bb{L}^6}
	\norm{(P_h-I) \Delta \bff{u}^n}{\bb{L}^3}
	\norm{\Delta_h\bff{\theta}^{n}}{\bb{L}^2}
	\leq
	Ch^{2(r+1)} \norm{\Delta \bff{u}^n}{\bb{W}^{2,3}} + \epsilon \norm{\Delta_h \bff{\theta}^{n}}{\bb{L}^2}^2.
\end{align*}
These stronger estimates allow us to avoid using \eqref{equ:Ritz stab u infty}, \eqref{equ:disc lapl L infty}, and~\eqref{equ:stab L infty}. Thus, we could remove the term $\norm{\bff{u}_h^n}{\bb{L}^\infty}^2$ in the estimates \eqref{equ:inpro uhn theta n} and \eqref{equ:inpro uhn Delta theta n} if $d=1$ or $2$, even without necessarily assuming global quasi-uniformity of the triangulation.
\end{remark}

\begin{lemma}\label{lem:inpro beta}
	Let $\epsilon >0$ be given. Let $\bff{u}_h^n$ be defined by \eqref{equ:scheme 2}. Then for any initial data $\bff{u}_h^0\in \bb{V}_h$ and $n\in \{1,2,\ldots, \lfloor T/k \rfloor\}$,
	\begin{align}
		\label{equ:inpro beta theta n}
		\big| \inpro{(\bff{\nu}^n\cdot\nabla)(\bff{u}_h^{n-1}-\bff{u}^n)}{\bff{\theta}^n} \big| 
		&\leq
		Ch^{2(r+1)} + Ck^2 + C\norm{\bff{\theta}^{n-1}}{\bb{L}^2}^2 + \epsilon\norm{\bff{\theta}^n}{\bb{H}^1}^2,
		\\
		\label{equ:inpro beta Delta theta n}
		\big| \inpro{(\bff{\nu}^n\cdot\nabla)(\bff{u}_h^{n-1}-\bff{u}^n)}{\Delta_h \bff{\theta}^n} \big| 
		&\leq 
		Ch^{2r}+Ck^2+ C\norm{\nabla\bff{\theta}^{n-1}}{\bb{L}^2}^2 +
		\epsilon \norm{\Delta_h \bff{\theta}^n}{\bb{L}^2}^2,
	\end{align}
	where $C$ depends on the coefficients of the equation, $K_r$, $T$, and $\mathscr{D}$ (but is independent of $n$, $h$ or $k$).
\end{lemma}

\begin{proof}
Writing $\bff{u}_h^{n-1}-\bff{u}^n= \bff{\theta}^{n-1}+\bff{\rho}^{n-1}-k\cdot \mathrm{d}_t \bff{u}^n$, by \eqref{equ:div thm beta1} and H\"older's inequality we have
\begin{align*}
	\big|\inpro{(\bff{\nu}^n\cdot\nabla)(\bff{u}_h^{n-1}-\bff{u}^n)}{\bff{\theta}^n} \big|
	&\leq
	\big| \inpro{(\bff{u}_h^{n-1}-\bff{u}^n)\otimes \bff{\nu}^n}{\nabla\bff{\theta}^n} \big|
	+
	\big| \inpro{(\nabla\cdot\bff{\nu}^n) (\bff{u}_h^{n-1}-\bff{u}^n)}{\bff{\theta}^n} \big|
	\\
	&\leq
	C \norm{\bff{\theta}^{n-1}+\bff{\rho}^{n-1}-k\cdot \mathrm{d}_t \bff{u}^n}{\bb{L}^2} \norm{\bff{\theta}^n}{\bb{H}^1}
	\\
	&\leq
	Ch^{2(r+1)} + Ck^2 + C\norm{\bff{\theta}^{n-1}}{\bb{L}^2}^2 + \epsilon\norm{\bff{\theta}^n}{\bb{H}^1}^2,
\end{align*}
where in the last step we used Young's inequality, \eqref{equ:norm delta un Lp}, and \eqref{equ:Ritz ineq}, thus showing \eqref{equ:inpro beta theta n}.

Next, by H\"older's and Young's inequalities,
\begin{align*}
	\big| \inpro{(\bff{\nu}^n\cdot\nabla)(\bff{u}_h^{n-1}-\bff{u}^n)}{\Delta_h \bff{\theta}^n} \big| 
	&\leq
	C\norm{\nabla\bff{\theta}^{n-1}+\nabla\bff{\rho}^{n-1}-k\cdot \nabla\mathrm{d}_t \bff{u}^n}{\bb{L}^2} \norm{\Delta_h \bff{\theta}^n}{\bb{L}^2}
	\\
	&\leq
	Ch^{2r}+Ck^2+ C\norm{\nabla\bff{\theta}^{n-1}}{\bb{L}^2}^2 +
	\epsilon \norm{\Delta_h \bff{\theta}^n}{\bb{L}^2}^2,
\end{align*}
which shows \eqref{equ:inpro beta Delta theta n}.
\end{proof}

\begin{lemma}
Let $\epsilon>0$ be given. Let $\bff{w}_h^n$ and $\bff{w}^n$ be defined by~\eqref{equ:whn} and~\eqref{equ:w}, respectively.
Then for any initial data $\bff{u}_h^0\in\bb{V}_h$ and $n\in \{1,2,\ldots, \lfloor T/k \rfloor\}$,
\begin{equation}\label{equ:wh mu pos}
	\norm{\bff{w}_h^n-\bff{w}^n}{\bb{L}^2}^2
	\leq
	Ch^{2(r+1)}
	+
	C \norm{\bff{\theta}^n}{\bb{H}^1}^2,
\end{equation} 
where $C$ depends on the coefficients of the equation, $K_r$, $T$, and $\mathscr{D}$ (but is independent of $n$, $h$ or $k$).
\end{lemma}

\begin{proof}
Note that
\begin{align*}
	\bff{w}_h^{n}-\bff{w}^{n}
	&=
	\bff{\theta}^n+ \bff{\rho}^n
	+
	P_h \left(|\bff{u}_h^{n}|^2 \bff{u}_h^{n}-|\bff{u}^{n}|^2 \bff{u}^{n}\right) 
	+
	(P_h-I) \left(|\bff{u}^{n}|^2 \bff{u}^{n}\right)
	+
	\bff{e} \left(\bff{e}\cdot (\bff{\theta}^n+\bff{\rho}^n)\right),
\end{align*}
where $I$ is the identity operator.
Noting \eqref{equ:proj approx}, \eqref{equ:Ritz ineq}, and \eqref{equ:ass 2}, we use the stability estimate \eqref{equ:stab H1} and~\eqref{equ:split cubic} to obtain
\begin{align*}
	\nonumber
	\norm{\bff{w}_h^{n}-\bff{w}^{n}}{\bb{L}^2}^2
	&\leq
	\norm{\bff{\theta}^n + \bff{\rho}^n}{\bb{L}^2}^2
	+
	C \norm{\bff{u}_h^{n}}{\bb{L}^6}^4
	\norm{\bff{\theta}^{n}+\bff{\rho}^{n}}{\bb{L}^6}^2
	+
	C
	\norm{\bff{u}_h^{n}+\bff{u}^{n}}{\bb{L}^6}^2
	\norm{\bff{u}^{n}}{\bb{L}^6}^2
	\norm{\bff{\theta}^{n}+\bff{\rho}^{n}}{\bb{L}^6}^2
	\\
	\nonumber
	&\quad
	+
	Ch^{2(r+1)} 
	+
	\norm{\bff{\theta}^n+\bff{\rho}^n}{\bb{L}^2}^2
	\\
	&\leq	
	Ch^{2(r+1)} 
	+
	C \norm{\bff{\theta}^n}{\bb{H}^1}^2.
\end{align*}
This proves \eqref{equ:wh mu pos}.
\end{proof}

We are now ready to prove an error estimate for the numerical scheme. In the proof, the following inequalities will be used:
\begin{align}
	\label{equ:dt rho L2}
	\norm{\mathrm{d}_t \bff{\rho}^n}{\bb{L}^2} &= \norm{\frac{1}{k} \int_{t_{n-1}}^{t_n} \partial_t \bff{\rho}(t) \,\dt}{\bb{L}^2} \leq \norm{\partial_t \bff{\rho}}{\bb{L}^2} \leq Ch^{r+1}
	\\
	\label{equ:dtuh minus dtun}
	\norm{\mathrm{d}_t \bff{u}^n- \partial_t \bff{u}^n}{\bb{L}^2}
	&=
	\norm{\frac{1}{2k} \int_{t_{n-1}}^{t_n} (t-t_{n-1}) \,\partial_t^2 \bff{u}(t) \, \dt}{\bb{L}^2} \leq Ck,
\end{align}
where \eqref{equ:Ritz ineq} was used in the last step of \eqref{equ:dt rho L2}, while Taylor's theorem and \eqref{equ:ass 2} were used in \eqref{equ:dtuh minus dtun}.
We begin by showing an auxiliary error estimate. In case $\bff{\nu}=\bff{0}$, the following proposition shows a superconvergence estimate for $\bff{\theta}^n$.

\begin{proposition}\label{pro:est theta no cur}
	Let $\bff{\theta}^n$ be as defined in \eqref{equ:theta rho split 2}. Then for $n\in \{0,1,\ldots,\lfloor T/k \rfloor\}$,
	\begin{align}\label{equ:theta not super}
		\norm{\bff{\theta}^n}{\bb{H}^1}^2
		+
		k \sum_{m=1}^n \norm{\Delta_h \bff{\theta}^m}{\bb{L}^2}^2
		+
		k \sum_{m=1}^n \norm{\bff{\theta}^m}{\bb{L}^\infty}^2
		&\leq
		C (h^{2r}+k^2).
	\end{align}
    If $\bff{\nu}=\bff{0}$, then we have a superconvergence estimate
    \begin{align}\label{equ:theta super}
		\norm{\bff{\theta}^n}{\bb{H}^1}^2
		+
		k \sum_{m=1}^n \norm{\Delta_h \bff{\theta}^m}{\bb{L}^2}^2
		+
		k \sum_{m=1}^n \norm{\bff{\theta}^m}{\bb{L}^\infty}^2
		&\leq
		C (h^{2(r+1)}+k^2).
	\end{align}
	In particular, if $\bff{\nu}=\bff{0}$ and $\mathscr{D}\subset \bb{R}^2$ with globally quasi-uniform triangulation, then
	\begin{equation}\label{equ:theta L infty}
		\norm{\bff{\theta}^n}{\bb{L}^\infty}^2
		\leq
		C\big(h^{2(r+1)} +k^2\big) \abs{\ln h},
	\end{equation}
	where $C$ depends on the coefficients of the equation, $K_r$, $T$, and $\mathscr{D}$ (but is independent of $n$, $h$ or $k$).
\end{proposition}

\begin{proof}
Subtracting~\eqref{equ:weak form no spin} from~\eqref{equ:scheme 2}, then applying~\eqref{equ:Delta uhn Delta un} and \eqref{equ:div thm beta1}, we have
\begin{align*}
	&\inpro{\mathrm{d}_t \bff{\theta}^n}{\bff{\chi}} + \inpro{\mathrm{d}_t \bff{\rho}^n+ \mathrm{d}_t \bff{u}^n- \partial_t \bff{u}^n}{\bff{\chi}}
	\\
	&=
	-\inpro{\bff{u}_h^n\times \bff{H}_h^n- \bff{u}^n\times \bff{H}^n}{\bff{\chi}}
	+
	\alpha \inpro{\Delta_h \bff{\theta}^n + (P_h-I) \Delta \bff{u}^n}{\bff{\chi}}
	-
	\alpha \inpro{\bff{w}_h^n-\bff{w}^n}{\bff{\chi}}
	\\
	&\quad
	+
	\beta_1 \inpro{(\bff{\nu}^n\cdot\nabla)(\bff{u}_h^{n-1}-\bff{u}^n)}{\bff{\chi}}.
\end{align*}
Now, we put $\bff{\chi}=\bff{\theta}^n- \Delta_h \bff{\theta}^n$ and rearrange the terms (noting the identity~\eqref{equ:aab}) to obtain
\begin{align}\label{equ:ineq theta H1}
	&\frac{1}{2k} \left(\norm{\bff{\theta}^n}{\bb{H}^1}^2 - \norm{\bff{\theta}^{n-1}}{\bb{H}^1}^2 \right)
	+
	\frac{1}{2k} \norm{\bff{\theta}^n-\bff{\theta}^{n-1}}{\bb{H}^1}^2
	+
	\alpha \norm{\nabla \bff{\theta}^n}{\bb{L}^2}^2
	+
	\alpha \norm{\Delta_h \bff{\theta}^n}{\bb{L}^2}^2
	\nonumber\\
	&=
	-
	\inpro{\mathrm{d}_t \bff{\rho}^n+ \mathrm{d}_t \bff{u}^n- \partial_t \bff{u}^n}{\bff{\theta}^n- \Delta_h \bff{\theta}^n} 
	-\inpro{\bff{u}_h^n\times \bff{H}_h^n- \bff{u}^n\times \bff{H}^n}{\bff{\theta}^n- \Delta_h \bff{\theta}^n}
	\nonumber\\
	&\quad
	+
	\alpha \inpro{(P_h-I)\Delta \bff{u}^n}{\bff{\theta}^n- \Delta_h \bff{\theta}^n}
	-
	\alpha \inpro{\bff{w}_h^n-\bff{w}^n}{\bff{\theta}^n- \Delta_h \bff{\theta}^n}
    +
	\beta_1 \inpro{(\bff{\nu}^n\cdot\nabla)(\bff{u}_h^{n-1}-\bff{u}^n)}{\bff{\theta}^n- \Delta_h \bff{\theta}^n}
	\nonumber\\
	&=: J_1+J_2+J_3+J_4+J_5.
\end{align}
It remains to estimate each term on the last line. For the term $J_1$, by \eqref{equ:dt rho L2} and \eqref{equ:dtuh minus dtun} we obtain
\begin{align*}
	\abs{J_1}
	&\leq
	\Big(\norm{\mathrm{d}_t \bff{\rho}^n}{\bb{L}^2} + \norm{\mathrm{d}_t \bff{u}^n- \partial_t \bff{u}^n}{\bb{L}^2}\Big) \norm{\bff{\theta}^n- \Delta_h \bff{\theta}^n}{\bb{L}^2}
	\\
	&\leq
	Ch^{2(r+1)}+ Ck^2 + \frac{\alpha}{8} \norm{\bff{\theta}^n}{\bb{L}^2}^2
	+
	\frac{\alpha}{8} \norm{\Delta_h \bff{\theta}^n}{\bb{L}^2}^2,
\end{align*}
where in the last step we also used Young's inequality.
For the next term, we apply Lemma~\ref{lem:inpro theta n} to obtain
\begin{align*}
	\abs{J_2}
	&\leq
	C \left(1+ \norm{\bff{u}_h^n}{\bb{L}^\infty}^2\right) h^{2(r+1)}  
	+ 
	C\norm{\bff{\theta}^n}{\bb{H}^1}^2
	+
	\frac{\alpha}{8} \norm{\Delta_h \bff{\theta}^n}{\bb{L}^2}^2.
\end{align*}
For the term $J_3$, by Young's inequality and~\eqref{equ:proj approx}, we have
\begin{align*}
	\abs{J_3}
	\leq
	Ch^{2(r+1)} 
	+
	\frac{\alpha}{8} \norm{\bff{\theta}^n}{\bb{L}^2}^2
	+
	\frac{\alpha}{8} \norm{\Delta_h \bff{\theta}^n}{\bb{L}^2}^2.
\end{align*}
For the term $J_4$, we use~\eqref{equ:wh mu pos} and Young's inequality to infer
\begin{align*}
	\abs{J_4} \leq Ch^{2(r+1)} + C\norm{\bff{\theta}^n}{\bb{H}^1}^2
	+
	\frac{\alpha}{8} \norm{\Delta_h \bff{\theta}^n}{\bb{L}^2}^2.
\end{align*}
For the last term, by Lemma~\ref{lem:inpro beta}. we obtain
\begin{align}\label{equ:J5 beta}
    \abs{J_5} 
    \leq
    Ch^{2r}+Ck^2+ C\norm{\bff{\theta}^{n-1}}{\bb{H}^1}^2 +
    \epsilon \norm{\bff{\theta}^n}{\bb{H}^1}^2
    +
	\epsilon \norm{\Delta_h \bff{\theta}^n}{\bb{L}^2}^2.
\end{align}
Altogether, substituting these estimates into \eqref{equ:ineq theta H1}, after rearranging the terms we obtain
\begin{align}\label{equ:theta n H1 gron}
	&\frac{1}{2k} \left(\norm{\bff{\theta}^n}{\bb{H}^1}^2 - \norm{\bff{\theta}^{n-1}}{\bb{H}^1}^2 \right)
	+
	\frac{1}{2k} \norm{\bff{\theta}^n-\bff{\theta}^{n-1}}{\bb{H}^1}^2
	+
	\alpha \norm{\nabla \bff{\theta}^n}{\bb{L}^2}^2
	+
	\frac{\alpha}{2} \norm{\Delta_h \bff{\theta}^n}{\bb{L}^2}^2
	\nonumber\\
	&\leq
	C \left(1+ \norm{\bff{u}_h^n}{\bb{L}^\infty}^2\right) h^{2r} + Ck^2 
	+ 
	C_1 \norm{\bff{\theta}^n}{\bb{H}^1}^2,
\end{align}
for some positive constant $C_1$ which depends on the coefficients of \eqref{equ:llb a} (that we do not specify in detail, but can be traced from the preceding calculations). Therefore, summing \eqref{equ:theta n H1 gron} over $m\in \{1,2,\ldots,n\}$ and noting~\eqref{equ:stab L infty} yield 
\begin{align*}
	\norm{\bff{\theta}^n}{\bb{H}^1}^2
	+
	\alpha k\sum_{m=1}^n \norm{\Delta_h \bff{\theta}^m}{\bb{L}^2}^2
	\leq
	C(1+C_\infty)h^{2r} + Ck^2 + C_1 k \sum_{m=1}^n \norm{\bff{\theta}^m}{\bb{H}^1}^2.
\end{align*}
For sufficiently small $k$ such that $C_1 k<1$, we obtain inequality~\eqref{equ:theta not super} for the first two terms on the left-hand side by the generalised discrete Gronwall lemma (Lemma~\ref{lem:disc gen gron}). For the last term in~\eqref{equ:theta not super}, the estimate follows by applying~\eqref{equ:disc lapl L infty}.

Next, if $\bff{\nu}=\bff{0}$, then the term $J_5$ is absent in \eqref{equ:ineq theta H1}, thus estimate \eqref{equ:J5 beta} is not needed. Instead of \eqref{equ:theta n H1 gron}, we then have
\begin{align*}
    &\frac{1}{2k} \left(\norm{\bff{\theta}^n}{\bb{H}^1}^2 - \norm{\bff{\theta}^{n-1}}{\bb{H}^1}^2 \right)
	+
	\frac{1}{2k} \norm{\bff{\theta}^n-\bff{\theta}^{n-1}}{\bb{H}^1}^2
	+
	\alpha \norm{\nabla \bff{\theta}^n}{\bb{L}^2}^2
	+
	\frac{\alpha}{2} \norm{\Delta_h \bff{\theta}^n}{\bb{L}^2}^2
	\nonumber\\
	&\leq
	C \left(1+ \norm{\bff{u}_h^n}{\bb{L}^\infty}^2\right) h^{2(r+1)} + Ck^2 
	+ 
	C_1 \norm{\bff{\theta}^n}{\bb{H}^1}^2,
\end{align*}
from which \eqref{equ:theta super} follows by similar argument as before.

Finally, inequality~\eqref{equ:theta L infty} follows from~\eqref{equ:theta super} and the discrete Sobolev inequality~\cite{Bre04}.
\end{proof}

We have the following main theorem for this section, which shows convergence at optimal rates in various norms.

\begin{theorem}\label{the:without spin error}
	Let $\bff{u}_h^n$ and $\bff{u}$ be the solution of \eqref{equ:scheme 2} and \eqref{equ:weak form no spin}, respectively. For $n\in \{0,1,\ldots,\lfloor T/k \rfloor\}$, 
    \begin{align*}
        \norm{\bff{u}_h^n-\bff{u}(t_n)}{\bb{H}^1}
        \leq C(h^r+k).
    \end{align*}
    If $\bff{\nu}=\bff{0}$, then for $s\in \{0,1\}$,
	\begin{align*}
		\norm{\bff{u}_h^n-\bff{u}(t_n)}{\bb{H}^s}
		&\leq 
		C(h^{r+1-s} +k),
		\\
		k \sum_{m=1}^n \norm{\bff{u}_h^n-\bff{u}(t_n)}{\bb{L}^\infty}^2 
		&\leq
		C(h^{2(r+1)} +k^2) \abs{\ln h}.
	\end{align*}
	If $\bff{\nu}=\bff{0}$ and $\mathscr{D}\subset \bb{R}^2$ with globally quasi-uniform triangulation, then
	\begin{align*}
		\norm{\bff{u}_h^n-\bff{u}(t_n)}{\bb{L}^\infty} 
		\leq
		C\big(h^{r+1} +k \big) \abs{\ln h}^{\frac12},
	\end{align*}
	where $C$ depends on the coefficients of the equation, $K_r$, $T$, and $\mathscr{D}$ (but is independent of $n$, $h$ or $k$).
\end{theorem}

\begin{proof}
	This follows from Proposition \ref{pro:est theta no cur} and the triangle inequality (noting \eqref{equ:theta rho split 2}, \eqref{equ:Ritz ineq}, and \eqref{equ:Ritz ineq L infty}).
\end{proof}

\section{Numerical experiments}\label{sec:num exp}

Numerical simulations for schemes~\eqref{equ:scheme spin} and \eqref{equ:scheme 2} are performed using the open-source package~\textsc{FEniCS}~\cite{AlnaesEtal15}. In these simulations, we take the domain $\mathscr{D}= [-0.5,0.5]^2\subset \bb{R}^2$ and use $\mathcal{P}_1$ Lagrange finite element space (corresponding to $r=1$ in \eqref{equ:Vh}). Since the exact solution of the equation is not known, we use extrapolation to verify the spatial order of convergence experimentally. To this end, let $\bff{u}_h^n$ be the finite element solution with spatial step size $h$ and time step size $k=\lfloor T/n\rfloor$. For $s\in \{0,1\}$, define the extrapolated order of convergence
\begin{equation*}
	\text{rate}_s :=  \log_2 \left[\frac{\max_n \norm{\bff{e}_{2h}}{\bb{H}^s}}{\max_n \norm{\bff{e}_{h}}{\bb{H}^s}}\right],
\end{equation*}
where $\bff{e}_h := \bff{u}_{h}^n-\bff{u}_{h/2}^n$. We expect that for both schemes, when $k$ is sufficiently small,~$\text{rate}_s \approx h^{2-s}$. To test for the temporal rate of convergence in the $\bb{L}^2$ norm, we set $k=Ch^2$ in the simulation and check that $\text{rate}_s \approx h^2$. Similarly, to test for the temporal rate of convergence in the $\bb{H}^1$ norm, we set $k=Ch$ in the simulation and check that $\text{rate}_s \approx h$. In each experiment, we plot the graph of energy vs time to show the energy evolution of the system for both numerical schemes.

\subsection{Simulation 1 (small current density)}\label{sec:simulation 1}
The coefficients in~\eqref{equ:llb a} are~$\gamma=2.2\times 10^5, \alpha=1.0\times 10^5, \beta_1=0.1, \beta_2=-0.01, \sigma=1.3\times 10^{-6}, \kappa=1.0, \mu=1.0\times 10^{-6}$, and $\lambda=1.0\times 10^{-3}$, which are of typical order of magnitude for a micromagnetic simulation. The anisotropy axis $\bff{e}=(0,1,0)^\top$. We choose the initial data $\bff{u}_0$, as in \cite{BarPro06, BenEssAyo24}, to be
\begin{equation*}
	\bff{u}_0(x,y)= \left(2x(1-2|\bff{x}|)^4, 2y(1-2|\bff{x}|)^4, \frac{(1-2|\bff{x}|)^8-|\bff{x}|^2}{(1-2|\bff{x}|)^8+|\bff{x}|^2} \right),
\end{equation*}
where $|\bff{x}|:= \sqrt{x^2+y^2}$.
Numerical simulations are performed for schemes~\eqref{equ:scheme spin} and \eqref{equ:scheme 2} with time step size $k=1.0\times 10^{-6}$ and a small current density $\bff{\nu}=(10^4,0)^\top$.

Although we establish convergence of the nonlinear scheme~\eqref{equ:scheme 2} only when $\beta_2=0$, both linear and nonlinear schemes are employed in this simulation. Snapshots of the magnetic spin field $\bff{u}$ at selected times are presented in Figure~\ref{fig:snapshots field 2d} for the linear scheme and in Figure~\ref{fig:snapshots field 2d scheme2 small} for the nonlinear scheme. The two figures exhibit close qualitative agreement, although the linear scheme achieves this with a shorter runtime. The colour scale represents the relative magnitude of the magnetisation vectors. In both simulations, a central cluster of magnetisation vectors aligned in the $+z$ direction is advected to the right by the current. This structure remains relatively stable even at elevated temperatures, illustrating its robustness and manipulability via current. After this structure exits the ferromagnetic domain, the current continues to drive a spin wave~\cite{Slo99}, which is visible as it propagates through the magnetic medium. 

Plots of $\bff{e}_h$ against $1/h$ are shown in Figures~\ref{fig:order u 1 scheme1} and~\ref{fig:order u 1 scheme2} for the linear and the nonlinear schemes, respectively. Plots of $\bff{e}_h$ against $1/h$ with $k=O(h^2)$ and $k=O(h)$ to measure temporal error of the schemes in $\bb{L}^2$ and $\bb{H}^1$ norms, respectively, are shown in Figures~\ref{fig:order k 1 scheme1} and~\ref{fig:order k 1 scheme2}. Graphs of energy vs time are plotted in Figures~\ref{fig:energy small scheme1} and \ref{fig:energy small scheme2} for the two schemes, which show the energy decay when the current density is relatively small.

\begin{figure}[!htb]
	\centering
	\begin{subfigure}[b]{0.27\textwidth}
		\centering
		\includegraphics[width=\textwidth]{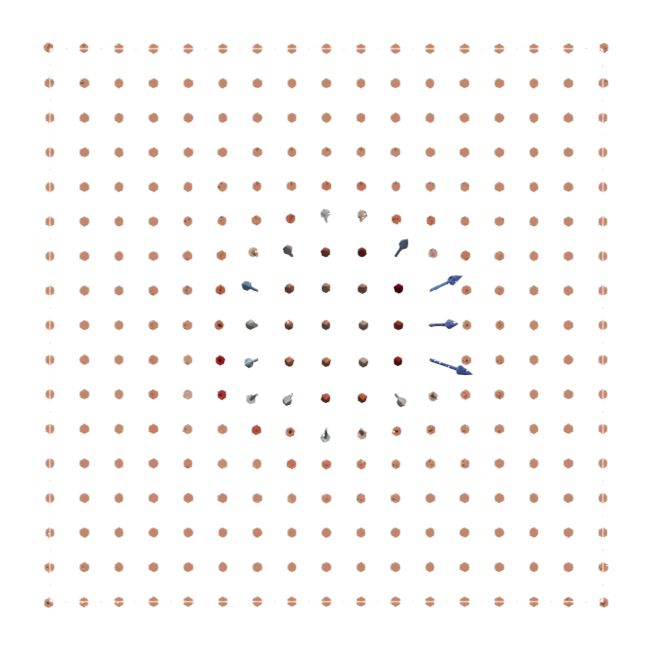}
		\caption{$t=0$}
	\end{subfigure}
	\begin{subfigure}[b]{0.27\textwidth}
		\centering
		\includegraphics[width=\textwidth]{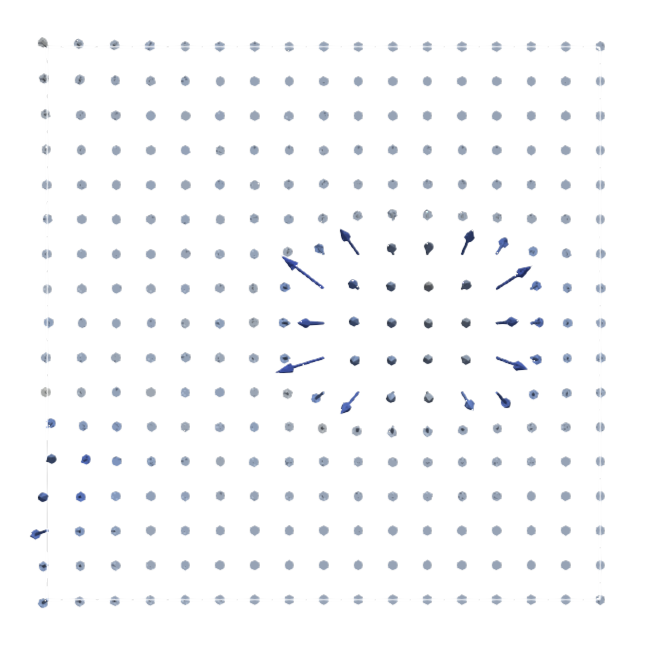}
		\caption{$t=1\times 10^{-4}$}
	\end{subfigure}
	\begin{subfigure}[b]{0.27\textwidth}
		\centering
		\includegraphics[width=\textwidth]{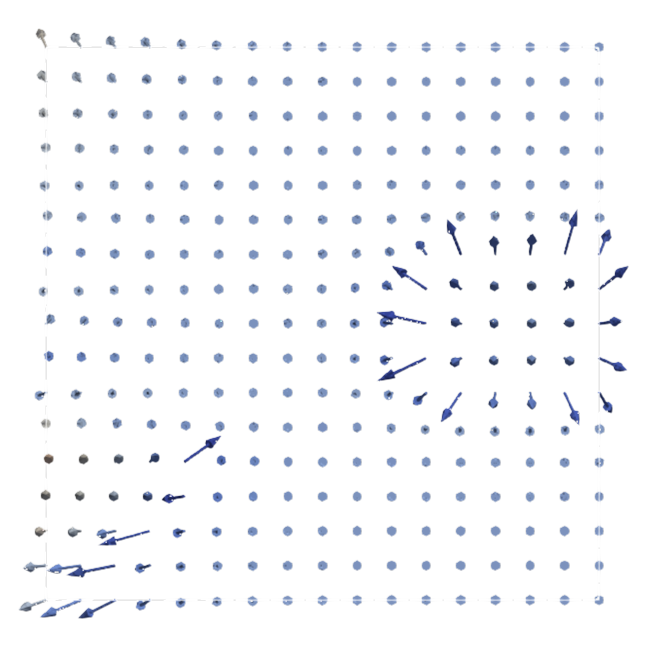}
		\caption{$t=2 \times 10^{-4}$}
	\end{subfigure}
	\begin{subfigure}[b]{0.1\textwidth}
		\centering
		\includegraphics[width=\textwidth]{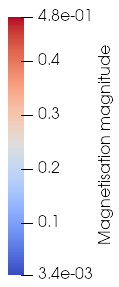}
	\end{subfigure}
	\begin{subfigure}[b]{0.27\textwidth}
		\centering
		\includegraphics[width=\textwidth]{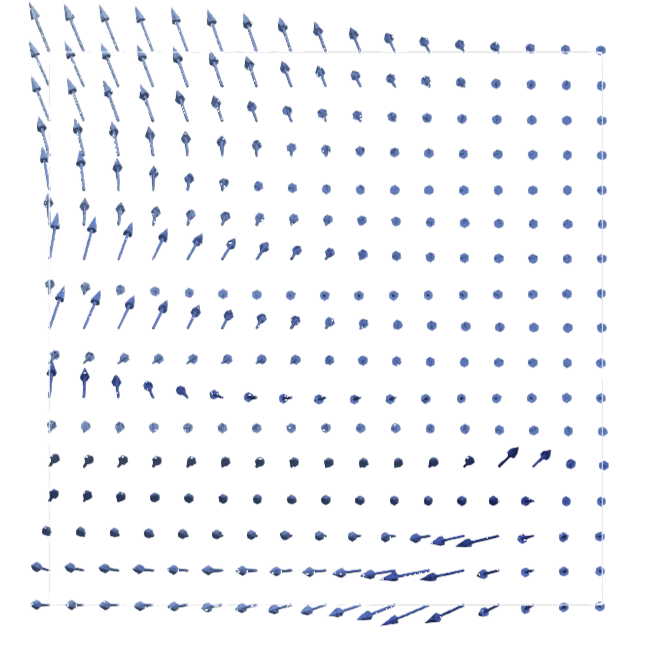}
		\caption{$t=5\times 10^{-4}$}
	\end{subfigure}
	\begin{subfigure}[b]{0.27\textwidth}
		\centering
		\includegraphics[width=\textwidth]{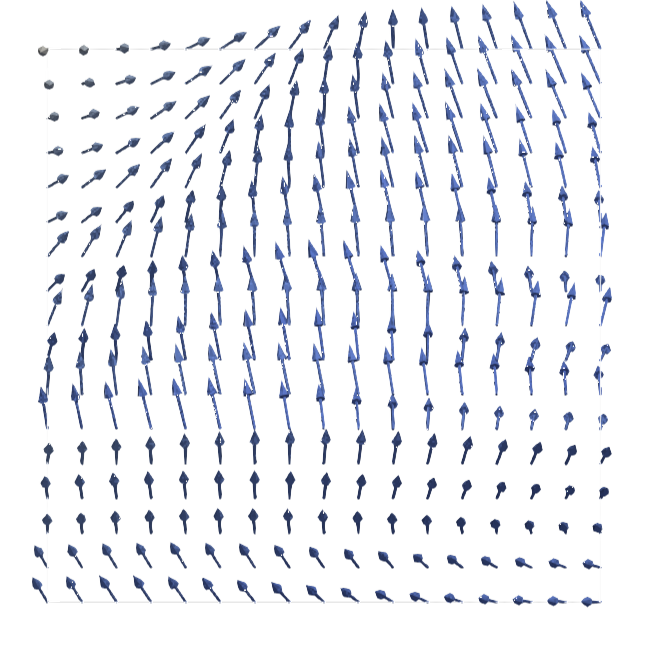}
		\caption{$t=1\times 10^{-3}$}
	\end{subfigure}
	\begin{subfigure}[b]{0.27\textwidth}
		\centering
		\includegraphics[width=\textwidth]{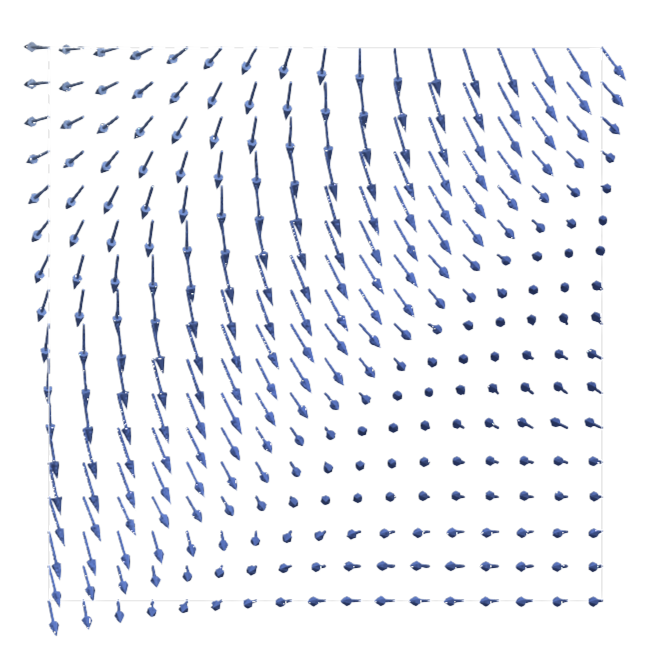}
		\caption{$t=2\times 10^{-3}$}
	\end{subfigure}
	\begin{subfigure}[b]{0.1\textwidth}
		\centering
		\includegraphics[width=\textwidth]{legend_exp1a.png}
	\end{subfigure}
	\caption{Snapshots of the spin field $\bff{u}$ (projected onto $\bb{R}^2$) for the \emph{linear} scheme~\eqref{equ:scheme spin} in Simulation 1.}
	\label{fig:snapshots field 2d}
\end{figure}

\begin{figure}[!htb]
	\centering
	\begin{subfigure}[b]{0.27\textwidth}
		\centering
		\includegraphics[width=\textwidth]{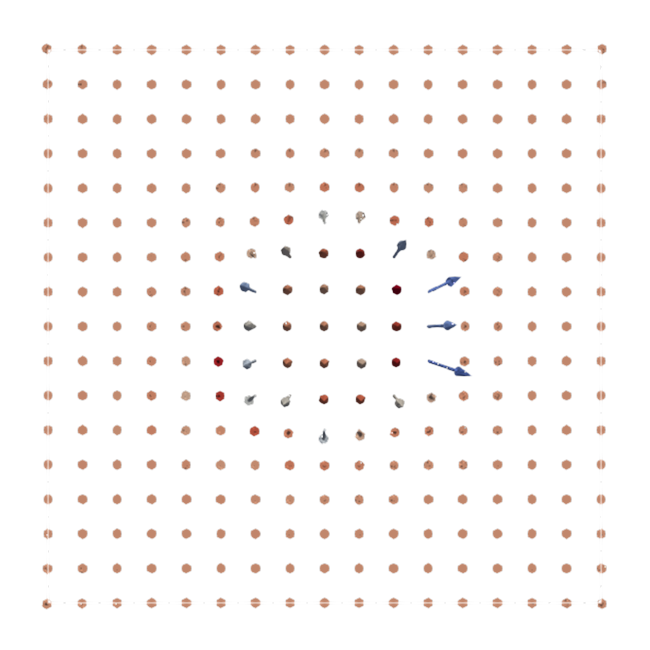}
		\caption{$t=0$}
	\end{subfigure}
	\begin{subfigure}[b]{0.27\textwidth}
		\centering
		\includegraphics[width=\textwidth]{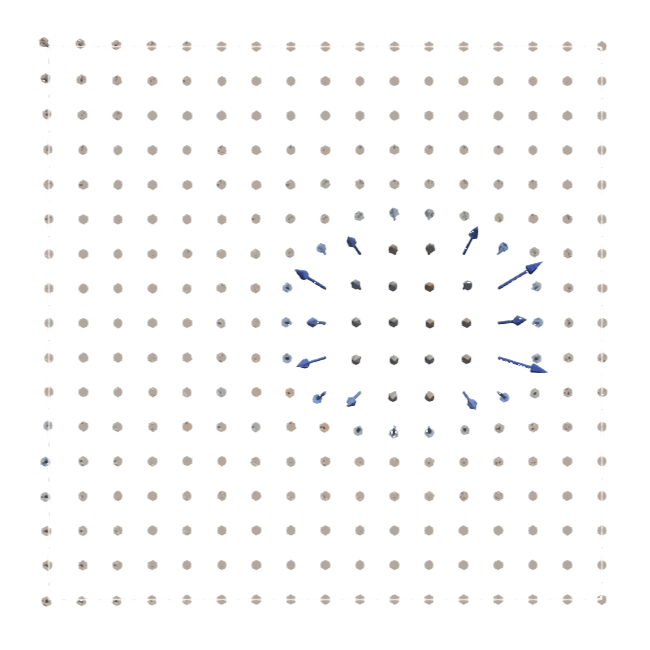}
		\caption{$t=1\times 10^{-4}$}
	\end{subfigure}
	\begin{subfigure}[b]{0.27\textwidth}
		\centering
		\includegraphics[width=\textwidth]{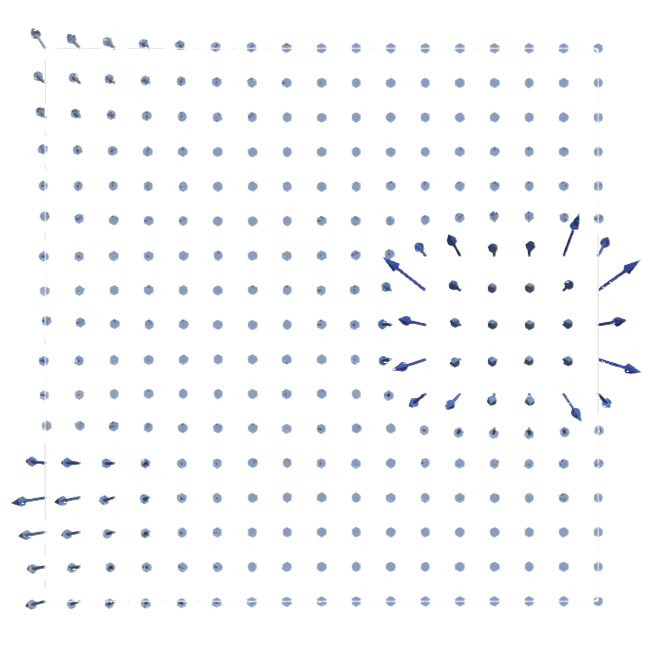}
		\caption{$t=2 \times 10^{-4}$}
	\end{subfigure}
	\begin{subfigure}[b]{0.1\textwidth}
		\centering
		\includegraphics[width=\textwidth]{legend_exp1a.png}
	\end{subfigure}
	\begin{subfigure}[b]{0.27\textwidth}
		\centering
		\includegraphics[width=\textwidth]{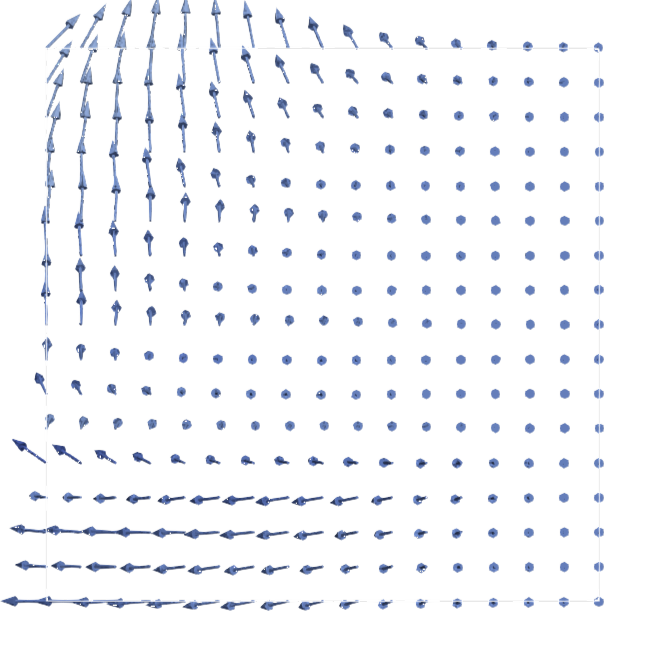}
		\caption{$t=5\times 10^{-4}$}
	\end{subfigure}
	\begin{subfigure}[b]{0.27\textwidth}
		\centering
		\includegraphics[width=\textwidth]{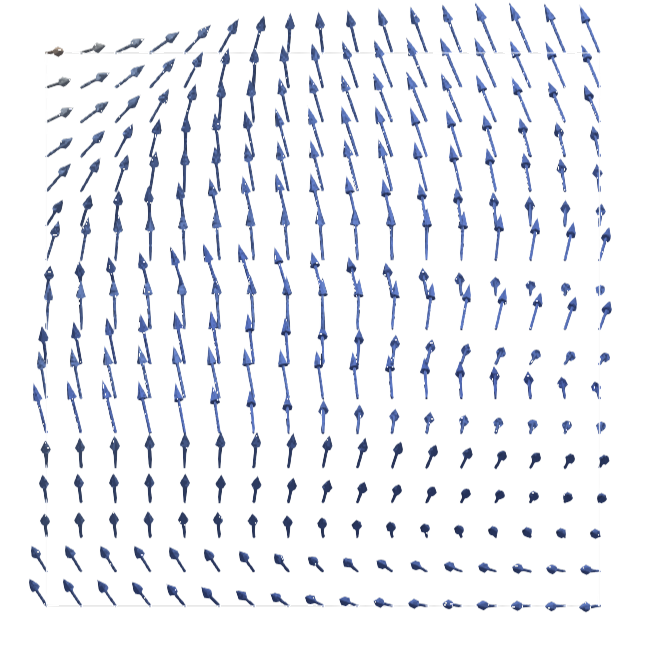}
		\caption{$t=1\times 10^{-3}$}
	\end{subfigure}
	\begin{subfigure}[b]{0.27\textwidth}
		\centering
		\includegraphics[width=\textwidth]{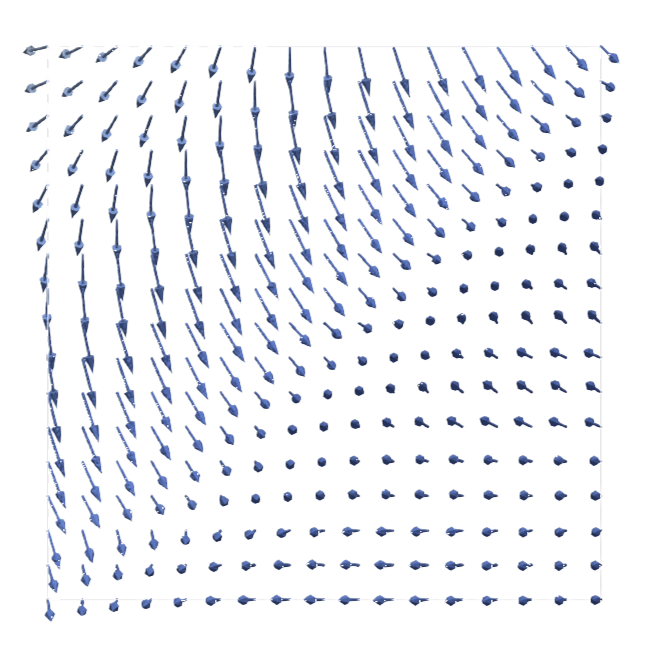}
		\caption{$t=2\times 10^{-3}$}
	\end{subfigure}
	\begin{subfigure}[b]{0.1\textwidth}
		\centering
		\includegraphics[width=\textwidth]{legend_exp1a.png}
	\end{subfigure}
	\caption{Snapshots of the spin field $\bff{u}$ (projected onto $\bb{R}^2$) for the \emph{nonlinear} scheme~\eqref{equ:scheme 2} in Simulation 1.}
	\label{fig:snapshots field 2d scheme2 small}
\end{figure}

\begin{figure}[!htb]
	\begin{subfigure}[b]{0.45\textwidth}
		\centering
		\begin{tikzpicture}
			\begin{axis}[
				title=Plot of $\bff{e}_h$ against $1/h$,
				height=1\textwidth,
				width=1\textwidth,
				xlabel= $1/h$,
				ylabel= $\bff{e}_h$,
				xmode=log,
				ymode=log,
				legend pos=south west,
				legend cell align=left,
				]
				\addplot+[mark=*,blue] coordinates {(4,2.51)(8,2.08)(16,1.48)(32,0.87)(64,0.42)(128,0.209)};
				\addplot+[mark=*,red] coordinates {(4,0.169)(8,0.0635)(16,0.022)(32,0.00585)(64,0.00148)(128,0.00037)};
				\addplot+[dashed,no marks,blue,domain=40:130]{11/x};
				\addplot+[dashed,no marks,red,domain=40:130]{2.5/x^2};
				\legend{\small{$\max_n \norm{\bff{e}_h}{\bb{H}_0^1}$}, \small{$\max_n \norm{\bff{e}_h}{\bb{L}^2}$}, \small{order 1 line}, \small{order 2 line}}
			\end{axis}
		\end{tikzpicture}
		\caption{Spatial error order of $\bff{u}$ for the \emph{linear} scheme~\eqref{equ:scheme spin} in Simulation 1.}
		\label{fig:order u 1 scheme1}
	\end{subfigure}
	\hspace{1em}
	\begin{subfigure}[b]{0.45\textwidth}
		\centering
		\begin{tikzpicture}
			\begin{axis}[
				title=Plot of $\bff{e}_h$ against $1/h$,
				height=1\textwidth,
				width=1\textwidth,
				xlabel= $1/h$,
				ylabel= $\bff{e}_h$,
				xmode=log,
				ymode=log,
				legend pos=south west,
				legend cell align=left,
				]
				\addplot+[mark=*,blue] coordinates {(4,2.06)(8,1.27)(16,0.67)(32,0.339)(64,0.17)(128,0.085)};
				\addplot+[mark=*,red] coordinates {(4,0.0653)(8,0.0226)(16,0.0061)(32,0.00155)(64,0.00039)(128,0.0001)};
				\addplot+[dashed,no marks,blue,domain=40:130]{5/x};
				\addplot+[dashed,no marks,red,domain=40:130]{0.7/x^2};
				\legend{\small{$\max_n \norm{\bff{e}_h}{\bb{H}_0^1}$}, \small{$\max_n \norm{\bff{e}_h}{\bb{L}^2}$}, \small{order 1 line}, \small{order 2 line}}
			\end{axis}
		\end{tikzpicture}
		\caption{Spatial error order of $\bff{u}$ for the \emph{nonlinear} scheme~\eqref{equ:scheme 2} in Simulation 1.}
		\label{fig:order u 1 scheme2}
	\end{subfigure}
\end{figure}

\begin{figure}[!htb]
	\begin{subfigure}[b]{0.45\textwidth}
		\centering
		\begin{tikzpicture}
			\begin{axis}[
				title=Plot of $\bff{e}_h$ against $1/h$,
				height=1\textwidth,
				width=1\textwidth,
				xlabel= $1/h$,
				ylabel= $\bff{e}_h$,
				xmode=log,
				ymode=log,
				legend pos=south west,
				legend cell align=left,
				]
				\addplot+[mark=*,blue] coordinates {(4,0.63)(8,1)(16,0.56)(32,0.23)(64,0.11)};
				\addplot+[mark=*,red] coordinates {(4,0.06)(8,0.022)(16,0.006)(32,0.00153)(64,0.00039)};
				\addplot+[dashed,no marks,blue,domain=20:70]{4/x};
				\addplot+[dashed,no marks,red,domain=20:70]{0.7/x^2};
				\legend{\small{$\max_n \norm{\bff{e}_h}{\bb{H}_0^1}$}, \small{$\max_n \norm{\bff{e}_h}{\bb{L}^2}$}, \small{order 1 line}, \small{order 2 line}}
			\end{axis}
		\end{tikzpicture}
		\caption{Error order of $\bff{u}$ for the \emph{linear} scheme~\eqref{equ:scheme spin} in Simulation 1 in $\bb{L}^2$ norm with $k=0.01h^2$, and in $\bb{H}^1_0$ norm with $k=0.01h$.}
		\label{fig:order k 1 scheme1}
	\end{subfigure}
	\hspace{1em}
	\begin{subfigure}[b]{0.45\textwidth}
		\centering
		\begin{tikzpicture}
			\begin{axis}[
				title=Plot of $\bff{e}_h$ against $1/h$,
				height=1\textwidth,
				width=1\textwidth,
				xlabel= $1/h$,
				ylabel= $\bff{e}_h$,
				xmode=log,
				ymode=log,
				legend pos=south west,
				legend cell align=left,
				]
				\addplot+[mark=*,blue] coordinates {(4,5.7)(8,5.2)(16,4)(32,1.5)(64,0.7)};
				\addplot+[mark=*,red] coordinates {(4,0.096)(8,0.038)(16,0.011)(32,0.0029)(64,0.0007)};
				\addplot+[dashed,no marks,blue,domain=20:70]{17/x};
				\addplot+[dashed,no marks,red,domain=20:70]{1/x^2};
				\legend{\small{$\max_n \norm{\bff{e}_h}{\bb{H}_0^1}$}, \small{$\max_n \norm{\bff{e}_h}{\bb{L}^2}$}, \small{order 1 line}, \small{order 2 line}}
			\end{axis}
		\end{tikzpicture}
		\caption{Error order of $\bff{u}$ for the \emph{nonlinear} scheme~\eqref{equ:scheme 2} in Simulation 1 in $\bb{L}^2$ norm with $k=0.01h^2$, and in $\bb{H}^1_0$ norm with $k=0.01h$.}
		\label{fig:order k 1 scheme2}
	\end{subfigure}
\end{figure}

\begin{figure}[!htb]
	\begin{subfigure}[b]{0.46\textwidth}
		\centering
		\includegraphics[width=\textwidth]{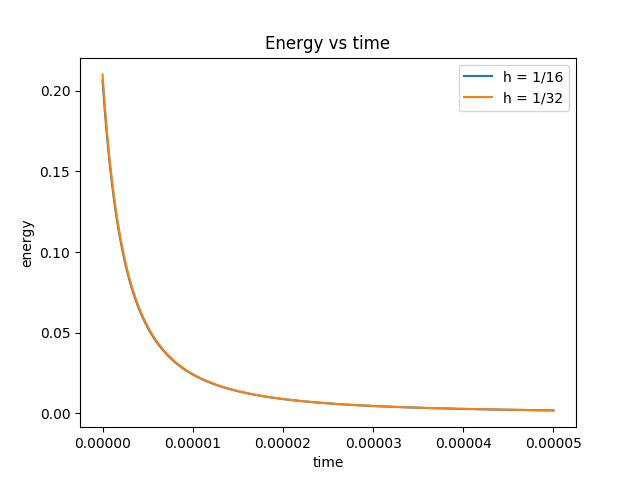}
		\caption{Plot of energy vs time for the linear scheme~\eqref{equ:scheme spin} in Simulation~1.}
		\label{fig:energy small scheme1}
	\end{subfigure}
	\hspace{1em}
	\begin{subfigure}[b]{0.46\textwidth}
		\centering
		\includegraphics[width=\textwidth]{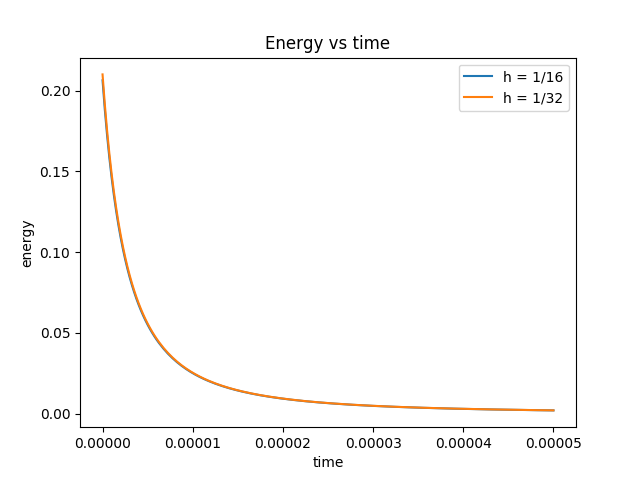}
		\caption{Plot of energy vs time for the nonlinear scheme~\eqref{equ:scheme 2} in Simulation~1.}
		\label{fig:energy small scheme2}
	\end{subfigure}
\end{figure}

\subsection{Simulation 2 (large current density)}
We repeat Simulation 1 using a larger current density $\bff{\nu}=(2\times 10^6,0)^\top$.
Snapshots of the magnetic spin field $\bff{u}$ at selected times are shown in Figure~\ref{fig:snapshots field 2d large} for the linear scheme and in Figure~\ref{fig:snapshots field 2d scheme2 large} for the nonlinear scheme. The colour scale represents the relative magnitude of the magnetisation vectors. As in Simulation 1, both figures display a central microstructure that is advected by the current, now at a higher speed due to the increased current density. The stronger current also generates a more energetic spin wave that propagates through the magnet.

Plots of $\bff{e}_h$ against $1/h$ are shown in Figures~\ref{fig:order u 1 large scheme1} and~\ref{fig:order u 1 large scheme2} for the linear and the nonlinear schemes, respectively. Plots of $\bff{e}_h$ against $1/h$ with $k=O(h^2)$ and $k=O(h)$ to measure temporal error of the schemes in $\bb{L}^2$ and $\bb{H}^1$ norms, respectively, are shown in Figures~\ref{fig:order k 1 large scheme1} and~\ref{fig:order k 1 large scheme2}. Graphs of energy vs time are plotted in Figures~\ref{fig:energy large scheme1} and \ref{fig:energy large scheme2} for the two schemes, which show an initial energy decay as ultrafast demagnetisation occurs, then a rise in energy due to a large injected current.

\begin{figure}[!htb]
	\centering
	\begin{subfigure}[b]{0.27\textwidth}
		\centering
		\includegraphics[width=\textwidth]{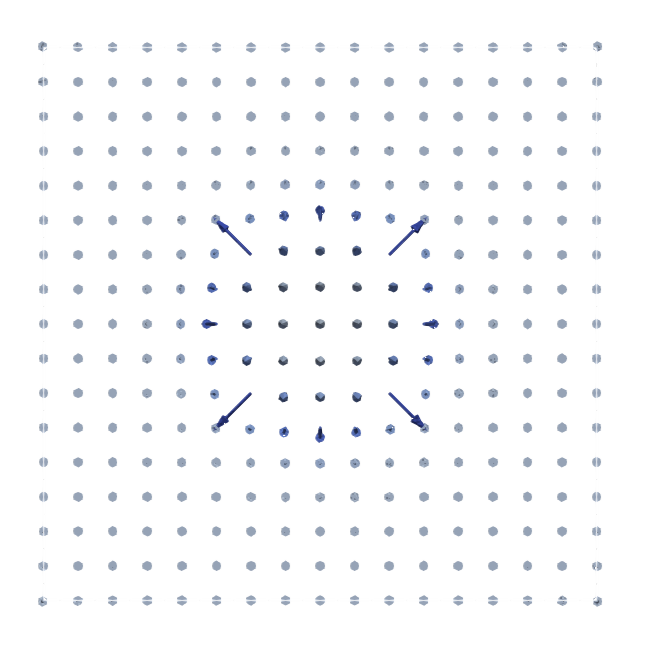}
		\caption{$t=0$}
	\end{subfigure}
	\begin{subfigure}[b]{0.27\textwidth}
		\centering
		\includegraphics[width=\textwidth]{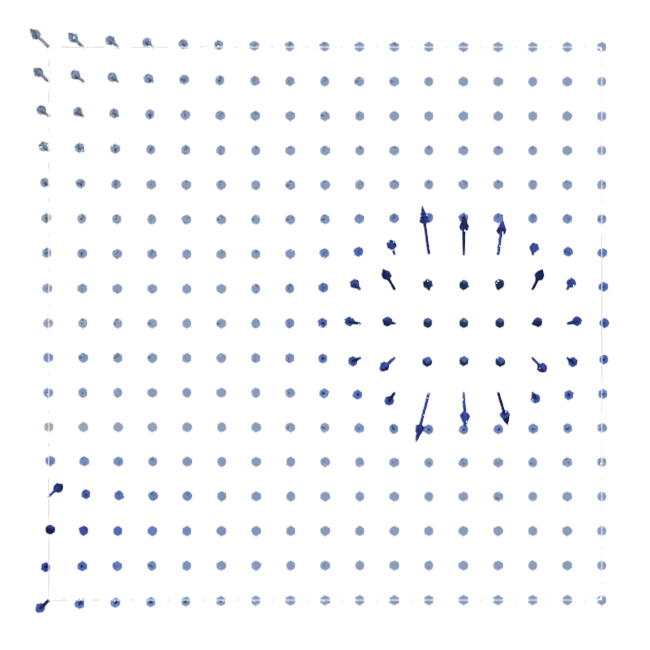}
		\caption{$t=2\times 10^{-6}$}
	\end{subfigure}
	\begin{subfigure}[b]{0.27\textwidth}
		\centering
		\includegraphics[width=\textwidth]{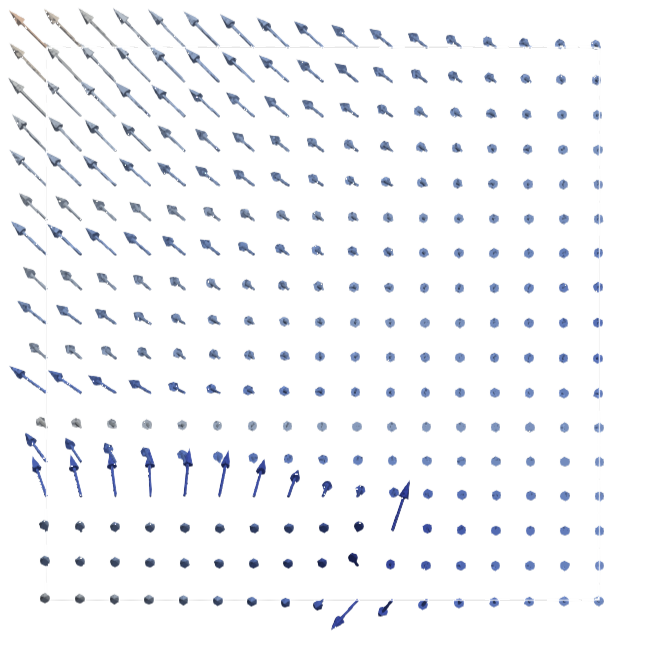}
		\caption{$t=5 \times 10^{-6}$}
	\end{subfigure}
	\begin{subfigure}[b]{0.1\textwidth}
		\centering
		\includegraphics[width=\textwidth]{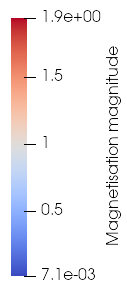}
	\end{subfigure}
	\begin{subfigure}[b]{0.27\textwidth}
		\centering
		\includegraphics[width=\textwidth]{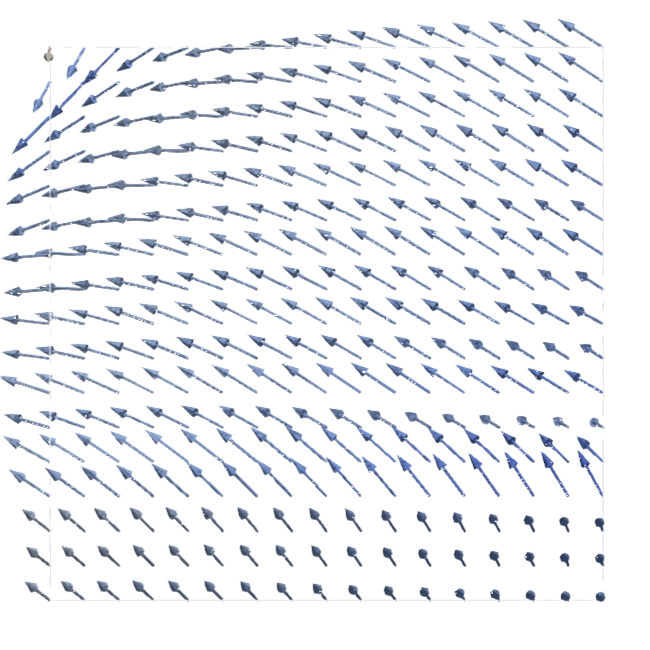}
		\caption{$t=1\times 10^{-5}$}
	\end{subfigure}
	\begin{subfigure}[b]{0.27\textwidth}
		\centering
		\includegraphics[width=\textwidth]{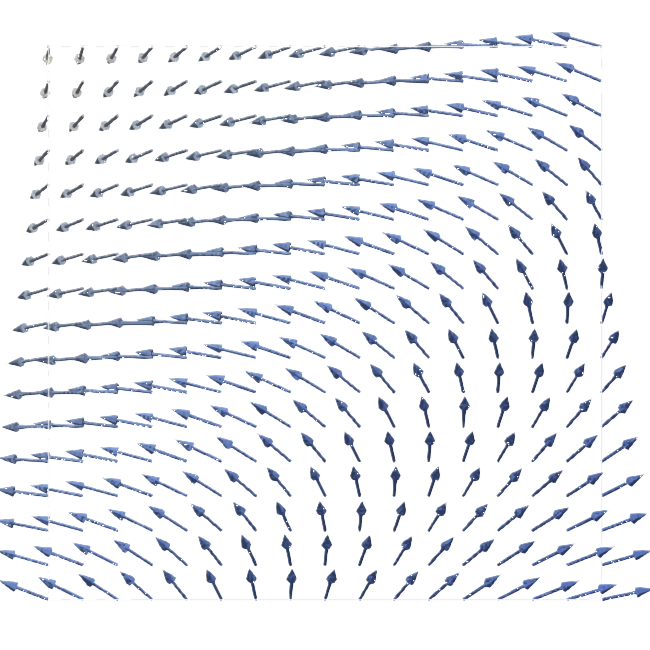}
		\caption{$t=4\times 10^{-5}$}
	\end{subfigure}
	\begin{subfigure}[b]{0.27\textwidth}
		\centering
		\includegraphics[width=\textwidth]{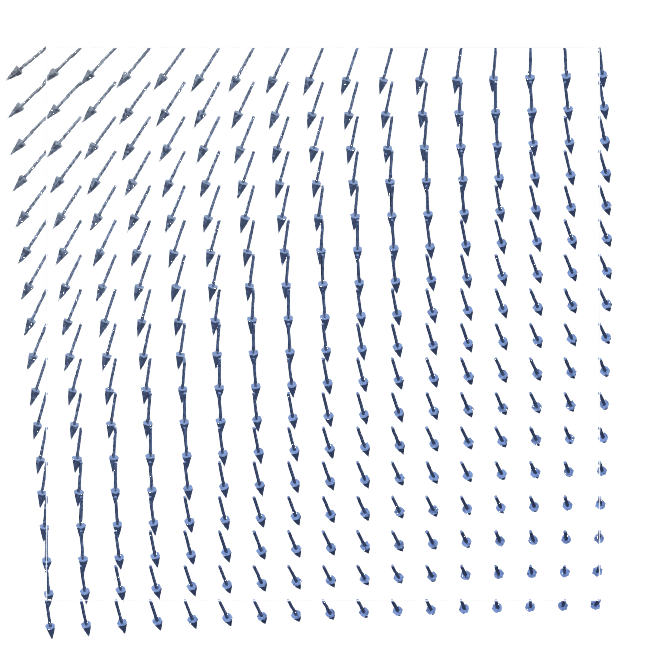}
		\caption{$t=5\times 10^{-5}$}
	\end{subfigure}
	\begin{subfigure}[b]{0.1\textwidth}
		\centering
		\includegraphics[width=\textwidth]{exp1b_legend.png}
	\end{subfigure}
	\caption{Snapshots of the spin field $\bff{u}$ (projected onto $\bb{R}^2$) for the \emph{linear} scheme~\eqref{equ:scheme spin} in Simulation 2.}
	\label{fig:snapshots field 2d large}
\end{figure}

\begin{figure}[!htb]
	\centering
	\begin{subfigure}[b]{0.27\textwidth}
		\centering
		\includegraphics[width=\textwidth]{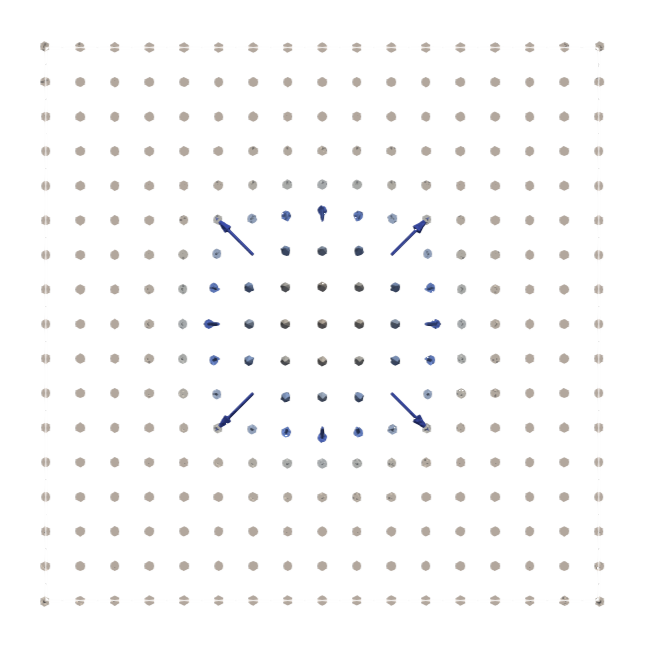}
		\caption{$t=0$}
	\end{subfigure}
	\begin{subfigure}[b]{0.27\textwidth}
		\centering
		\includegraphics[width=\textwidth]{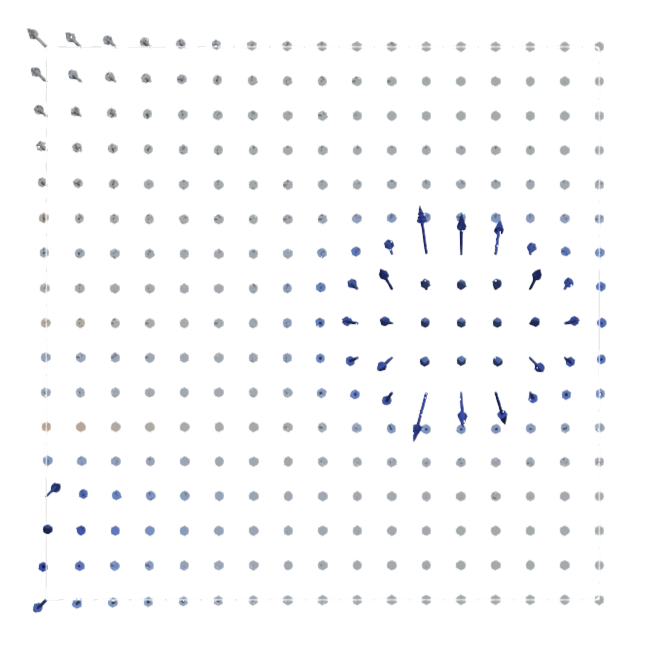}
		\caption{$t=1\times 10^{-4}$}
	\end{subfigure}
	\begin{subfigure}[b]{0.27\textwidth}
		\centering
		\includegraphics[width=\textwidth]{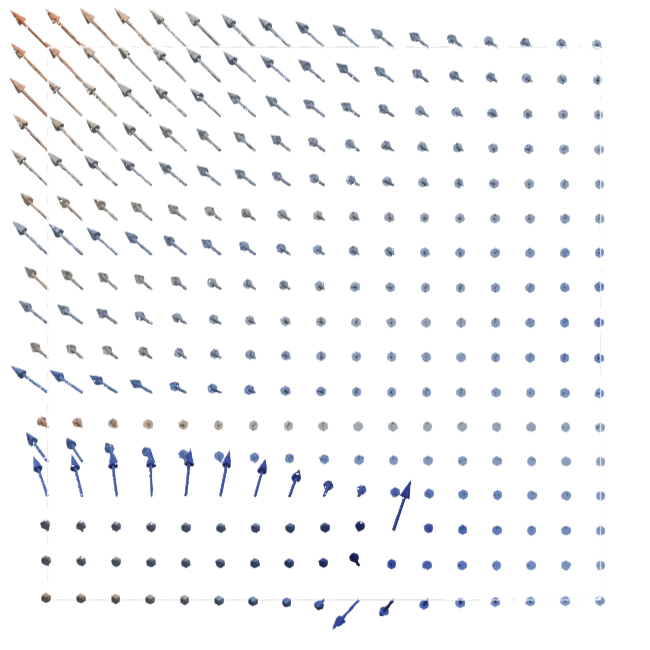}
		\caption{$t=2 \times 10^{-4}$}
	\end{subfigure}
	\begin{subfigure}[b]{0.1\textwidth}
		\centering
		\includegraphics[width=\textwidth]{exp1b_legend.png}
	\end{subfigure}
	\begin{subfigure}[b]{0.27\textwidth}
		\centering
		\includegraphics[width=\textwidth]{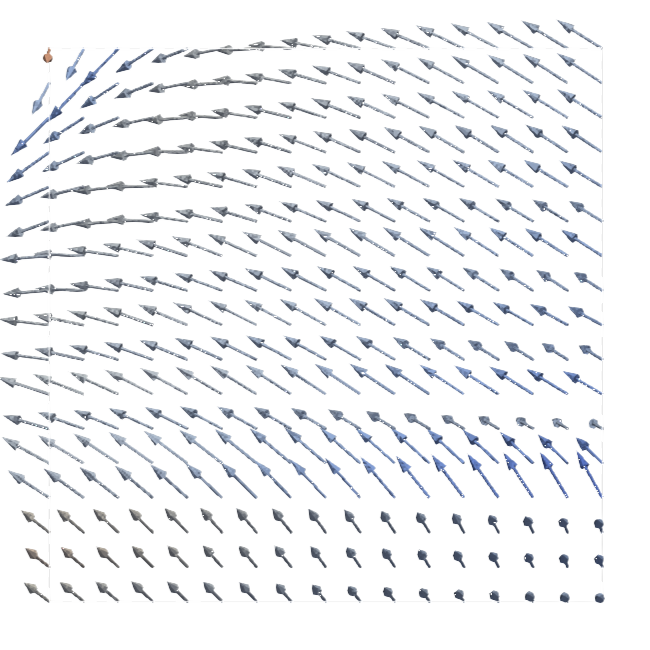}
		\caption{$t=5\times 10^{-4}$}
	\end{subfigure}
	\begin{subfigure}[b]{0.27\textwidth}
		\centering
		\includegraphics[width=\textwidth]{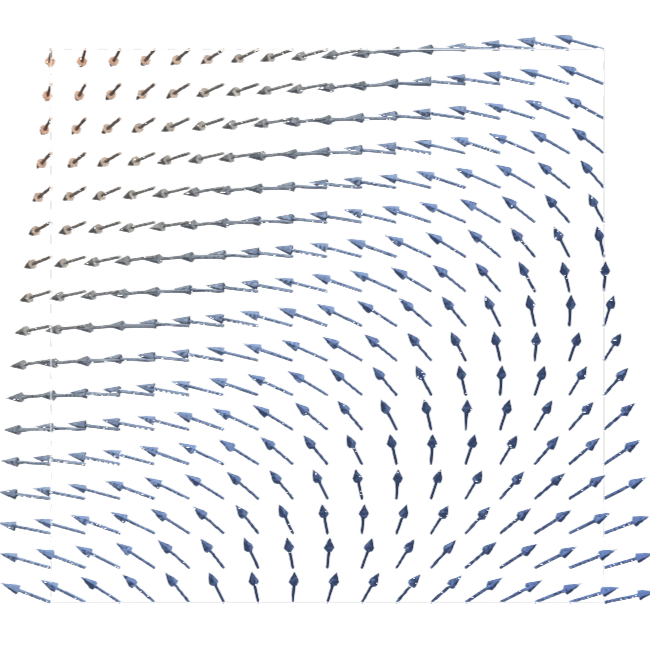}
		\caption{$t=1\times 10^{-3}$}
	\end{subfigure}
	\begin{subfigure}[b]{0.27\textwidth}
		\centering
		\includegraphics[width=\textwidth]{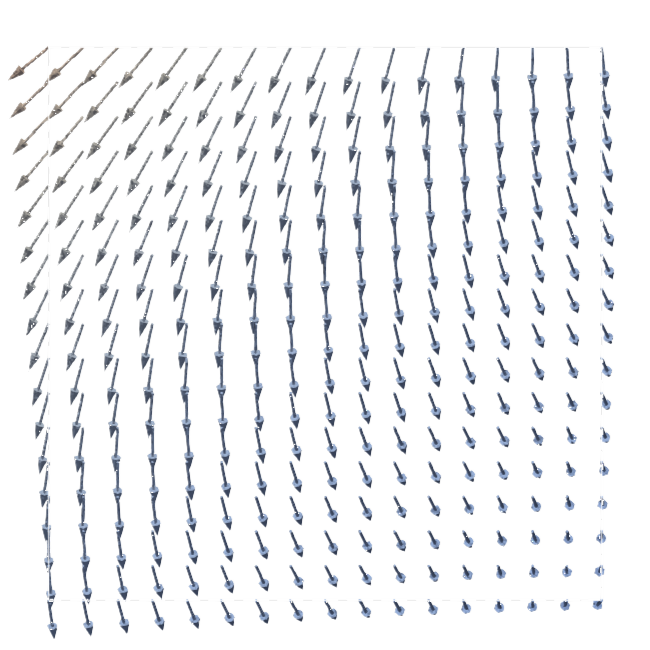}
		\caption{$t=5\times 10^{-3}$}
	\end{subfigure}
	\begin{subfigure}[b]{0.1\textwidth}
		\centering
		\includegraphics[width=\textwidth]{exp1b_legend.png}
	\end{subfigure}
	\caption{Snapshots of the spin field $\bff{u}$ (projected onto $\bb{R}^2$) for the \emph{nonlinear} scheme~\eqref{equ:scheme 2} in Simulation 2.}
	\label{fig:snapshots field 2d scheme2 large}
\end{figure}

\begin{figure}[!htb]
	\begin{subfigure}[b]{0.45\textwidth}
		\centering
		\begin{tikzpicture}
			\begin{axis}[
				title=Plot of $\bff{e}_h$ against $1/h$,
				height=1\textwidth,
				width=1\textwidth,
				xlabel= $1/h$,
				ylabel= $\bff{e}_h$,
				xmode=log,
				ymode=log,
				legend pos=south west,
				legend cell align=left,
				]
				\addplot+[mark=*,blue] coordinates {(4,2.53)(8,1.33)(16,0.655)(32,0.332)(64,0.165)};
				\addplot+[mark=*,red] coordinates {(4,0.083)(8,0.0251)(16,0.006)(32,0.00154)(64,0.00039)};
				\addplot+[dashed,no marks,blue,domain=20:70]{5/x};
				\addplot+[dashed,no marks,red,domain=20:70]{0.6/x^2};
				\legend{\small{$\max_n \norm{\bff{e}_h}{\bb{H}_0^1}$}, \small{$\max_n \norm{\bff{e}_h}{\bb{L}^2}$}, \small{order 1 line}, \small{order 2 line}}
			\end{axis}
		\end{tikzpicture}
		\caption{Spatial error order of $\bff{u}$ for the \emph{linear} scheme~\eqref{equ:scheme spin} in Simulation 2.}
		\label{fig:order u 1 large scheme1}
	\end{subfigure}
	\hspace{1em}
	\begin{subfigure}[b]{0.45\textwidth}
		\centering
		\begin{tikzpicture}
			\begin{axis}[
				title=Plot of $\bff{e}_h$ against $1/h$,
				height=1\textwidth,
				width=1\textwidth,
				xlabel= $1/h$,
				ylabel= $\bff{e}_h$,
				xmode=log,
				ymode=log,
				legend pos=south west,
				legend cell align=left,
				]
				\addplot+[mark=*,blue] coordinates {(4,2.09)(8,1.28)(16,0.665)(32,0.336)(64,0.167)};
				\addplot+[mark=*,red] coordinates {(4,0.089)(8,0.031)(16,0.0078)(32,0.0019)(64,0.0005)};
				\addplot+[dashed,no marks,blue,domain=20:70]{5/x};
				\addplot+[dashed,no marks,red,domain=20:70]{0.8/x^2};
				\legend{\small{$\max_n \norm{\bff{e}_h}{\bb{H}_0^1}$}, \small{$\max_n \norm{\bff{e}_h}{\bb{L}^2}$}, \small{order 1 line}, \small{order 2 line}}
			\end{axis}
		\end{tikzpicture}
		\caption{Spatial error order of $\bff{u}$ for the \emph{nonlinear} scheme~\eqref{equ:scheme 2} in Simulation 2.}
		\label{fig:order u 1 large scheme2}
	\end{subfigure}
\end{figure}

\begin{figure}[!htb]
	\begin{subfigure}[b]{0.45\textwidth}
		\centering
		\begin{tikzpicture}
			\begin{axis}[
				title=Plot of $\bff{e}_h$ against $1/h$,
				height=1\textwidth,
				width=1\textwidth,
				xlabel= $1/h$,
				ylabel= $\bff{e}_h$,
				xmode=log,
				ymode=log,
				legend pos=south west,
				legend cell align=left,
				]
				\addplot+[mark=*,blue] coordinates {(4,1.86)(8,1.4)(16,0.8)(32,0.45)(64,0.22)};
				\addplot+[mark=*,red] coordinates {(4,0.07)(8,0.023)(16,0.0061)(32,0.00156)(64,0.00039)};
				\addplot+[dashed,no marks,blue,domain=20:70]{6/x};
				\addplot+[dashed,no marks,red,domain=20:70]{0.7/x^2};
				\legend{\small{$\max_n \norm{\bff{e}_h}{\bb{H}_0^1}$}, \small{$\max_n \norm{\bff{e}_h}{\bb{L}^2}$}, \small{order 1 line}, \small{order 2 line}}
			\end{axis}
		\end{tikzpicture}
		\caption{Error order of $\bff{u}$ for the \emph{linear} scheme~\eqref{equ:scheme spin} in Simulation 2 in $\bb{L}^2$ norm with $k=0.01h^2$, and in $\bb{H}^1_0$ norm with $k=0.01h$.}
		\label{fig:order k 1 large scheme1}
	\end{subfigure}
	\hspace{1em}
	\begin{subfigure}[b]{0.45\textwidth}
		\centering
		\begin{tikzpicture}
			\begin{axis}[
				title=Plot of $\bff{e}_h$ against $1/h$,
				height=1\textwidth,
				width=1\textwidth,
				xlabel= $1/h$,
				ylabel= $\bff{e}_h$,
				xmode=log,
				ymode=log,
				legend pos=south west,
				legend cell align=left,
				]
				\addplot+[mark=*,blue] coordinates {(4,6.0)(8,5.0)(16,1.8)(32,0.9)(64,0.45)};
				\addplot+[mark=*,red] coordinates {(4,0.1)(8,0.03)(16,0.006)(32,0.0016)(64,0.0004)};
				\addplot+[dashed,no marks,blue,domain=20:70]{9/x};
				\addplot+[dashed,no marks,red,domain=20:70]{0.5/x^2};
				\legend{\small{$\max_n \norm{\bff{e}_h}{\bb{H}_0^1}$}, \small{$\max_n \norm{\bff{e}_h}{\bb{L}^2}$}, \small{order 1 line}, \small{order 2 line}}
			\end{axis}
		\end{tikzpicture}
		\caption{Error order of $\bff{u}$ for the \emph{nonlinear} scheme~\eqref{equ:scheme 2} in Simulation 2 in $\bb{L}^2$ norm with $k=0.01h^2$, and in $\bb{H}^1_0$ norm with $k=0.01h$.}
		\label{fig:order k 1 large scheme2}
	\end{subfigure}
\end{figure}

\begin{figure}[!htb]
	\begin{subfigure}[b]{0.46\textwidth}
		\centering
		\includegraphics[width=\textwidth]{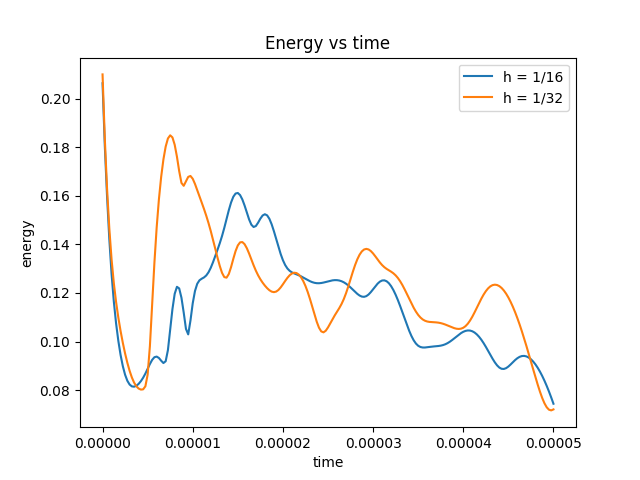}
		\caption{Plot of energy vs time for the \emph{linear} scheme~\eqref{equ:scheme spin} in Simulation~2.}
		\label{fig:energy large scheme1}
	\end{subfigure}
	\hspace{1em}
	\begin{subfigure}[b]{0.46\textwidth}
		\centering
		\includegraphics[width=\textwidth]{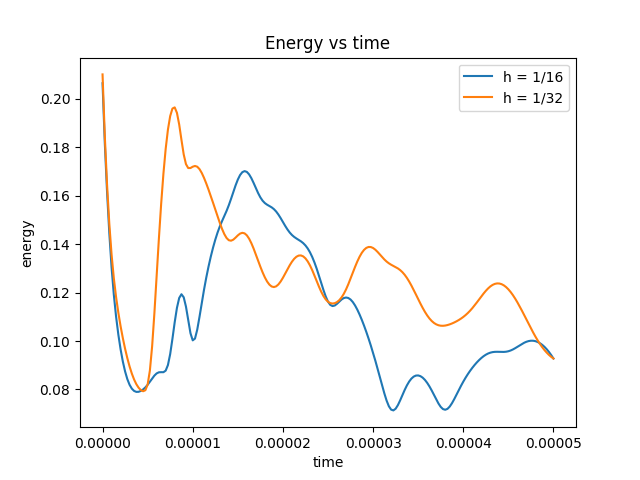}
		\caption{Plot of energy vs time for the \emph{nonlinear} scheme~\eqref{equ:scheme 2} in Simulation~2.}
		\label{fig:energy large scheme2}
	\end{subfigure}
\end{figure}

\subsection{Simulation 3 (adiabatic torque only)}
We consider scheme~\eqref{equ:scheme 2} with positive $\mu$ and take $k=1.0\times 10^{-6}$. The coefficients in~\eqref{equ:llb a} are $\gamma=2.3\times 10^5, \alpha=2.0\times 10^5, \sigma=1.0\times 10^{-6}, \kappa=2.0, \mu=2.0\times 10^{-6}$, $\lambda=0.01$, $\beta_1=0.2$, and $\beta_2=0$. The anisotropy axis vector is $\bff{e}=(0,0,1)^\top$ and the current density is $\bff{\nu}=(0, 10^4)^\top$. The initial data $\bff{u}_0$ is given by the vortex state
\begin{equation*}
	\bff{u}_0(x,y)= \big(-y, x, 0 \big).
\end{equation*}

Snapshots of the magnetic spin field $\bff{u}$ at selected times are shown in Figure~\ref{fig:snapshots exp 3 scheme1} for the linear scheme and in Figure~\ref{fig:snapshots exp 3 scheme2} for the nonlinear scheme. The two results exhibit good qualitative agreement, although the linear scheme appears to be computationally more efficient. The colour scale represents the relative magnitude of the magnetisation vectors. A vortex-like microstructure, which is widely explored as a potential means of data storage~\cite{GenJin17}, remains stable even as the object undergoes demagnetisation at temperatures exceeding the Curie temperature. The vortex is advected in the $+y$-direction by the current, suggesting that electric current may be used to transport such vortex within the magnet even at high temperatures. Following this motion, spin waves are observed propagating through the ferromagnet.

Plots of $\bff{e}_h$ against $1/h$ are shown in Figures~\ref{fig:order sim 3 scheme1} and~\ref{fig:order sim 3 scheme2} for the linear and the nonlinear schemes, respectively. Plots of $\bff{e}_h$ against $1/h$ with $k=O(h^2)$ and $k=O(h)$ to measure temporal error of the schemes in $\bb{L}^2$ and $\bb{H}^1$ norms, respectively, are shown in Figures~\ref{fig:order k sim3 scheme1} and~\ref{fig:order k sim3 scheme2}. Graphs of energy vs time are plotted in Figures~\ref{fig:energy sim3 scheme1} and \ref{fig:energy sim3 scheme2} for the two schemes, which show an a decay in energy as the current density is relatively small.

\begin{figure}[!htb]
	\centering
	\begin{subfigure}[b]{0.27\textwidth}
		\centering
		\includegraphics[width=\textwidth]{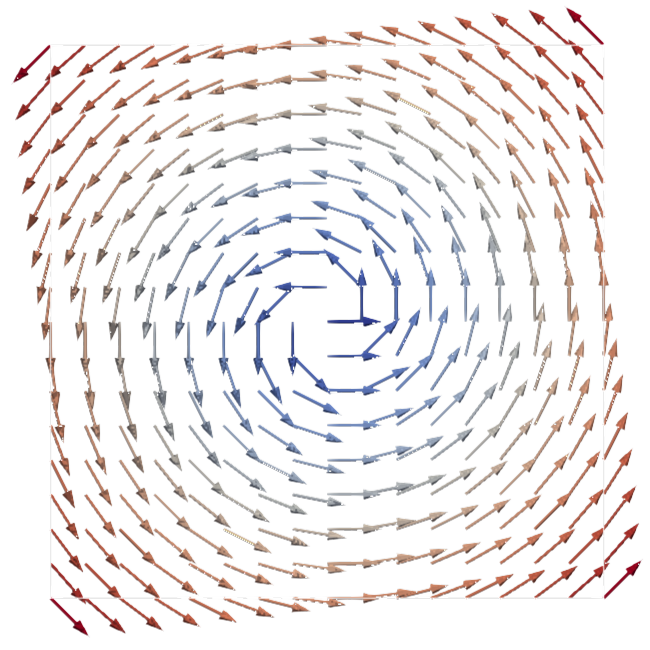}
		\caption{$t=0$}
	\end{subfigure}
	\begin{subfigure}[b]{0.27\textwidth}
		\centering
		\includegraphics[width=\textwidth]{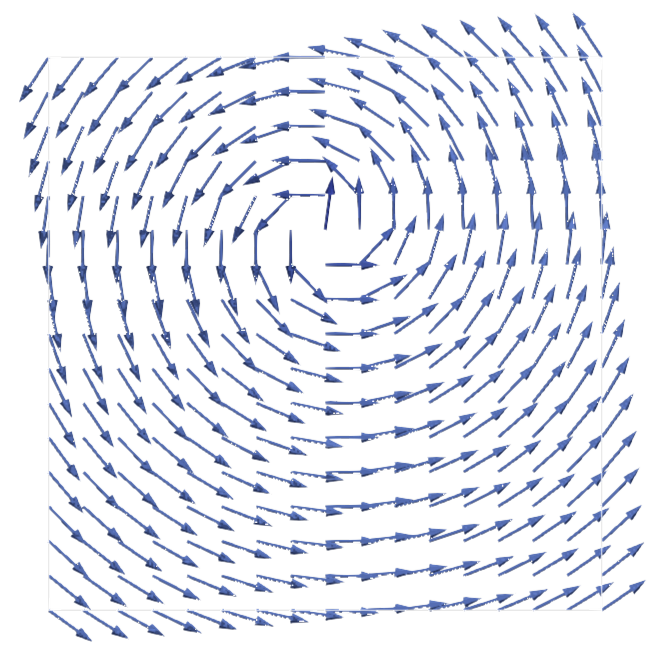}
		\caption{$t=1\times 10^{-4}$}
	\end{subfigure}
	\begin{subfigure}[b]{0.27\textwidth}
		\centering
		\includegraphics[width=\textwidth]{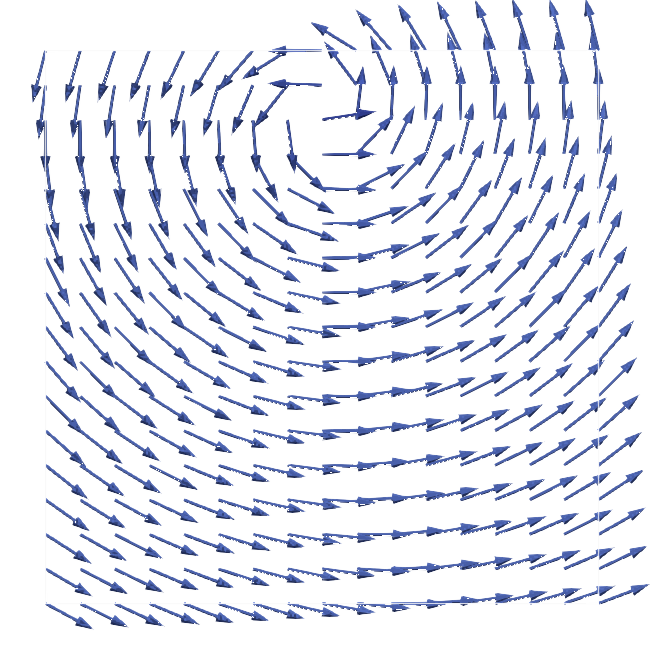}
		\caption{$t=2 \times 10^{-4}$}
	\end{subfigure}
	\begin{subfigure}[b]{0.1\textwidth}
		\centering
		\includegraphics[width=\textwidth]{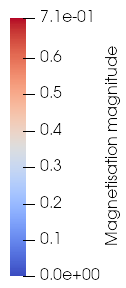}
	\end{subfigure}
	\begin{subfigure}[b]{0.27\textwidth}
		\centering
		\includegraphics[width=\textwidth]{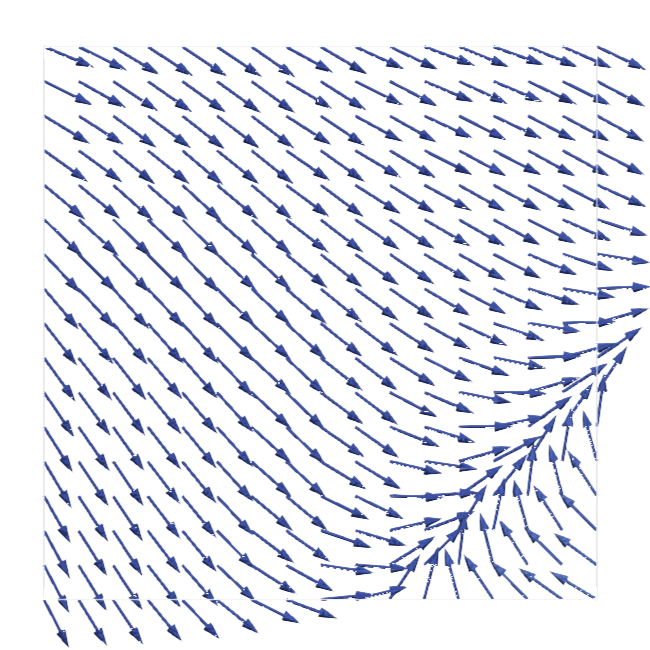}
		\caption{$t=5\times 10^{-4}$}
	\end{subfigure}
	\begin{subfigure}[b]{0.27\textwidth}
		\centering
		\includegraphics[width=\textwidth]{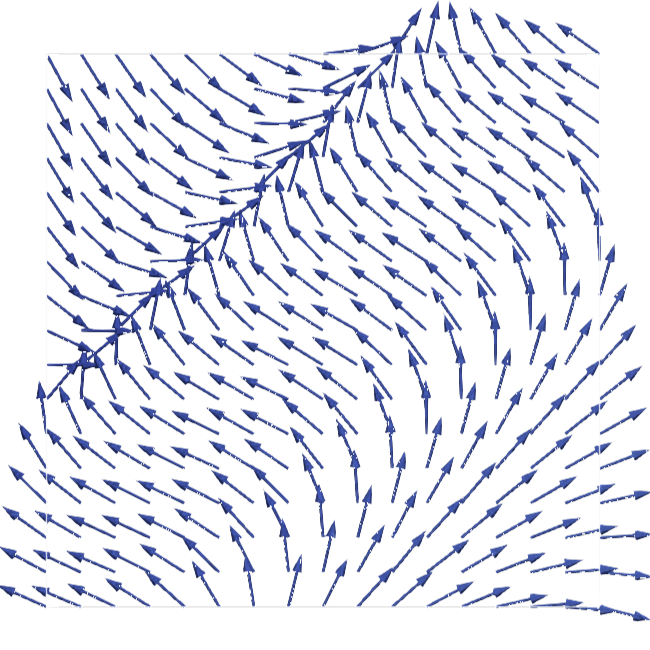}
		\caption{$t=1\times 10^{-3}$}
	\end{subfigure}
	\begin{subfigure}[b]{0.27\textwidth}
		\centering
		\includegraphics[width=\textwidth]{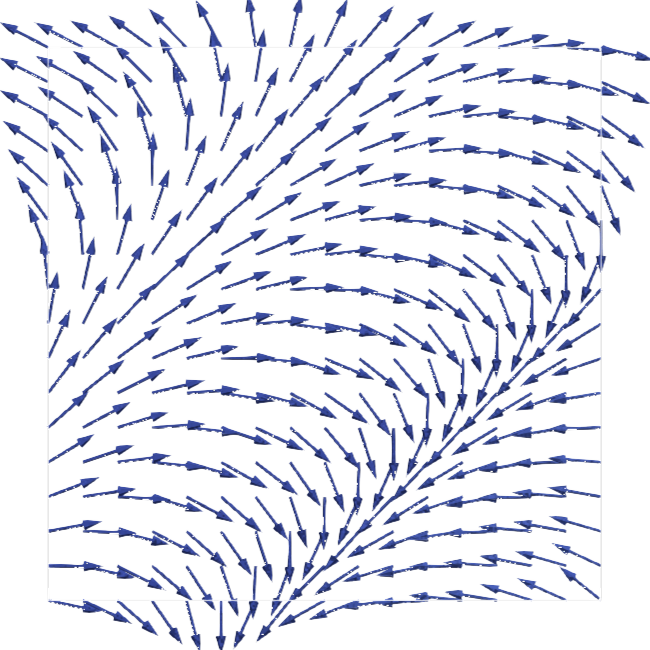}
		\caption{$t=5\times 10^{-3}$}
	\end{subfigure}
	\begin{subfigure}[b]{0.1\textwidth}
		\centering
		\includegraphics[width=\textwidth]{exp3_legend.png}
	\end{subfigure}
	\caption{Snapshots of the spin field $\bff{u}$ (projected onto $\bb{R}^2$) for the \emph{nonlinear} scheme~\eqref{equ:scheme spin} in Simulation 3.}
	\label{fig:snapshots exp 3 scheme1}
\end{figure}

\begin{figure}[!htb]
	\centering
	\begin{subfigure}[b]{0.27\textwidth}
		\centering
		\includegraphics[width=\textwidth]{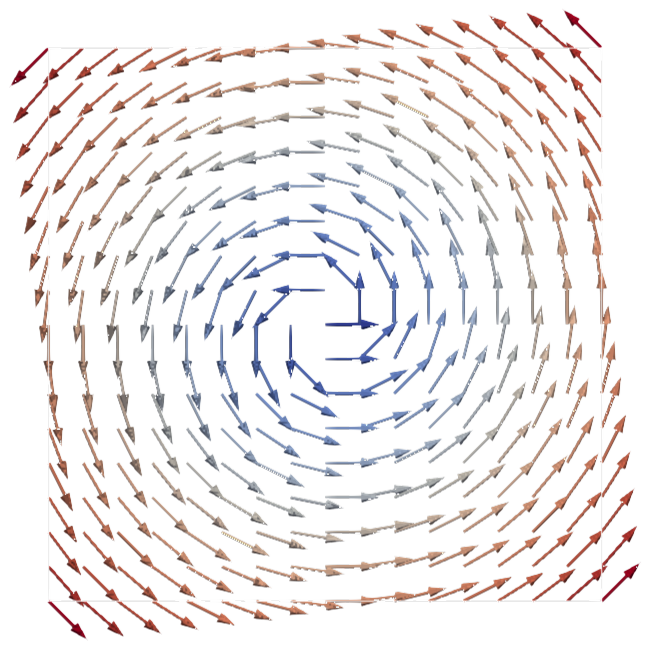}
		\caption{$t=0$}
	\end{subfigure}
	\begin{subfigure}[b]{0.27\textwidth}
		\centering
		\includegraphics[width=\textwidth]{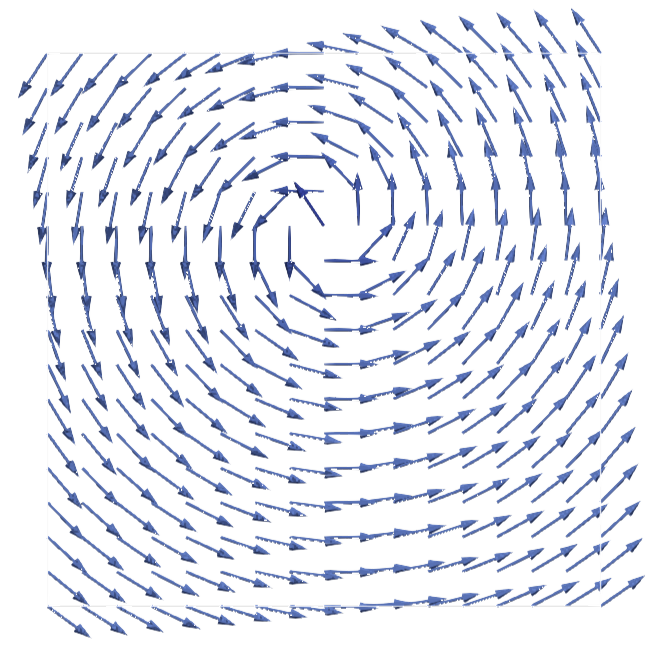}
		\caption{$t=1\times 10^{-4}$}
	\end{subfigure}
	\begin{subfigure}[b]{0.27\textwidth}
		\centering
		\includegraphics[width=\textwidth]{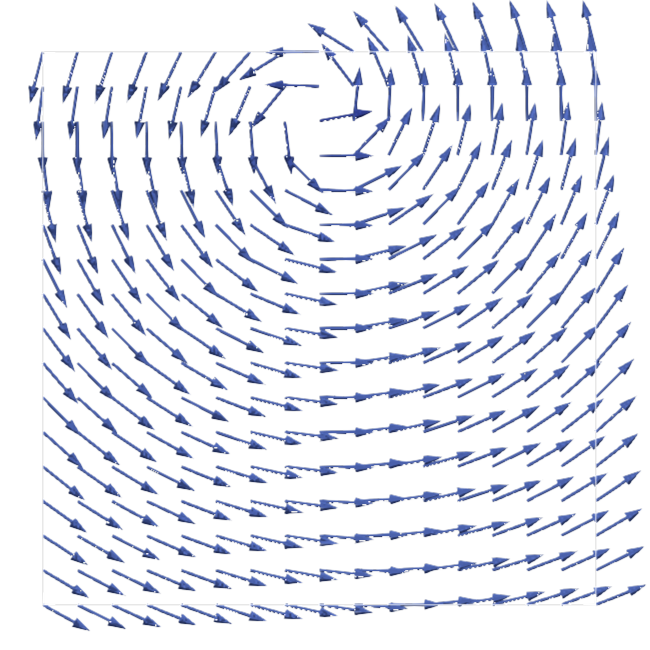}
		\caption{$t=2 \times 10^{-4}$}
	\end{subfigure}
	\begin{subfigure}[b]{0.1\textwidth}
		\centering
		\includegraphics[width=\textwidth]{exp3_legend.png}
	\end{subfigure}
	\begin{subfigure}[b]{0.27\textwidth}
		\centering
		\includegraphics[width=\textwidth]{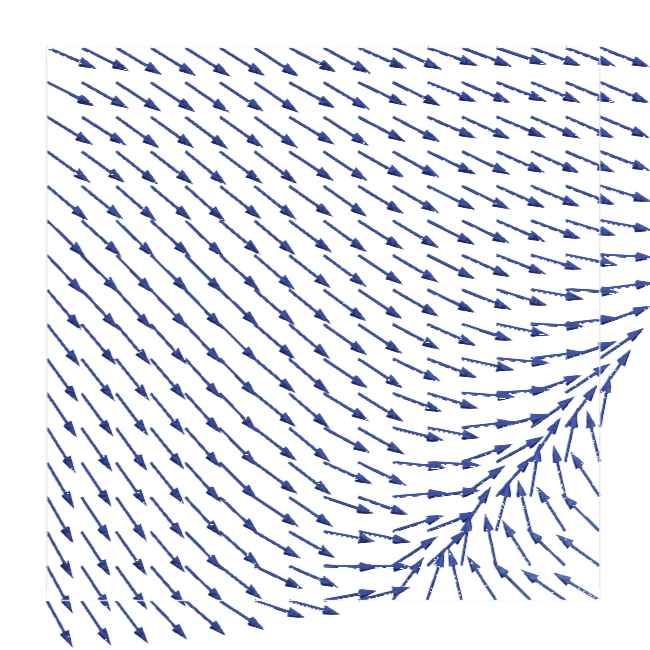}
		\caption{$t=5\times 10^{-4}$}
	\end{subfigure}
	\begin{subfigure}[b]{0.27\textwidth}
		\centering
		\includegraphics[width=\textwidth]{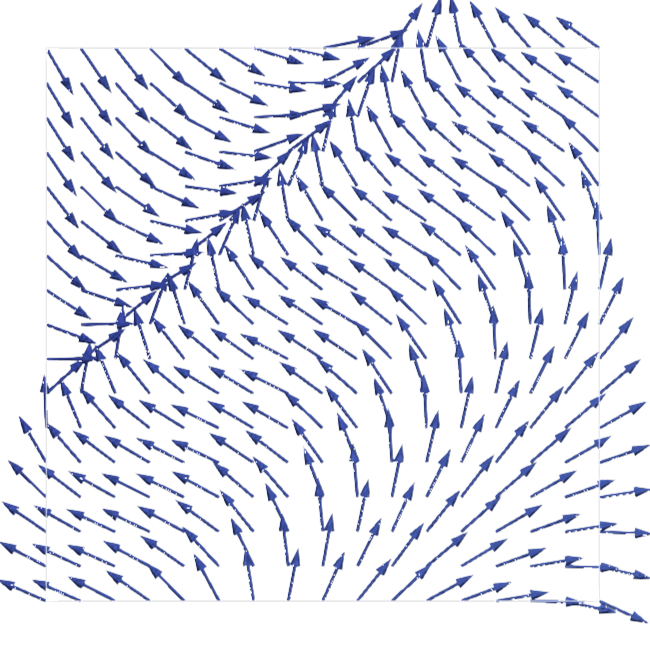}
		\caption{$t=1\times 10^{-3}$}
	\end{subfigure}
	\begin{subfigure}[b]{0.27\textwidth}
		\centering
		\includegraphics[width=\textwidth]{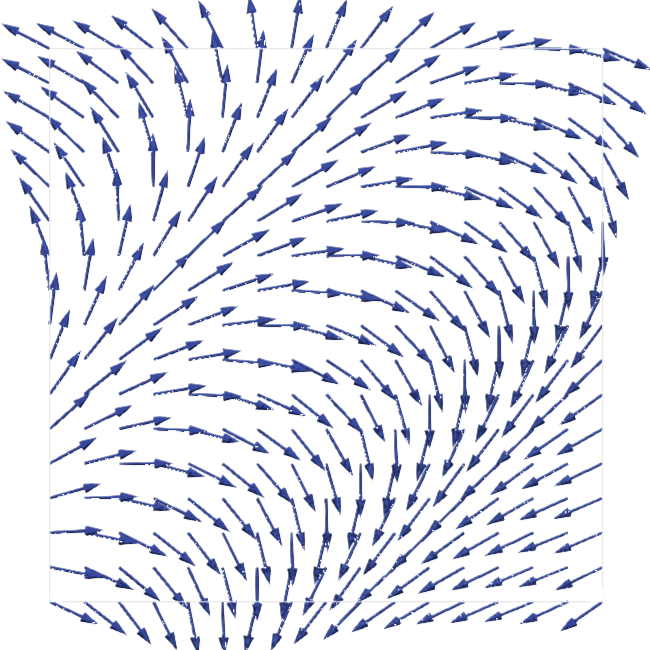}
		\caption{$t=5\times 10^{-3}$}
	\end{subfigure}
	\begin{subfigure}[b]{0.1\textwidth}
		\centering
		\includegraphics[width=\textwidth]{exp3_legend.png}
	\end{subfigure}
	\caption{Snapshots of the spin field $\bff{u}$ (projected onto $\bb{R}^2$) for the \emph{nonlinear} scheme~\eqref{equ:scheme 2} in Simulation 3.}
	\label{fig:snapshots exp 3 scheme2}
\end{figure}

\begin{figure}[!htb]
	\begin{subfigure}[b]{0.45\textwidth}
		\centering
		\begin{tikzpicture}
			\begin{axis}[
				title=Plot of $\bff{e}_h$ against $1/h$,
				height=1\textwidth,
				width=1\textwidth,
				xlabel= $1/h$,
				ylabel= $\bff{e}_h$,
				xmode=log,
				ymode=log,
				legend pos=south west,
				legend cell align=left,
				]
				\addplot+[mark=*,blue] coordinates {(2,0.047)(4,0.023)(8,0.012)(16,0.006)(32,0.0031)(64,0.0015)};
				\addplot+[mark=*,red] coordinates {(2,0.0008)(4,0.00037)(8,0.000092)(16,0.000023)(32,0.0000058)(64,0.0000014)};
				\addplot+[dashed,no marks,blue,domain=20:70]{0.03/x};
				\addplot+[dashed,no marks,red,domain=20:70]{0.0017/x^2};
				\legend{\small{$\max_n \norm{\bff{e}_h}{\bb{H}_0^1}$}, \small{$\max_n \norm{\bff{e}_h}{\bb{L}^2}$}, \small{order 1 line}, \small{order 2 line}}
			\end{axis}
		\end{tikzpicture}
		\caption{Spatial error order of $\bff{u}$ for the \emph{linear} scheme~\eqref{equ:scheme spin} in Simulation 3.}
		\label{fig:order sim 3 scheme1}
	\end{subfigure}
	\hspace{1em}
	\begin{subfigure}[b]{0.45\textwidth}
		\centering
		\begin{tikzpicture}
			\begin{axis}[
				title=Plot of $\bff{e}_h$ against $1/h$,
				height=1\textwidth,
				width=1\textwidth,
				xlabel= $1/h$,
				ylabel= $\bff{e}_h$,
				xmode=log,
				ymode=log,
				legend pos=south west,
				legend cell align=left,
				]
				\addplot+[mark=*,blue] coordinates {(2,0.047)(4,0.023)(8,0.012)(16,0.0059)(32,0.003)(64,0.0017)};
				\addplot+[mark=*,red] coordinates {(2,0.0016)(4,0.00037)(8,0.000091)(16,0.000023)(32,0.0000064)(64,0.0000018)};
				\addplot+[dashed,no marks,blue,domain=20:70]{0.03/x};
				\addplot+[dashed,no marks,red,domain=20:70]{0.0017/x^2};
				\legend{\small{$\max_n \norm{\bff{e}_h}{\bb{H}_0^1}$}, \small{$\max_n \norm{\bff{e}_h}{\bb{L}^2}$}, \small{order 1 line}, \small{order 2 line}}
			\end{axis}
		\end{tikzpicture}
		\caption{Spatial error order of $\bff{u}$ for the \emph{nonlinear} scheme~\eqref{equ:scheme 2} in Simulation 3.}
		\label{fig:order sim 3 scheme2}
	\end{subfigure}
\end{figure}

\begin{figure}[!htb]
	\begin{subfigure}[b]{0.45\textwidth}
		\centering
		\begin{tikzpicture}
			\begin{axis}[
				title=Plot of $\bff{e}_h$ against $1/h$,
				height=1\textwidth,
				width=1\textwidth,
				xlabel= $1/h$,
				ylabel= $\bff{e}_h$,
				xmode=log,
				ymode=log,
				legend pos=south west,
				legend cell align=left,
				]
				\addplot+[mark=*,blue] coordinates {(4,0.18)(8,0.099)(16,0.054)(32,0.03)(64,0.016)};
				\addplot+[mark=*,red] coordinates {(4,0.008)(8,0.0023)(16,0.0007)(32,0.00016)(64,0.00004)};
				\addplot+[dashed,no marks,blue,domain=20:70]{0.4/x};
				\addplot+[dashed,no marks,red,domain=20:70]{0.06/x^2};
				\legend{\small{$\max_n \norm{\bff{e}_h}{\bb{H}_0^1}$}, \small{$\max_n \norm{\bff{e}_h}{\bb{L}^2}$}, \small{order 1 line}, \small{order 2 line}}
			\end{axis}
		\end{tikzpicture}
		\caption{Error order of $\bff{u}$ for the \emph{linear} scheme~\eqref{equ:scheme spin} in Simulation 3 in $\bb{L}^2$ norm with $k=0.01h^2$, and in $\bb{H}^1_0$ norm with $k=0.01h$.}
		\label{fig:order k sim3 scheme1}
	\end{subfigure}
	\hspace{1em}
	\begin{subfigure}[b]{0.45\textwidth}
		\centering
		\begin{tikzpicture}
			\begin{axis}[
				title=Plot of $\bff{e}_h$ against $1/h$,
				height=1\textwidth,
				width=1\textwidth,
				xlabel= $1/h$,
				ylabel= $\bff{e}_h$,
				xmode=log,
				ymode=log,
				legend pos=south west,
				legend cell align=left,
				]
				\addplot+[mark=*,blue] coordinates {(4,0.2)(8,0.09)(16,0.05)(32,0.03)(64,0.016)};
				\addplot+[mark=*,red] coordinates {(4,0.01)(8,0.003)(16,0.0006)(32,0.00011)(64,0.00003)};
				\addplot+[dashed,no marks,blue,domain=20:70]{0.4/x};
				\addplot+[dashed,no marks,red,domain=20:70]{0.04/x^2};
				\legend{\small{$\max_n \norm{\bff{e}_h}{\bb{H}_0^1}$}, \small{$\max_n \norm{\bff{e}_h}{\bb{L}^2}$}, \small{order 1 line}, \small{order 2 line}}
			\end{axis}
		\end{tikzpicture}
		\caption{Error order of $\bff{u}$ for the \emph{nonlinear} scheme~\eqref{equ:scheme 2} in Simulation 3 in $\bb{L}^2$ norm with $k=0.01h^2$, and in $\bb{H}^1_0$ norm with $k=0.01h$.}
		\label{fig:order k sim3 scheme2}
	\end{subfigure}
\end{figure}

\begin{figure}[!htb]
	\begin{subfigure}[b]{0.46\textwidth}
		\centering
		\includegraphics[width=\textwidth]{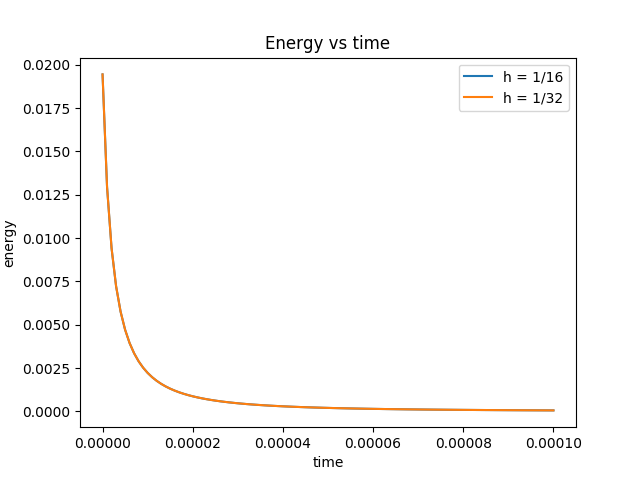}
		\caption{Plot of energy vs time for the \emph{linear} scheme~\eqref{equ:scheme spin} in Simulation~3.}
		\label{fig:energy sim3 scheme1}
	\end{subfigure}
	\hspace{1em}
	\begin{subfigure}[b]{0.46\textwidth}
		\centering
		\includegraphics[width=\textwidth]{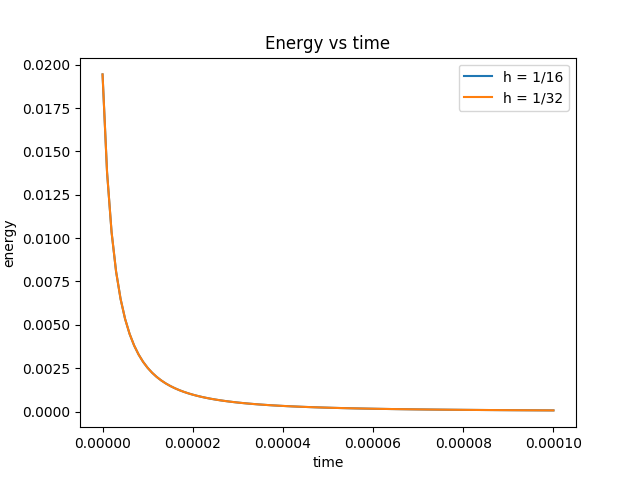}
		\caption{Plot of energy vs time for the \emph{nonlinear} scheme~\eqref{equ:scheme 2} in Simulation~3.}
		\label{fig:energy sim3 scheme2}
	\end{subfigure}
\end{figure}

\subsection{Simulation 4 (energy dissipation)}
In this last simulation, we consider an academic example where no current and no anisotropy is present. For this numerical experiment, scheme \eqref{equ:scheme 2} is seen to preserve energy dissipation at the discrete level, while scheme \eqref{equ:scheme spin} does not. The coefficients in \eqref{equ:llb a} are taken to be $\gamma=2.5\times 10^{12}$, $\alpha=0.2$, $\sigma=1.0\times 10^{-10}$, $\kappa=0.1$, $\mu=1.0\times 10^{-7}$, and $\lambda=0$. Since no current is assumed, $\bff{\nu}=\bff{0}$.
We take $k=1\times 10^{-5}$ and let the initial data $\bff{u}_0$ be given by
\begin{equation*}
	\bff{u}_0(x,y)= \big(-y, x, 0.01 \big).
\end{equation*}

Graphs of energy vs time are plotted in Figures~\ref{fig:energy sim4 scheme1} and \ref{fig:energy sim4 scheme2} for the two schemes. Physical theory indicates that energy should be dissipated by the system, and this is reflected at the discrete level for scheme~\ref{equ:scheme 2} in Figure~\ref{fig:energy sim4 scheme2}. However, as seen in Figure~\ref{fig:energy sim4 scheme1}, this is not maintained at the discrete level for scheme~\ref{equ:scheme spin}.

\begin{figure}[!htb]
	\begin{subfigure}[b]{0.46\textwidth}
		\centering
		\includegraphics[width=\textwidth]{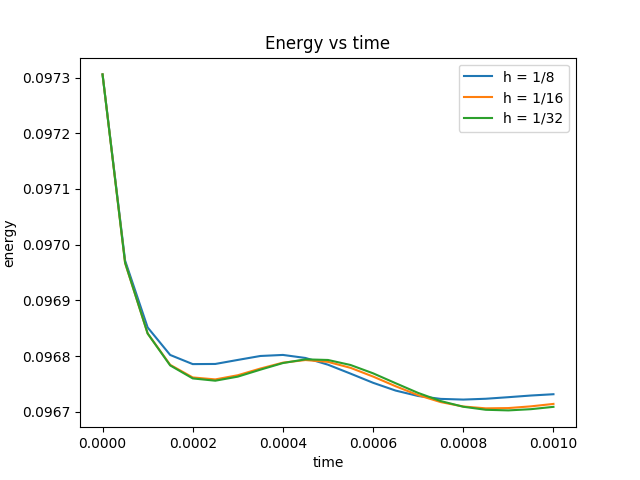}
		\caption{Plot of energy vs time for the \emph{linear} scheme~\eqref{equ:scheme spin} in Simulation~4.}
		\label{fig:energy sim4 scheme1}
	\end{subfigure}
	\hspace{1em}
	\begin{subfigure}[b]{0.46\textwidth}
		\centering
		\includegraphics[width=\textwidth]{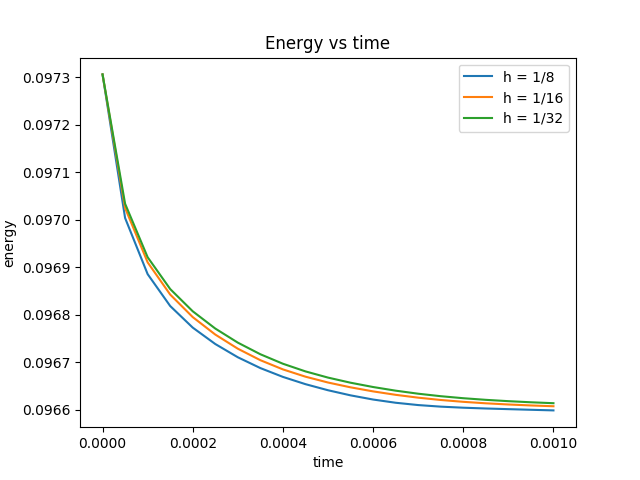}
		\caption{Plot of energy vs time for the \emph{nonlinear} scheme~\eqref{equ:scheme 2} in Simulation~4.}
		\label{fig:energy sim4 scheme2}
	\end{subfigure}
\end{figure}

\section{Conclusion}\label{sec:conclusion}

In this paper, we show the unique existence of global strong (and more regular) solution to the LLB equation with spin-torques above the Curie temperature on $\mathscr{D}\subset \bb{R}^d$, where $d\leq 3$ (with an additional small data assumption if $d=3$). Certain smoothing and decay estimates are also given. We then propose two finite element schemes to approximate the solution: a linear scheme~\ref{equ:scheme spin} and a nonlinear scheme~\ref{equ:scheme 2}. We show that the first (linear) scheme is stable and convergent (with an optimal order) in $\ell^\infty(0,T;\bb{L}^2)\cap \ell^2(0,T;\bb{H}^1)$, but unfortunately we are not able to show stability and convergence for this scheme in $\ell^\infty(0,T;\bb{H}^1)$. Furthermore, while it is linear, this scheme may not maintain the energy dissipation property of the system in the absence of current at the discrete level. On the other hand, the second scheme, albeit being nonlinear, is provably stable and convergent in $\ell^\infty(0,T;\bb{H}^1)$, but we can only show this when the non-adiabatic torque is omitted (which is actually physically justifiable in most standard cases). Furthermore, in the absence of current, this second scheme is able to maintain energy dissipativity at the discrete level. 

This begs the question: Which scheme is the `best'? At this point, we concede that each scheme has its own advantages and shortcomings. In general, when a current is present, the linear scheme exhibits shorter runtimes while producing results that are qualitatively similar to those of the nonlinear scheme (see Simulation 1, 2, and 3 in Section~\ref{sec:num exp}). In the absence of current, when simulating processes that require more precise knowledge of the energy relaxation behaviour -- such as laser- or heat-induced (de)magnetisation dynamics~\cite{AtxChuKaz07} -- the nonlinear energy-dissipative scheme~\eqref{equ:scheme 2} may be preferred (cf. Simulation 4).

Altogether, these results improve upon the existing theoretical and numerical findings on the LLB equation in the literature~\cite{AyoKotMouZak21, BenEssAyo24, Le16, LeSoeTra24, Soe24}, and further extend them to incorporate spin-torque effects. The general LLB equation~\eqref{equ:general LLB} will be addressed in a forthcoming paper. As for future work, we aim to develop a numerical scheme for the LLB equation that is both linear and energy-dissipative in the absence of current. It would also be of interest to rigorously establish convergence of this linear scheme~\eqref{equ:scheme spin} in $\ell^\infty(0,T;\bb{H}^1)$, as suggested by the numerical experiments.

\section*{Acknowledgements}
Agus L. Soenjaya is supported by the Australian Government through the Research Training Program (RTP) Scholarship awarded at the University of New South Wales, Sydney.
Financial support from the Australian Research Council under grant number DP200101866 (awarded to Prof. Thanh Tran) is gratefully acknowledged.
The author wishes to thank the referees for their careful reading and valuable comments, which led to a significant improvement in the paper.

\bibliographystyle{myabbrv}
\bibliography{mybib}

\appendix

\section{A fixed-point iteration for the nonlinear scheme}\label{sec:linearisation}

At each time step in scheme~\eqref{equ:scheme 2}, a system of nonlinear equations must be solved. To address this, we propose a simple fixed-point iteration, outlined in the following algorithm, which also demonstrates the uniqueness of the finite element solution $\bff{u}_h^n$ under certain technical conditions.

\begin{algorithm}\label{alg:iteration alg}
	Input: initial data $\bff{u}_h^0$, current $\{\bff{\nu}^n\}_{0\leq n\leq N}$, and tolerance $\mathsf{tol}>0$.
	\\
	For $n=1,2,\ldots,N$, iterate:
	\begin{enumerate}[(i)]
		\item \label{item:init} Set $\bff{u}_h^{n,0}:= \bff{u}_h^{n-1}$.
	\end{enumerate}
	\quad For $j\in \bb{N}$:
	\begin{enumerate}
		\item \label{item:iter} Compute $\bff{u}_h^{n, j}\in \bb{V}_h$ such that for all $\bff{\chi}\in \bb{V}_h$,
		\begin{align}\label{equ:inner iterate}
			\frac{1}{k} \inpro{\bff{u}_h^{n,j}- \bff{u}_h^{n-1}}{\bff{\chi}} 
			&=
			-\inpro{\bff{u}_h^{n,j}\times \bff{H}_h^{n,j-1}}{\bff{\chi}}
			+
			\alpha \inpro{\Delta_h \bff{u}_h^{n,j}}{\bff{\chi}}
			-
			\alpha \inpro{|\bff{u}_h^{n,j-1}|^2 \bff{u}_h^{n,j}}{\bff{\chi}}
			\nonumber\\
			&\quad
			-
			\alpha \inpro{\bff{u}_h^{n,j}}{\bff{\chi}}
			-
			\alpha \inpro{\bff{e}(\bff{e}\cdot \bff{u}_h^{n,j})}{\bff{\chi}}
			+
			\beta_1 \inpro{(\bff{\nu}^n\cdot\nabla)\bff{u}_h^{n-1}}{\bff{\chi}},
		\end{align}
		where $\bff{H}_h^{n,j-1}:= \Delta_h \bff{u}_h^{n,j-1} - \bff{u}_h^{n,j-1} - P_h\big(|\bff{u}_h^{n,j-1}|^2 \bff{u}_h^{n,j-1}\big)$;
		\item If $\norm{\bff{u}_h^{n,j}-\bff{u}_h^{n,j-1}}{\bb{L}^2} \geq \mathsf{tol}$, then replace $j\mapsto j+1$ and go back to \eqref{item:iter};
		\item Otherwise if $\norm{\bff{u}_h^{n,j}-\bff{u}_h^{n,j-1}}{\bb{L}^2} < \mathsf{tol}$, then set $\bff{u}_h^{n}:= \bff{u}_h^{n,j}$, replace $n\mapsto n+1$, and return to \eqref{item:init}.
	\end{enumerate}
	Output: a sequence of discrete functions $\{\bff{u}_h^n\}_{1\leq n\leq N}$.
\end{algorithm}

Note that step \eqref{equ:inner iterate} is well-posed. To see this, for a given $\bff{u}_h^{n,j-1}, \bff{u}_h^{n-1} \in \bb{V}_h$, we define a bilinear form $a:\bb{V}_h\times \bb{V}_h\to \bb{R}$ by
\begin{align*}
	a(\bff{v},\bff{\phi})
	&:=
	\inpro{\bff{v}}{\bff{\phi}}
	+
	k\inpro{\bff{v}\times \bff{H}_h^{n,j-1}}{\bff{\phi}}
	-
	\alpha k\inpro{\Delta_h \bff{v}}{\bff{\phi}}
	+
	\alpha k\inpro{|\bff{u}_h^{n,j-1}|^2 \bff{v}}{\bff{\phi}}
	+
	\alpha k \inpro{\bff{v}}{\bff{\phi}}
	+
	\alpha k\inpro{\bff{e}(\bff{e}\cdot \bff{v})}{\bff{\phi}}
\end{align*}
and a linear form $L:\bb{V}_h\to \bb{R}$ by
\begin{align*}
	L(\bff{\phi})
	&:=
	\inpro{\bff{u}_h^{n-1}}{\bff{\phi}}
	+
	\beta_1 k\inpro{(\bff{\nu}^n\cdot\nabla)\bff{u}_h^{n-1}}{\bff{\phi}}.
\end{align*}
Then \eqref{equ:inner iterate} is equivalent to solving
\begin{align}\label{equ:equiv iterate}
	a(\bff{u}_h^{n,j}, \bff{\chi})= L(\bff{\chi}), \quad \forall\bff{\chi}\in \bb{V}_h.
\end{align}
Note that the bilinear form $a$ is bounded and $\bb{V}_h$-elliptic, namely for any $\bff{v},\bff{\phi}\in \bb{V}_h$,
\begin{align*}
	|a(\bff{v},\bff{\phi})|
	&\leq
	\norm{\bff{v}}{\bb{L}^2} \norm{\bff{\phi}}{\bb{L}^2}
	+
	k\norm{\bff{v}}{\bb{L}^4} \norm{\bff{H}_h^{n,j-1}}{\bb{L}^2} \norm{\bff{\phi}}{\bb{L}^4}
	+
	\alpha k \norm{\nabla \bff{v}}{\bb{L}^2} \norm{\nabla\bff{\phi}}{\bb{L}^2}
	\\
	&\quad
	+
	\alpha k \norm{\bff{u}_h^{n,j-1}}{\bb{L}^4}^2 \norm{\bff{v}}{\bb{L}^4} \norm{\bff{\phi}}{\bb{L}^4}
	+
	2\alpha k \norm{\bff{v}}{\bb{L}^2} \norm{\bff{\phi}}{\bb{L}^2}
	\\
	&\leq
	C\norm{\bff{v}}{\bb{H}^1} \norm{\bff{\phi}}{\bb{H}^1},
\end{align*}
and
\begin{align*}
	a(\bff{v},\bff{v})
	&=
	\norm{\bff{v}}{\bb{L}^2}^2
	+
	\alpha k\norm{\nabla \bff{v}}{\bb{L}^2}^2
	+
	\alpha k\norm{|\bff{u}_h^{n,j-1}| |\bff{v}|}{\bb{L}^2}^2
	+
	\alpha \norm{\bff{v}}{\bb{L}^2}^2
	+
	\alpha k\norm{\bff{e}\cdot \bff{v}}{\bb{L}^2}^2
	\geq
	\min\left(1,\alpha k\right) \norm{\bff{v}}{\bb{H}^1}^2.
\end{align*}
For the linear form $L$ we have
\begin{align*}
	\abs{L(\bff{\phi})} \leq \left(\norm{\bff{u}_h^{n-1}}{\bb{L}^2} + \beta_1 \nu_\infty k \norm{\nabla \bff{u}_h^{n-1}}{\bb{L}^2} \right) \norm{\bff{\phi}}{\bb{L}^2}.
\end{align*}
Hence, we deduce that \eqref{equ:inner iterate} admits a unique solution $\bff{u}_h^{n,j}\in \bb{V}_h$ by the Lax--Milgram lemma. The following proposition shows that under certain conditions, Algorithm~\ref{alg:iteration alg} is guaranteed to be well-posed.

\begin{proposition}\label{pro:alg wellposed}
	If $k=o(h^\beta)$, where $\beta:= \max\left(2+\frac{d}{2}, \frac{3d}{2}\right)$, then there exists a unique sequence $\{\bff{u}_h^n\}_{1\leq n\leq N}$ solving Algorithm~\ref{alg:iteration alg}.
\end{proposition}

\begin{proof}
	Observe that step \eqref{equ:inner iterate} is well-defined by the argument preceding this proposition. We start with some preliminary estimates. Firstly, for a given $\bff{u}_h^{n-1}$, we put $\bff{\chi}=\bff{u}_h^{n,j}$ in \eqref{equ:equiv iterate} and use~\eqref{equ:div thm beta1} to obtain
	\begin{align*}
		&\norm{\bff{u}_h^{n,j}}{\bb{L}^2}^2
		+
		\alpha k\norm{\nabla \bff{u}_h^{n,j}}{\bb{L}^2}^2
		+
		\alpha k\norm{|\bff{u}_h^{n,j-1}| |\bff{u}_h^{n,j}|}{\bb{L}^2}^2
		+
		\alpha k \norm{\bff{u}_h^{n,j}}{\bb{L}^2}^2
		+
		\alpha k\norm{\bff{e}\cdot \bff{u}_h^{n,j}}{\bb{L}^2}^2
		\\
		&=
		\inpro{\bff{u}_h^{n-1}}{\bff{u}_h^{n,j}}
		-
		\beta_1 k \inpro{\bff{u}_h^{n-1}\otimes \bff{\nu}^n}{\nabla \bff{u}_h^{n,j}}
		-
		\beta_1 k \inpro{(\nabla \cdot\bff{\nu}^n) \bff{u}_h^{n-1}}{\bff{u}_h^{n,j}}
		\\
		&\leq
		\frac12 \norm{\bff{u}_h^{n-1}}{\bb{L}^2}^2
		+
		\frac12 \norm{\bff{u}_h^{n,j}}{\bb{L}^2}^2
		+
		\frac{\alpha k}{4} \norm{\nabla \bff{u}_h^{n,j}}{\bb{L}^2}^2
		+
		\frac{\alpha k}{4} \norm{\bff{u}_h^{n,j}}{\bb{L}^2}^2
		+
		\frac{2k (\beta_1 \nu_\infty)^2}{\alpha} \norm{\bff{u}_h^{n-1}}{\bb{L}^2}^2,
	\end{align*}
	where in the last step we used Young's inequality and the assumption that $\norm{\bff{\nu}}{L^\infty(0,T;\bb{W}^{1,\infty}(\mathscr{D};\,\bb{R}^d))} \leq \nu_\infty$.
	This implies, after rearranging the terms,
	\begin{align}\label{equ:uhnj bdd}
		\norm{\bff{u}_h^{n,j}}{\bb{L}^2}^2
		\leq
		\left(1+4k (\beta_1 \nu_\infty)^2 \alpha^{-1}\right) \norm{\bff{u}_h^{n-1}}{\bb{L}^2}^2.
	\end{align}
	
	Next, subtraction of equation~\eqref{equ:inner iterate} for two consecutive iterations yields
	\begin{align*}
		\inpro{\bff{u}_h^{n,j+1}-\bff{u}_h^{n,j}}{\bff{\chi}}
		&=
		-k \inpro{\bff{u}_h^{n,j+1}\times \bff{H}_h^{n,j}- \bff{u}_h^{n,j} \times \bff{H}_h^{n,j-1}}{\bff{\chi}}
		+
		\alpha k \inpro{\Delta_h \bff{u}_h^{n,j+1}-\Delta_h \bff{u}_h^{n,j}}{\bff{\chi}}
		\\
		&\quad
		-
		\alpha k\inpro{|\bff{u}_h^{n,j}|^2 \bff{u}_h^{n,j+1}- |\bff{u}_h^{n,j-1}|^2 \bff{u}_h^{n,j}}{\bff{\chi}}
		-
		\alpha k\inpro{\bff{e}\big(\bff{e}\cdot (\bff{u}_h^{n,j+1}-\bff{u}_h^{n,j})\big)}{\bff{\chi}}
	\end{align*}
	Putting $\bff{\chi}= \bff{u}_h^{n,j+1}-\bff{u}_h^{n,j}$ and rearranging the terms, we obtain
	\begin{align}\label{equ:subtract norm}
		&\norm{\bff{u}_h^{n,j+1}-\bff{u}_h^{n,j}}{\bb{L}^2}^2
		+
		\alpha k \norm{\nabla \bff{u}_h^{n,j+1}- \nabla \bff{u}_h^{n,j+1}}{\bb{L}^2}^2
		+
		\alpha k \norm{\bff{e}\cdot (\bff{u}_h^{n,j+1}- \bff{u}_h^{n,j})}{\bb{L}^2}^2
		\nonumber\\
		&\quad
		+
		\alpha k \norm{|\bff{u}_h^{n,j-1}| |\bff{u}_h^{n,j}|}{\bb{L}^2}^2
		+
		\alpha k \norm{|\bff{u}_h^{n,j}| |\bff{u}_h^{n,j+1}|}{\bb{L}^2}^2
		\nonumber\\
		&=
		-k\inpro{\bff{u}_h^{n,j}\times (\bff{H}_h^{n,j}-\bff{H}_h^{n,j-1})}{\bff{u}_h^{n,j+1}-\bff{u}_h^{n,j}}
		-
		\alpha k \inpro{\left(|\bff{u}_h^{n,j}|^2 - |\bff{u}_h^{n,j-1}|^2\right) \bff{u}_h^{n,j}}{\bff{u}_h^{n,j+1}}
		\nonumber\\
		&\leq
		k\norm{\bff{u}_h^{n,j}}{\bb{L}^\infty} \norm{\bff{H}_h^{n,j}-\bff{H}_h^{n,j-1}}{\bb{L}^2} \norm{\bff{u}_h^{n,j+1}- \bff{u}_h^{n,j}}{\bb{L}^2}
		\nonumber\\
		&\quad
		+
		\alpha k \norm{\bff{u}_h^{n,j}+\bff{u}_h^{n,j-1}}{\bb{L}^\infty}
		\norm{\bff{u}_h^{n,j}-\bff{u}_h^{n,j-1}}{\bb{L}^2}
		\norm{|\bff{u}_h^{n,j}| |\bff{u}_h^{n,j+1}|}{\bb{L}^2}.
	\end{align}
	Note that we have
	\begin{align*}
		\bff{H}_h^{n,j}-\bff{H}_h^{n,j-1}
		&=
		\left(\Delta_h \bff{u}_h^{n,j}- \Delta_h \bff{u}_h^{n,j-1}\right)
		-
		\left(\bff{u}_h^{n,j}-\bff{u}_h^{n,j-1} \right)
		\\
		&\quad
		-
		P_h \left(|\bff{u}_h^{n,j}|^2 (\bff{u}_h^{n,j}-\bff{u}_h^{n,j-1})\right)
		-
		P_h\left( \big(|\bff{u}_h^{n,j}|^2 - |\bff{u}_h^{n,j-1}|^2 \big) \bff{u}_h^{n,j-1} \right).
	\end{align*}
	Therefore, by H\"older's inequality and inverse estimates (\eqref{equ:inverse} and \eqref{equ:inverse disc lapl}), we obtain
	\begin{align*}
		\norm{\bff{H}_h^{n,j}-\bff{H}_h^{n,j-1}}{\bb{L}^2}
		&\leq
		\norm{\Delta_h \bff{u}_h^{n,j}- \Delta_h \bff{u}_h^{n,j-1}}{\bb{L}^2}
		+
		\norm{\bff{u}_h^{n,j}-\bff{u}_h^{n,j-1}}{\bb{L}^2}
		\\
		&\quad
		+
		\left(\norm{\bff{u}_h^{n,j}}{\bb{L}^\infty}^2 + \norm{\bff{u}_h^{n,j-1}}{\bb{L}^\infty}^2 \right) \norm{\bff{u}_h^{n,j}-\bff{u}_h^{n,j-1}}{\bb{L}^2}
		\\
		&\leq
		Ch^{-2} \norm{\bff{u}_h^{n,j}-\bff{u}_h^{n,j-1}}{\bb{L}^2}
		+
		\norm{\bff{u}_h^{n,j}-\bff{u}_h^{n,j-1}}{\bb{L}^2}
		\\
		&\quad
		+
		Ch^{-d} \left(\norm{\bff{u}_h^{n,j}}{\bb{L}^2}^2 + \norm{\bff{u}_h^{n,j-1}}{\bb{L}^2}^2 \right) \norm{\bff{u}_h^{n,j}-\bff{u}_h^{n,j-1}}{\bb{L}^2}.
	\end{align*}
	Continuing from~\eqref{equ:subtract norm}, using this inequality and inverse estimates, we then have
	\begin{align*}
		&\norm{\bff{u}_h^{n,j+1}-\bff{u}_h^{n,j}}{\bb{L}^2}^2
		+
		\alpha k \norm{\nabla \bff{u}_h^{n,j+1}- \nabla \bff{u}_h^{n,j+1}}{\bb{L}^2}^2
		+
		\alpha k \norm{\bff{e}\cdot (\bff{u}_h^{n,j+1}- \bff{u}_h^{n,j})}{\bb{L}^2}^2
		\nonumber\\
		&\quad
		+
		\alpha k \norm{|\bff{u}_h^{n,j-1}| |\bff{u}_h^{n,j}|}{\bb{L}^2}^2
		+
		\alpha k \norm{|\bff{u}_h^{n,j}| |\bff{u}_h^{n,j+1}|}{\bb{L}^2}^2
		\nonumber\\
		&\leq
		k\norm{\bff{u}_h^{n,j}}{\bb{L}^\infty} \norm{\bff{H}_h^{n,j}-\bff{H}_h^{n,j-1}}{\bb{L}^2} \norm{\bff{u}_h^{n,j+1}- \bff{u}_h^{n,j}}{\bb{L}^2}
		\nonumber\\
		&\quad
		+
		\alpha k \norm{\bff{u}_h^{n,j}+\bff{u}_h^{n,j-1}}{\bb{L}^\infty}
		\norm{\bff{u}_h^{n,j}-\bff{u}_h^{n,j-1}}{\bb{L}^2}
		\norm{|\bff{u}_h^{n,j}| |\bff{u}_h^{n,j+1}|}{\bb{L}^2}
		\nonumber\\
		&\leq
		Ckh^{-\left(2+\frac{d}{2}\right)} \norm{\bff{u}_h^{n,j}}{\bb{L}^2} \norm{\bff{u}_h^{n,j}-\bff{u}_h^{n,j-1}}{\bb{L}^2}^2
		+
		Ckh^{-\frac{d}{2}} \norm{\bff{u}_h^{n,j}}{\bb{L}^2} \norm{\bff{u}_h^{n,j}-\bff{u}_h^{n,j-1}}{\bb{L}^2}^2
		\\
		&\quad
		+
		Ckh^{-\frac{3d}{2}} \norm{\bff{u}_h^{n,j}}{\bb{L}^2} \left(\norm{\bff{u}_h^{n,j}}{\bb{L}^2}^2 + \norm{\bff{u}_h^{n,j-1}}{\bb{L}^2}^2 \right)  \norm{\bff{u}_h^{n,j}-\bff{u}_h^{n,j-1}}{\bb{L}^2}^2
		\\
		&\quad
		+
		\frac{\alpha k}{2}  \norm{|\bff{u}_h^{n,j}| |\bff{u}_h^{n,j+1}|}{\bb{L}^2}^2
		+
		\frac{\alpha kh^{-d}}{2} \norm{\bff{u}_h^{n,j}+\bff{u}_h^{n,j-1}}{\bb{L}^2}^2
		\norm{\bff{u}_h^{n,j}-\bff{u}_h^{n,j-1}}{\bb{L}^2}^2.
	\end{align*}
	Rearranging the terms and using \eqref{equ:uhnj bdd}, we infer
	\begin{align}\label{equ:uhnh1 minus uhnj}
		\norm{\bff{u}_h^{n,j+1}-\bff{u}_h^{n,j}}{\bb{L}^2}^2
		&\leq
		Ckh^{-\left(2+\frac{d}{2}\right)} \norm{\bff{u}_h^{n,j}}{\bb{L}^2} \norm{\bff{u}_h^{n,j}-\bff{u}_h^{n,j-1}}{\bb{L}^2}^2
		\nonumber\\
		&\quad
		+
		Ckh^{-\frac{d}{2}} \norm{\bff{u}_h^{n,j}}{\bb{L}^2} \norm{\bff{u}_h^{n,j}-\bff{u}_h^{n,j-1}}{\bb{L}^2}^2
		\nonumber\\
		&\quad
		+
		Ckh^{-\frac{3d}{2}} \norm{\bff{u}_h^{n,j}}{\bb{L}^2} \left(\norm{\bff{u}_h^{n,j}}{\bb{L}^2}^2 + \norm{\bff{u}_h^{n,j-1}}{\bb{L}^2}^2 \right)  \norm{\bff{u}_h^{n,j}-\bff{u}_h^{n,j-1}}{\bb{L}^2}^2
		\nonumber\\
		&\quad
		+
		Ckh^{-d} \norm{\bff{u}_h^{n,j}+\bff{u}_h^{n,j-1}}{\bb{L}^2}^2
		\norm{\bff{u}_h^{n,j}-\bff{u}_h^{n,j-1}}{\bb{L}^2}^2
		\nonumber\\
		&\leq
		Ckh^{-\beta} \left(1+4k(\beta_1 \nu_\infty)^2 \alpha^{-1}\right)^{\frac32} \norm{\bff{u}_h^{n-1}}{\bb{L}^2}^3
		\norm{\bff{u}_h^{n,j}-\bff{u}_h^{n,j-1}}{\bb{L}^2}^2,
	\end{align}
	where $\beta= \max\left(2+\frac{d}{2}, \frac{3d}{2}\right)$, and in the last step we used \eqref{equ:uhnj bdd}. 
	
	We now proceed to prove the proposition inductively. Given that Algorithm~\ref{alg:iteration alg} has produced an output $\{\bff{u}_h^{\ell}\}_{0\leq \ell \leq n-1}$, we have $\bff{u}_h^\ell:= \bff{u}_h^{\ell, j_\ell}$ for some $j_\ell$, where $\ell=1,2,\ldots, n-1$. Thus, by repeatedly applying \eqref{equ:uhnj bdd}, for $\ell=1,2,\ldots, n-1$ we have
	\begin{align*}
		\norm{\bff{u}_h^{\ell}}{\bb{L}^2}^2
		=
		\norm{\bff{u}_h^{\ell,j_{\ell}}}{\bb{L}^2}^2
		\leq 
		\left(1+4k (\beta_1 \nu_\infty)^2 \alpha^{-1}\right)^{n-1} \norm{\bff{u}_h^{0}}{\bb{L}^2}^2
		\leq
		\exp\left(4T(\beta_1\nu_\infty)^2 \alpha^{-1}\right) \norm{\bff{u}_h^0}{\bb{L}^2}^2.
	\end{align*}
	This inequality and \eqref{equ:uhnh1 minus uhnj} imply that
	\begin{align}\label{equ:contraction}
		\norm{\bff{u}_h^{n,j+1}-\bff{u}_h^{n,j}}{\bb{L}^2}^2
		&\leq
		Ckh^{-\beta} \norm{\bff{u}_h^{n,j}-\bff{u}_h^{n,j-1}}{\bb{L}^2}^2,
	\end{align}
	where $C$ is a constant which depends only on the coefficients of the equation, the initial data, and $T$. By the Banach fixed-point theorem, if $k=o(h^\beta)$ as $h,k\to 0^+$, then the contraction property~\eqref{equ:contraction} implies the existence of a unique $\bff{u}_h^n\in \bb{V}_h$ solving Algorithm~\ref{alg:iteration alg}. This completes the proof of the proposition.
\end{proof}

\begin{remark}
	The condition $k=o(h^\beta)$ in Proposition~\ref{pro:alg wellposed} is restrictive and appears to arise primarily due to the techniques employed in our analysis. In practical computations, however, this condition seems overly conservative. Although we are currently unable to establish energy dissipativity of the fully discrete scheme when the nonlinear system is solved approximately at each time step using Algorithm~\ref{alg:iteration alg}, as opposed to being solved exactly, our numerical experiments seem to suggest that the scheme retains a form of energy dissipation (in the absence of current) even under inexact solver at each time step. This issue merits further investigation.
\end{remark}

\end{document}